\documentclass[reqno]{amsart}
\usepackage{amsthm,amsmath,amssymb}
\usepackage{stmaryrd,mathrsfs}
\usepackage{esint}
\usepackage{tikz}
\usepackage{hyperref}
\usepackage{enumerate}

\allowdisplaybreaks

\newcommand{\ud}[0]{\,\mathrm{d}}

\newcommand{\abs}[1]{|#1|}
\newcommand{\Babs}[1]{\Big|#1\Big|}

\newcommand{\Norm}[2]{\|#1\|_{#2}}

\newcommand{\BNorm}[2]{\Big\|#1\Big\|_{#2}}
\newcommand{\pair}[2]{\langle #1,#2 \rangle}
\newcommand{\Bpair}[2]{\Big\langle #1,#2 \Big\rangle}
\newcommand{\ave}[1]{\langle #1\rangle}

\newcommand{\lspan}[0]{\operatorname{span}}

\newcommand{\bddlin}[0]{\mathscr{L}}

\newcommand{\BMO}[0]{\operatorname{BMO}}
\newcommand{\supp}[0]{\operatorname{supp}}
\newcommand{\loc}[0]{\operatorname{loc}}

\newcommand{\R}{\mathbb{R}}
\newcommand{\C}{\mathbb{C}}
\newcommand{\N}{\mathbb{N}}
\newcommand{\Z}{\mathbb{Z}}

\renewcommand{\S}{\mathbb{S}}

\newcommand{\prob}[0]{\mathbb{P}}

\newcommand{\D}[0]{\mathbb{D}}
\newcommand{\E}[0]{\mathbb{E}}
\renewcommand{\P}[0]{\mathbb{P}}
\newcommand{\eps}[0]{\varepsilon}

\newcommand{\avL}[0]{\textit{\L}}

\swapnumbers \numberwithin{equation}{section}

\theoremstyle{plain}
\newtheorem{theorem}[equation]{Theorem}
\newtheorem{proposition}[equation]{Proposition}
\newtheorem{corollary}[equation]{Corollary}
\newtheorem{lemma}[equation]{Lemma}

\theoremstyle{definition}
\newtheorem{definition}[equation]{Definition}

\theoremstyle{remark}
\newtheorem{remark}[equation]{Remark}
\newtheorem{example}[equation]{Example}

\makeatletter
\@namedef{subjclassname@2010}{%
  \textup{2010} Mathematics Subject Classification}
\makeatother

\begin{document}

\title[Schatten properties of commutators]{Schatten properties of commutators\\ on metric spaces}

% Schatten properties of commutators on metric spaces
% Commutators and Sobolev norms on metric spaces

\author{Tuomas Hyt\"onen}
\address{Aalto University,
Department of Mathematics and Systems Analysis,
P.O. Box 11100, FI-00076 Aalto, Finland}
\email{tuomas.hytonen@aalto.fi}

\thanks{The author was supported by the Research Council of Finland through projects 346314, 364208, and 371637.
Parts of the paper were written while the author participated the trimester ``Boolean Analysis in Computer Science'' at the Hausdorff Research Institute for Mathematics in Bonn in 9--10/2024, and he would like to thank the Institute for excellent working conditions.}
%\date{\today}


% 42B20 Singular and oscillatory integrals (Calder\'on-Zygmund, etc.) 
% 42B35 Function spaces arising in harmonic analysis 
% 46E36 Sobolev (and similar kinds of) spaces of functions on metric spaces; analysis on metric spaces
% 47B10 Linear operators belonging to operator ideals (nuclear, p-summing, in the Schatten-von Neumann classes, etc.)
% 47B47 Commutators, derivations, elementary operators, etc.

% 42B25 Maximal functions, Littlewood-Paley theory
% 46B09 Probabilistic methods in Banach space theory
% 46E40 Spaces of vector- and operator-valued functions
% 47A60 Functional calculus
% 47F05 Partial differential operators
% 60G46 Martingales and classical analysis

%\vspace{-0.8cm}

\begin{abstract}
We characterise the Schatten class $S^p$ properties of commutators $[b,T]$ of singular integrals and pointwise multipliers in a general framework of \mbox{(quasi-)metric} measure spaces. This covers, unifies, and extends a range of previous results in different special cases. As in the classical results on $\mathbb R^d$, the characterisation has three parts: (1) For $p>d$, we have $[b,T]\in S^p$ if and only if $b$ is in a suitable Besov (or fractional Sobolev) space. (2) For $p\leq d$, we have $[b,T]\in S^p$ if and only if $b$ is constant. (3) For $p=d$, we have $[b,T]\in S^{d,\infty}$ (a weak-type Schatten class) if and only if $b$ is in a first-order Sobolev space.

Result (1) extends to all spaces of homogeneous type as long as there are appropriate singular integrals, but for the more delicate properties (2) and (3), we assume a complete doubling metric space supporting a suitable Poincar\'e inequality, which is still very general. These latter results depend on new characterisations of constant functions and Sobolev spaces over such spaces obtained in a companion paper of the author with R.~Korte.

Even when specialised to various concrete domains considered earlier, the present results extend ones available in the literature by covering a larger class of operators with minimal kernel assumptions, removing {\em a priori} assumptions on the pointwise multiplier $b$, and allowing Schatten classes on the weighted spaces $L^2(w)$ with an arbitrary Muckenhoupt weight $w\in A_2$. Even on $\mathbb R^d$, such weighted results were previously known for a few special operators $T$ only, and on all other domains, they are completely new.
\end{abstract}

\maketitle

\vspace{-0.8cm}

\setcounter{tocdepth}{1}

\tableofcontents

\part{Background and overview}

\section{Introduction and main results}

The theme of this paper is characterising, via oscillatory norms of a function~$b$, quantitative compactness 
of its commutators $[b,T]:f\mapsto b(Tf)-T(bf)$ with a singular integral $T$.
In a critical endpoint case, this involves a first-order Sobolev norm of~$b$ and its equivalent descriptions of independent interest. Over the past few years, classical commutator theory on $\R^d$ has been extended through case studies in several new settings, including the Heisenberg groups and more. The aim of this paper is to present a general framework that covers, unifies, and simplifies many of these previous results with far-reaching extensions, including weighted versions.

A large part of this theory can be and will be developed in general spaces of homogeneous type of \cite{CW:71}, but the more delicate aspects related to endpoint estimates and Sobolev norms will require a somewhat controlled geometry, which is quantified by postulating a suitable Poincar\'e inequality. This is a well-established assumption in the part of analysis on metric spaces dealing with questions of first-order smoothness and known to cover a broad range of examples \cite{Heinonen:book,HKST}. However, to the best of our knowledge, it is only here that this theory is linked with the study of commutators for the first time, thus bridging a gap between two communities.

Before stating our new results, let us briefly recall the state-of-the-art. With somewhat different phenomena in the special dimension $d=1$ investigated by Peller \cite{Peller:80}, the big picture of the situation in the Euclidean spaces $\R^d$ with $d\geq 2$ is as follows:
\begin{enumerate}[\rm(1)]
  \item\label{it:Rd<p} For $p\in(d,\infty)$, the commutator $[b,T]$ is of Schatten class $S^p$ if and only if $b$ belongs to a suitable homogeneous Besov (or fractional Sobolev) space.
  \item\label{it:Rd>p} For $p\in(0,d]$, the commutator $[b,T]$ is of class $S^p$ if and only if $b$ is constant.
  \item\label{it:Rd=p} At the critical index $p=d$, the commutator $[b,T]$ is of weak Schatten class $S^{d,\infty}$ if and only if $b$ is in the homogeneous first-order Sobolev space $\dot W^{1,d}$.
\end{enumerate} 
The dichotomy of \eqref{it:Rd<p} and \eqref{it:Rd>p} was already discovered by Janson and Wolff \cite{JW:82}. Rochberg and Semmes \cite{RS:NWO} provided a new approach, allowing them to describe the more general Schatten--Lorentz $S^{p,q}$ properties of $[b,T]$ in terms of the membership of $b$ in suitable oscillatory spaces $\operatorname{Osc}^{p,q}$. This included a preliminary result in the direction of \eqref{it:Rd=p}; however, the somewhat abstract characterising condition $\operatorname{Osc}^{d,\infty}$ at the critical index was not identified with any classical function space.

This shortcoming was rectified by Connes, Sullivan, Teleman, and Semmes \cite{CST}, who sketched a proof of the surprising identification $\operatorname{Osc}^{d,\infty}=\dot W^{1,d}$. (The direction ``$\supseteq$'' had already been obtained in \cite{RS:end}.) More recently, a completely different proof was given by Lord, McDonald, Sukochev, and Zanin  \cite{LMSZ} (under the technical {\em a priori} assumption of $L^\infty$ functions), and an elaboration of the approach of \cite{CST} was presented by Frank \cite{Frank}, lifting the $L^\infty$ assumption and providing new asymptotic information. A particular interest in the results at this critical index comes from their connection to {\em quantum differentiability} in the sense of Connes \cite[Chapter IV]{Connes:book}; this aspect has been further elaborated in \cite{FSZ:23}.

Suppressing the precise assumptions for the moment, the operators $T$ featuring in these results may be vaguely described as {\em classical singular integrals} associated (in a sense) with the usual Laplacian $\Delta$ on the Euclidean space: either directly as in the case of the Riesz transforms $\partial_j\Delta^{-\frac12}$, or indirectly by means of obeying size and regularity estimates modelled after these basic examples. Over the past few years, these Schatten properties of commutators have seen multiple extensions to new settings, where
\begin{enumerate}[\rm(a)]
  \item\label{it:weights} $\R^d$ is equipped with one \cite{GLW:23} or two \cite{LLW:2wH,LLWW:2wR} weights, or
  \item\label{it:BN} $\R^d$ with the usual Laplacian is replaced by the half space $\R^d_+$ with the Neumann \cite{FLLVW:Neu} or the Bessel \cite{FLLX,FLSZ} Laplacian, or
  \item\label{it:HL} $\R^d$ is replaced by 
the Heisenberg group \cite{FLL:23,FLMSZ},
a more general stratified Lie group \cite{LXY},
or a model domain without any group structure \cite{CLOW}, each coming with their natural singular integrals, or
  \item\label{it:WZ} on $\R^d$, fractional-order \cite{FSZ:24}, multiparameter \cite{LLW:multi}, operator-valued \cite{Wei:thesis,WZ:24,Zhang:thesis}, or less regular \cite{WZ:24b} singular integrals are considered.
\end{enumerate}

The results in these diverse examples have brought to light some general patterns, making the time ripe for a unification.
Our main results are summarised in the following theorem.
We will comment on the various assumptions and conclusions after the statement.
Close to completing this project, the author learned that related results in spaces of homogeneous type spaces have been independently obtained in \cite{FWZZ,WXZZ}.

\begin{theorem}\label{thm:main}
We make the following assumptions on the space of homogeneous type $(X,\rho,\mu)$ and the bounded linear operator $T\in\bddlin(L^2(\mu))$ with the off-support integral representation
\begin{equation*}
  Tf(x)=\int_X K(x,y)f(y)\ud\mu(y),\qquad x\in X\setminus\supp f,
\end{equation*}
in terms of a kernel $K\in C(X\times X\setminus\{(x,x):x\in X\})$:
\begin{enumerate}[\rm(i)]
  \item\label{it:cover} For some $\Delta>0$, every ball $B(x,R)$ contains at most $C(R/r)^\Delta$ mutually $r$-separated points $x_i$, uniformly in $x\in X$ and $0<r\leq R<\infty$.

  \item\label{it:RDD} For some $0\leq d\leq D<\infty$, the measures $V(x,r):=\mu(B(x,r))$ satisfy
\begin{equation}\label{eq:RDD}
  \Big(\frac{R}{r}\Big)^d\lesssim\frac{V(x,R)}{V(x,r)}\lesssim\Big(\frac{R}{r}\Big)^D
\end{equation}
uniformly in $x\in X$ and $0<r\leq R<\infty$.

  \item\label{it:Poincare} If $d\in(1,\infty)$, then we further assume that $(X,\rho)$ is a complete metric space, and for some $\lambda\in[1,\infty)$, we assume that the $(1,d)$-Poincar\'e inequality
\begin{equation}\label{eq:Poincare}
   \fint_{B(x,r)}\abs{f-\ave{f}_{B(x,r)}}\ud\mu
   \leq c_P\cdot r\cdot\Big(\fint_{B(x,\lambda r)}(\operatorname{lip}f)^d\ud\mu\Big)^{\frac1d},
\end{equation}
holds uniformly in $x\in X$, $r>0$, and all Lipschitz functions $f$ with
\begin{equation}\label{eq:lipf}
  \operatorname{lip}f(x):=\liminf_{r\to 0}\sup_{\rho(x,y)\leq r}\frac{\abs{f(x)-f(y)}}{r}.
\end{equation}

  \item\label{it:CZK} For some $\eta\in(0,1]$, the kernel $K$ satisfies the Calder\'on--Zygmund estimates
\begin{equation}\label{eq:CZ0}
   \abs{K(x,y)}\lesssim\frac{1}{V(x,y)}:=\frac{1}{\mu(B(x,\rho(x,y)))},
\end{equation}
\begin{equation}\label{eq:CZ1}
  \qquad\abs{K(x,y)-K(x',y)}+\abs{K(y,x)-K(y,x')}
  \lesssim\Big(\frac{\rho(x,x')}{\rho(x,y)}\Big)^\eta\frac{1}{V(x,y)}
\end{equation}
uniformly in $x,x',y\in X$ such that $\rho(x,x')\ll\rho(x,y)$.

  \item\label{it:nondeg} The kernel $K$ is non-degenerate in the following sense: for every $x\in X$ and $r\in(0,\infty)$, there is a point $y\in X$ with $\rho(x,y)\approx r$ and
\begin{equation}\label{eq:nondeg0}
  \abs{K(x,y)}+\abs{K(y,x)}\gtrsim\frac{1}{V(x,y)},
\end{equation}
uniformly in $x$ and $r$.
\end{enumerate}
Then the following conclusions hold for all $b\in L^1_{\loc}(\mu)$:
\begin{enumerate}[\rm(1)]
  \item\label{it:p>d} In order that $[b,T]\in S^p(L^2(\mu))$, the finiteness of the homogeneous Besov norm
\begin{equation}\label{eq:Besov}
  \Norm{b}{\dot B^p(\mu)}:=\Big(\iint_{X\times X}\frac{\abs{b(x)-b(y)}^p}{V(x,y)^2}\ud\mu(x)\ud\mu(y)\Big)^{\frac1p}
\end{equation}
is necessary for all $p\in(1,\infty)$, and sufficient for all
\begin{equation}\label{eq:p>d+pert}
  p>p(\eta):=\max\Big\{1,\ \Big(\frac{\eta}{\Delta}+\frac12\Big)^{-1}\Big\},
\end{equation}
and we have
\begin{align}
  \Norm{b}{\dot B^p(\mu)} \lesssim\ &\Norm{[b,T]}{S^p(L^2(\mu))}, \qquad & p\in(1,\infty), \label{eq:Sp>Bp} \\
    &\Norm{[b,T]}{S^p(L^2(\mu))} \lesssim \Norm{b}{\dot B^p(\mu)}, \qquad & p\in(p(\eta),\infty). \label{eq:Sp<Bp}
\end{align}
These conclusions do not use Assumptions \eqref{it:RDD} and \eqref{it:Poincare}.

  \item\label{it:p<d} If $d>1$ and $p\in(0,d]$,  then $[b,T]\in S^p(L^2(\mu))$ if and only if $b$ is constant.

  \item\label{it:p=d} If $d=D>1$ and $\eta>(1-\frac{d}{2})_+$, then $[b,T]\in S^{d,\infty}(L^2(\mu))$ if and only if $b$ has a Haj\l{}asz upper gradient $h\in L^d(\mu)$, meaning that
\begin{equation}\label{eq:HajlaszUpper}
   \abs{b(x)-b(y)}\leq(h(x)+h(y))\rho(x,y),\qquad\text{for $\mu$-a.a. }x,y\in X,
\end{equation}
and in this case
\begin{equation*}
  \Norm{[b,T]}{S^{d,\infty}(L^2(\mu))}\approx\Norm{b}{\dot M^{1,d}(\mu)}
  :=\inf\Big\{\Norm{h}{L^d(\mu)}: h\text{ satisfies \eqref{eq:HajlaszUpper}}\Big\}.
\end{equation*}
  \item\label{it:A2} If $w\in A_2$ is a Muckenhoupt weight, i.e., if
\begin{equation*}
    [w]_{A_2}:=\sup_{B\text{ ball}}\fint_B w\ud\mu\fint_B w^{-1}\ud\mu
\end{equation*}
is finite, then all conclusion \eqref{it:p>d} through \eqref{it:p=d} hold more generally with $S^p(L^2(\mu))$ or $S^{d,\infty}(L^2(\mu))$ replaced by the Schatten classes $S^p(L^2(w))$ or $S^{d,\infty}(L^2(w))$ of compact operators on the weighted space $L^2(w)=L^2(w\ud\mu)$, but with the same unweighted spaces $\dot B^p(\mu)$ and $\dot M^{1,p}(\mu)$ in the characterising conditions.
\end{enumerate}
\end{theorem}

Weaker kernel regularity assumptions suffice in place of \eqref{eq:CZ1}; see Remark \ref{rem:finite}.

\begin{remark}\label{rem:main}
Every space of homogeneous type satisfies assumptions \eqref{it:cover} and \eqref{it:RDD} with $d=0$ and some finite $D,\Delta$.
Note that these parameters always satisfy (see Proposition \ref{prop:dims})
\begin{equation}\label{eq:dDeltaD}
  d\leq\Delta\leq D.
\end{equation}

It is easy to see that $p(\eta)<2$, and hence we always get the characterisation
\begin{equation*}
  \Norm{[b,T]}{S^p(L^2(\mu))}\approx\Norm{b}{\dot B^p(\mu)}
\end{equation*}
at least for $p\in[2,\infty)$ in any space of homogeneous type. Extending this range to all $p\in(1,\infty)$, and asking for the further conclusion \eqref{it:p<d} in Theorem \ref{thm:main} is only available under more restricted assumptions on the parameters, even more so for conclusion \eqref{it:p=d}. This agrees with the earlier results for particular cases of $(X,\rho,\mu)$.

Since the introduction of spaces of homogeneous type over 50 years ago \cite{CW:71}, extending Euclidean harmonic analysis to this setting (for those results for which this is possible) is often a somewhat routine exercise by now. While parts of Theorem \ref{thm:main} may be considered merely examples of this routine, a key aspect of our contribution is separating such generic elements of the theory from ones that require additional input, and identifying the correct extra hypotheses, when these are needed. The justification of our assumptions comes both from their established role in the abstract theory \cite{Heinonen:book,HKST} as well as from their good coverage of concrete examples, as discussed further below (see Examples \ref{ex:Carnot} through \ref{ex:A2}).

If $d\geq p(\eta)$ (in particular, if $d\geq 2>p(\eta)$), a combination of conclusions \eqref{it:p>d} and \eqref{it:p<d} of Theorem \ref{thm:main} characterises the property $[b,T]\in S^p$ for all $p\in(0,\infty)$, namely
\begin{equation}\label{eq:p>0}
  [b,T]\in S^p\quad\Leftrightarrow\quad\begin{cases} b\in\dot B^p(\mu), & p>d, \\
  b=\text{constant}, & p\leq d.\end{cases}
\end{equation}
If $d>p(\eta)$, conclusions \eqref{it:p>d} and \eqref{it:p<d} give two competing characterisations for $[b,T]\in S^p$ in the overlapping range $p\in(p(\eta),d]$ but there is no contradiction: this simply means that $\dot B^p(\mu)$ only consists of constant functions in this range.
\end{remark}

\begin{remark}
In general, the lower dimension $d$ in \eqref{eq:RDD} and the power in the Poincar\'e inequality \eqref{eq:Poincare} need not be related. But, in Theorem \ref{thm:main}, the lower dimension~$d$ only plays a role if the space also supports the Poincar\'e inequality with this same exponent.
Assumption \eqref{it:Poincare} will never be used directly; it is only needed to apply characterisations of constant functions and the Haj\l{}asz--Sobolev space $\dot M^{1,d}(\mu)$ from \cite{HK:W1p} (quoted as Proposition \ref{prop:Bp=const} and Theorem \ref{thm:Rupert} below). As explained in \cite[Remark 2.9]{HK:W1p}, the assumption of ``a complete space with a $(1,p)$-Poincar\'e inequality'' could be replaced by assuming a $(1,p-\eps)$-Poincar\'e inequality, since completeness is only needed to apply a self-improvement of the Poincar\'e inequality from \cite[Theorem 1.0.1]{KZ:08}. 

Various examples of spaces supporting a Poincar\'e inequality can be found in \cite{BBG}, \cite[Section 14.2]{HKST}, and \cite[page 274]{Keith:04}. In particular, so-called {\em Laakso spaces} of \cite[Theorem 2.7]{Laakso} provide examples that satisfy the assumption of Theorem \ref{thm:main} for any given $d=D=\Delta\in(1,\infty)$. (For $d=D=\Delta\in(0,1]$, assumption \eqref{it:Poincare} becomes void, and Cantor sets of dimension $d$ will give examples in this range of $d$.)
\end{remark}

\begin{remark}\label{rem:finite}
Assumptions \eqref{it:CZK} are the usual ones of a Calder\'on--Zygmund kernel on a space of homogeneous type, following \cite[Definition IV.2.6]{Christ:book}. Boundedness and qualitative compactness (as opposed to the quantitative compactness described by the Schatten properties as in Theorem \ref{thm:main}) of commutators on spaces of homogeneous type have been studied under such assumptions already in \cite{BC:96,KL:01-I,KL:01-II}. For parts of Theorem \ref{thm:main}, the H\"older-type regularity \eqref{eq:CZ1} may be further relaxed to more general $\omega$-Calder\'on--Zygmund kernels with a suitable modulus of continuity~$\omega$. See Definition \ref{def:CZomega} for the general definition of such kernels, Proposition \ref{prop:Osc<S} for a (very weak) condition on $\omega$ that suffices for commutator lower bounds, and Corollary \ref{cor:bTSpq} and Remark \ref{rem:bTSpq} for relaxations of \eqref{eq:CZ1} that work for the upper bounds. In particular, for those cases of Theorem \ref{thm:main} where H\"older regularity of any order $\eta>0$ suffices, it can further be relaxed to a Dini-type modulus of continuity.

The non-degeneracy assumption \eqref{it:nondeg} is only needed to deduce properties of $b$ from the properties of $[b,T]$; the converse implications remain valid without this assumption. Note that this assumption implicitly requires that $X$ has some pairs of point $x,y$ essentially at any given distance $\rho(x,y)\approx r\in(0,\infty)$. This requires, in particular, that $\operatorname{diam}(X)=\infty$. Since the statement of Theorem \ref{thm:main} is already quite lengthy as it stands, we only note at this point that modification can be made to cover spaces of finite diameter, and refer the reader to Section \ref{sec:finite} for details.
\end{remark}

\begin{remark}
Classical Besov spaces $B^s_{p,q}(\R^d)$ (see e.g. \cite[Chapter 6]{BL:book}) come with three indices. Arguably the most important ones are those with $p=q$, in which case they are also known as fractional Sobolev spaces (at least for smoothness $s\in(0,1)$), and these have been also widely studied over spaces of homogeneous type under the definition
\begin{equation}\label{eq:Bspp}
  \Norm{b}{\dot B^s_{p,p}(\mu)}^p
  :=\int_X\int_X\Big(\frac{\abs{b(x)-b(y)}}{\rho(x,y)^s}\Big)^p\frac{\ud\mu(x)\ud\mu(y)}{V(x,y)};
\end{equation}
see e.g. \cite{BB:23,GKS:24}. If $(X,\rho,\mu)$ is Ahlfors $d$-regular, it is immediate that $\dot B^p(\mu)$ in \eqref{eq:Besov} is equivalent to $\dot B^{d/p}_{p,p}(\mu)$ in \eqref{eq:Bspp}, which also clarifies the connection of Theorem \ref{thm:main}\eqref{it:p>d} with classical results involving $B^{d/p}_{p,p}(\R^d)$. In general spaces of homogeneous type, the choice of definition \eqref{eq:Besov} is justified by the very result of Theorem \ref{thm:main}\eqref{it:p>d}: it is under this definition that the characterisation works in general.
\end{remark}

\begin{remark}
Conclusion \eqref{it:p=d} is probably the most surprising part of Theorem \ref{thm:main}: given that all previous related results dealt with Sobolev spaces defined through some form of derivatives, with proofs involving also higher than the first order calculus (like Taylor expansions with second-order terms), it was by no means clear at the outset whether such a result could be expected at this level of generality. This depends on a new characterisation of the homogeneous Haj\l{}asz Sobolev space $\dot M^{1,d}(\mu)$ obtained in the companion paper \cite{HK:W1p} of the author with R.~Korte. The resulting Case \eqref{it:p=d} of Theorem \ref{thm:main} is philosophically pleasing in the sense that a statement involving first order smoothness is proved by strictly first order methods only, which is manifested by formulating it in a setting where higher order smoothness is not even meaningfully defined.

We have stated conclusion \eqref{it:p=d} of Theorem \ref{thm:main} in terms of the Haj\l{}asz Sobolev space $M^{1,d}(\mu)$ of \cite{Hajlasz:96}, since the key tool in Theorem \ref{thm:Rupert}, borrowed from \cite{HK:W1p}, is most naturally proved for this space. However, under the assumptions of Theorem \ref{thm:main}, this space is known to be equivalent so several other established notions of Sobolev spaces over a metric measure space, including the Cheeger Sobolev space $\textit{Ch}^{1,d}$ of \cite{Cheeger}, the Korevaar--Schoen Sobolev space $\textit{KS}^{1,d}$ of \cite{KS:93}, the Newtonian Sobolev space $N^{1,d}$ of \cite{Nages:00}, and the Poincar\'e Sobolev space $P^{1,d}$ of \cite{HK:00}, as well as various Sobolev spaces $H^{1,d}$ associated with given systems of vector fields, which is often the preferred definition in concrete situations. See  \cite[Theorem 10.5.3]{HKST} for $M^{1,d}=P^{1,d}=\textit{KS}^{1,d}=N^{1,d}=\textit{Ch}^{1,d}$, and \cite[Theorem 10]{FHK:99} for $H^{1,d}=P^{1,d}$ under the assumptions of Theorem \ref{thm:main} on $(X,\rho,\mu)$, and quite general assumptions on the vector fields defining $H^{1,d}$ in the last equality. The said results are formulated for the inhomogeneous versions of these spaces (involving the size of both the function and its ``gradient''), but they can be readily adapted to the homogeneous versions relevant to Theorem \ref{thm:main}. (Note also that, by \cite[Theorem 1.0.1]{KZ:08}, the assumptions of Theorem \ref{thm:main} (in particular, completeness and $(1,d)$-Poincar\'e, when $d>1$) imply those of  \cite[Theorem 10.5.3]{HKST} (namely, $(1,q)$-Poincar\'e for some $q\in[1,d)$).)
\end{remark}

In the important special case that $d=\Delta$, Theorem \ref{thm:main} takes the following form: (By \eqref{eq:dDeltaD}, this is more general than Ahlfors regular spaces defined by the condition that $d=D$; see Example \ref{ex:Bessel} for a class of examples with $d=\Delta<D$.)

\begin{corollary}\label{cor:Ahlfors}
Let $(X,\rho,\mu)$ and $T\in\bddlin(L^2(\mu))$ be as in Theorem \ref{thm:main}, and assume moreover that $d=\Delta\in(0,\infty)$. Then
\begin{enumerate}[\rm(1)]
  \item\label{it:Ap>d} for $p>\max(d,1)$ and $\eta\geq\min\{\frac{d}{2},(1-\frac{d}{2})_+\}$, we have $[b,T]\in S^p$ iff $b\in\dot B^p(\mu)$;
  \item\label{it:Ap<d} if $d>1$ and $p\in(0,d]$, we have $[b,T]\in S^p$ iff $b$ is constant;
  \item\label{it:Ap=d} if $d=D>1$ and $\eta>(1-\frac{d}{2})_+$, then $[b,T]\in S^{d,\infty}$ iff $b\in\dot M^{1,d}(\mu)$;
  \item\label{it:Aw} if $w\in A_2$, these same conclusions also hold for Schatten classes on $L^2(w)$.
\end{enumerate}
In particular, if $d\geq 2$, then these conclusions are valid for every $\eta\in(0,1]$.
\end{corollary}

\begin{proof}
The last three conclusions are just restatements of the corresponding ones in Theorem \ref{thm:main}. 
As for the first conclusion, recall that Theorem \ref{thm:main}\eqref{it:p>d} says that $\Norm{[b,T]}{S^p}\approx\Norm{b}{\dot B^p(\mu)}$ for $p>p(\eta)$. We consider separately $d>1$ and $d\leq 1$:

If $d=\Delta>1$ and $\eta\geq(1-\frac{d}{2})_+$, then
$  \frac{\eta}{\Delta}+\frac12\geq\frac{1}{d},$
hence $p(\eta)\leq d$. Thus the range of $p>p(\eta)$ covers in particular all $p>d$, as claimed in Corollary \ref{cor:Ahlfors}\eqref{it:Ap>d}.

If $d=\Delta\in(0,1]$ and $\eta\geq\frac{d}{2}$, then $\frac{\eta}{\Delta}+\frac12\geq 1$, hence $p(\eta)=1$, and thus the range of $p>p(\eta)$ covers all $p\in(1,\infty)$, again as claimed in Corollary \ref{cor:Ahlfors}\eqref{it:Ap>d}.
\end{proof}

\begin{example}\label{ex:Carnot}
Any Heisenberg group $\mathbb H^n$, and more generally any Carnot group $\mathbb G$ equipped with the Carnot--Carath\'eodory metric $d_{cc}$ and Haar measure $\mu$, is an Ahlfors regular space (i.e., $d=D=\Delta$) supporting the $(1,1)$-Poincar\'e inequality, and the Newtonian Sobolev space $N^{1,p}(\mathbb G)$ (and thus the Haj\l{}asz Sobolev space $M^{1,p}(\mathbb G)$ by \cite[Theorem 10.5.3]{HKST}) is identified with the {\em horizontal Sobolev space} $W^{1,p}_H(\mathbb G)$; see \cite[Section 14.2]{HKST} for details and additional pointers to the literature. This shows that the set-ups of \cite{FLL:23,FLMSZ,LXY} fall under the scope of Corollary \ref{cor:Ahlfors}. In the Heisenberg groups $\mathbb H^n$ considered in \cite{FLL:23,FLMSZ}, the homogeneous dimension is $d=2n+2\geq 4$. In \cite[page 77]{LXY} it is noted that Carnot groups (necessarily of integer dimension) with $d\leq 3$ are in fact just $\R^d$, and hence the assumption $d\geq 4$ is also made there. For such dimensions, we have $(1-\frac{d}{2})_+=0$, and hence the conclusions of Corollary \ref{cor:Ahlfors} are valid for operators $T$ of arbitrary Calder\'on--Zygmund regularity $\eta\in(0,1]$. Conclusions \eqref{it:Ap>d} and \eqref{it:Ap<d} of the said corollary reproduce \cite[Theorem 1.1]{LXY}, which was stated under similar assumptions, while our conclusion \eqref{it:Ap=d}, still valid under minimal Calder\'on--Zygmund regularity, is a significant extension of \cite[Theorem 1.2]{LXY}, where derivative bounds on $K$ up to order $\gamma>d$ were required. Conclusion \eqref{it:Aw} on the weighted estimates is completely new, even for kernels with the higher-order derivative bounds considered in \cite{LXY}.
\end{example}

\begin{example}\label{ex:Bessel}
The Bessel setting considered in \cite{FLLX} consists of $X=\R_+^n$ with the Euclidean distance and the weighted Lebesgue measure $x_n^{2\lambda}\ud x$. In \cite{FLLX}, the assumption $\lambda\in(0,\infty)$ is made, but one can more generally consider $\lambda\in(-\frac12,\infty)$ as in \cite{CZ:14}; the restriction $\lambda>-\frac12$ is necessary to make this measure locally finite. We show in Proposition \ref{prop:Bessel} that this is a doubling metric measure space supporting a $(1,1)$-Poincar\'e inequality and having dimension parameters
\begin{equation*}
   d=n-2\lambda_-\in(n-1,n],\qquad\Delta=n,\qquad D=n+2\lambda_+\in[n,\infty).
\end{equation*}
For $\lambda>0$ as in \cite{FLLX}, it follows that $d=\Delta=n$, bringing us to the setting of Corollary \ref{cor:Ahlfors}. For $d=n\geq 2$, the said corollary shows that we have \eqref{eq:p>0} for operators $T$ of any Calder\'on--Zygmund regularity $\eta\in(0,1]$. For $d=n=1$, we have the characterisation $\Norm{[b,T]}{S^p}\approx\Norm{b}{\dot B^p(\mu)}$ for all $p\in(1,\infty)$ and all operators $T$ of Calder\'on--Zygmund regularity $\eta\in[\frac12,1]$. 

These two statements are extensions of \cite[Theorems 1.4 and 1.5]{FLLX}, where the special case of the Bessel--Riesz transforms in place of $T$ was considered, and the higher order derivative bounds of their kernels, established in \cite[Lemma 2.5]{FLLX} played a role in the proof. Moreover, the version of the constancy characterisation of Corollary \ref{cor:Ahlfors}\eqref{it:Ap<d} was obtained in \cite[Theorem 1.5]{FLLX} only under the {\em a priori} assumption that $b\in C^2(\R_+^n)$. Whether this assumption could be removed was raised there as an open problem. Our Corollary \ref{cor:Ahlfors} answers this in the affirmative.

By invoking the full Theorem \ref{thm:main} (instead of Corollary \ref{cor:Ahlfors}), we also obtain results for $\lambda\in(-\frac12,0)$, which are completely new. See Remark \ref{rem:Bessel} for details.

Shortly after the original arXiv announcement of the present results in 11/2024, another work in the Bessel setting appeared, \cite{FLSZ}, where the authors obtain a variant of the weak-type mapping property \eqref{it:p=d} of Theorem \ref{thm:main} in the Bessel setting. Since we only address the weak-type characterisation in Ahlfors-regular spaces, the results of \cite{FLSZ}, in a specific non-Ahlfors setting, fall outside the scope of Theorem \ref{thm:main}.
\end{example}

\begin{example}[{\em plus Open Problem}]
The setting of \cite{CLOW} involves a space of homogeneous type $\partial\Omega_k$ of lower dimension $d=4$ and upper dimension $D=2k+2$, where $k\geq 2$; see the penultimate paragraph of \cite[Introduction]{CLOW}. For these values, \cite[Theorem 1.3]{CLOW} gives the characterisation \eqref{eq:p>0}, which agrees with the numerology of our setting. However, the following question is open to our knowledge:
\begin{quote}
   Does $\partial\Omega_k$ of \cite{CLOW} satisfy the Poincar\'e assumption \eqref{it:Poincare}?
\end{quote}
If the answer is positive, this would have the advantage of not only recovering \cite[Theorem 1.3]{CLOW} but also lifting the {\em a priori} assumption $b\in C^2(\partial\Omega_k)$ in part (2) of the said result, since no such condition appears in our Theorem \ref{thm:main}.
\end{example}

\begin{example}\label{ex:A2}
For $X=\R^d$ and $T$ equal to any of the Riesz transforms $R_j=\partial_j\Delta^{-\frac12}$, the weighted conclusion \eqref{it:A2} of Theorem \ref{thm:main} coincides with \cite[Theorems 1.1 and 1.2]{GLW:23}. Concerning the critical-index result of Theorem \ref{thm:main}\eqref{it:p=d}, note that
\begin{equation*}
    \dot M^{1,p}(\R^d)=\dot W^{1,p}(\R^d)
\end{equation*}
by \cite[Theorem 1]{Hajlasz:96}. Even on $\R^d$, this weighted result of Theorem \ref{thm:main}\eqref{it:A2} seems to be new for more general Calder\'on--Zygmund operators. On all other spaces covered by Theorem \ref{thm:main}, no weighted results seemed to be available before.
\end{example}
 
The proof of Theorem \ref{thm:main} will proceed via an intermediate notion of oscillation spaces $\operatorname{Osc}^{p,q}$, which are natural extensions of corresponding spaces on $\R^d$ defined in \cite{RS:NWO}. This part of the argument is fairly general, and works for any space of homogeneous type, but it depends on the auxiliary notion of dyadic cubes in such spaces, which we recall in Section \ref{sec:cubes}:
 
 \begin{proposition}\label{prop:comVsOsc}
 Let $(X,\rho,\mu)$ be a space of homogeneous type and $T\in\bddlin(L^2(\mu))$ an operator with a non-degenerate Calder\'on--Zygmund kernel of regularity $\eta\in(0,1]$. Let $\mathscr D$ be a system of dyadic cubes on $(X,\rho,\mu)$ in the sense of Definition \ref{def:cubes} and, for each $Q\in\mathscr D$, let $B_Q$ be a ball centred at $Q$ and of radius $c\ell(Q)$ for a constant $c$ that only depends on the space.
For $b\in L^1_{\loc}(\mu)$, in order that $[b,T]\in S^{p,q}(L^2(\mu))$, the finiteness of
\begin{equation*}
  \Norm{b}{\operatorname{Osc}^{p,q}}
  :=\BNorm{\Big\{\fint_{B_Q}\abs{b-\ave{b}_{B_Q}}\ud\mu\Big\}_{Q\in\mathscr D}}{\ell^{p,q}}
\end{equation*}
is necessary for all $p\in(1,\infty)$ and $q\in[1,\infty]$, and sufficient for all $p\in(p(\eta),\infty)$ and $q\in[1,\infty]$, where $p(\eta)$ is as in \eqref{eq:p>d+pert}. Moreover, we have
\begin{align}
  \Norm{b}{\operatorname{Osc}^{p,q}}\lesssim\ &\Norm{[b,T]}{S^{p,q}(L^2(\mu))},\quad
 &  p\in(1,\infty),\ q\in[1,\infty],\label{eq:Osc<S} \\
 &\Norm{[b,T]}{S^{p,q}(L^2(\mu))}
 \lesssim \Norm{b}{\operatorname{Osc}^{p,q}},\quad & p\in(p(\eta),\infty),\ q\in[1,\infty].\label{eq:S<Osc}
\end{align}
For $w\in A_2$, the same results are valid for $L^2(w)$ in place of $L^2(\mu)$.
 \end{proposition}

While Proposition \ref{prop:comVsOsc} already achieves the goal of characterising the Schatten--Lorentz properties of the commutator $[b,T]$ in terms of {\em some} function space properties of the multiplier $b$, these conditions in terms of the norms $\Norm{b}{\operatorname{Osc}^{p,q}}$ remain somewhat exotic, and a more useful characterisation, as in Theorem \ref{thm:main}, is only achieved after identifying $\operatorname{Osc}^{p,q}$ with more {\em familiar} function spaces. But it seems worth noting that, up to this point, everything is valid in general spaces of homogeneous type, and the finer assumptions of Theorem \ref{thm:main} are only needed for the mutual identification of the characterising function spaces, not for the analysis of the commutators as such. This point might not have been obvious from the previous literature, where versions of Proposition \ref{prop:comVsOsc} (and not just of Theorem \ref{thm:main}, where additional assumptions do play a role) are formulated only for the concrete special cases of $(X,\rho,\mu)$ under consideration; this is explicit in \cite[Theorem 1.2]{FLLX} in the Bessel setting, and implicit in some other works.

Note that all the weighted aspect of Theorem \ref{thm:main} are captured by Proposition \ref{prop:comVsOsc}; the characterising $\operatorname{Osc}^{p,q}$ spaces are the same independently of the weight $w\in A_2$.
The following result that takes care of Theorem \ref{thm:main}\eqref{it:p>d} (and paves the way for the other parts of the theorem) is still relatively general:

\begin{proposition}\label{prop:Osc=classical}
Let $(X,\rho,\mu)$ be a space of homogeneous type and $b\in L^1_{\loc}(\mu)$.
\begin{enumerate}[\rm(1)]
\item For all $p\in(1,\infty)$, we have
\begin{equation}\label{eq:Osc=Besov}
  \Norm{b}{\operatorname{Osc}^{p,p}}\approx\Norm{b}{\dot B^p(\mu)}.
\end{equation}
\item If $X$ is Ahlfors regular of homogeneous dimension $d$, then
\begin{equation}\label{eq:Osc=weakLp}
  \Norm{b}{\operatorname{Osc}^{d,\infty}}
  \approx\Norm{m_b}{L^{d,\infty}(\nu_d)},
\end{equation}
where
\begin{equation}\label{eq:mb-nu}
  m_b(x,t):=\fint_{B(x,t)}\abs{b-\ave{b}_{B(x,t)}}\ud\mu,\qquad
  \ud\nu_d(x,t):=\frac{\ud\mu(x)\ud t}{t^{d+1}}
\end{equation}
are defined on the product space $X\times(0,\infty)$.
\end{enumerate}
\end{proposition}

For the proof of Theorem \ref{thm:main}\eqref{it:p<d}, we have
\begin{equation*}
\begin{split}
  \Norm{b}{\dot B^d(\mu)} &\approx\Norm{b}{\operatorname{Osc}^{d,d}(\mu)}\qquad\text{by Proposition \ref{prop:Osc=classical}} \\
  &\lesssim\Norm{[b,T]}{S^d(L^2(\mu))}\qquad\text{by Proposition \ref{prop:comVsOsc}} \\
  &\leq\Norm{[b,T]}{S^p(L^2(\mu))}\qquad\text{for }p\leq d, \\
\end{split}
\end{equation*}
and the argument is completed by the following result of the companion paper \cite{HK:W1p}:

\begin{proposition}[\cite{HK:W1p}, Theorem 1.5(2) with $p=d$]\label{prop:Bp=const}
Let $d\in(1,\infty)$ and let $(X,\rho,\mu)$ be a doubling metric measure space having lower dimension $d$ and supporting the $(1,d)$-Poincar\'e inequality.
If $f\in L^1_{\loc}(\mu)$ satisfies
\begin{equation*}
  \Big[(x,y)\mapsto  \frac{\abs{f(x)-f(y)}^d}{V(x,y)^2}\Big]\in L^1_{\loc}(\mu\times\mu),
\end{equation*}
then $f$ is equal to a constant almost everywhere. In particular,
\begin{equation}\label{eq:Bp=const}
  \dot B^d(\mu)=\{\text{constants}\}.
\end{equation}
\end{proposition}

From the same work, we also have an identification of the space on the right of \eqref{eq:Osc=weakLp}, the main result of \cite{HK:W1p}:

\begin{theorem}[\cite{HK:W1p}, Theorem 1.1]\label{thm:Rupert}
Let $(X,\rho,\mu)$ be a complete doubling metric measure space supporting a $(1,p)$-Poincar\'e inequality for some $p\in(1,\infty)$. Let $b\in L^1_{\loc}(\mu)$. Then $b\in\dot M^{1,p}(\mu)$ if and only if $m_b\in L^{p,\infty}(\nu_p)$. Moreover, we have
\begin{align}
  \Norm{b}{\dot M^{1,p}(\mu)} \label{eq:CST}
  &\approx\Norm{m_b}{L^{p,\infty}(\nu_p)}:=\sup_{\kappa>0}\kappa\cdot\nu_p(\{m_b>\kappa\})^{1/p} \\
  &\approx\liminf_{\kappa\to0}\kappa\cdot\nu_p(\{m_b>\kappa\})^{1/p}, \label{eq:Rupert}
\end{align}
where $m_b$ and $\nu_p$ are as in \eqref{eq:mb-nu}.
\end{theorem}

Thus Theorem \ref{thm:main}\eqref{it:p=d} is obtained via the chain
\begin{equation*}
\begin{split}
  \Norm{[b,T]}{S^{d,\infty}}
  &\approx\Norm{b}{\operatorname{Osc}^{d,\infty}}\qquad\text{by Proposition \ref{prop:comVsOsc}} \\
  &\approx\Norm{m_b}{L^{d,\infty}(\nu_d)}\qquad\text{by Proposition \ref{prop:Osc=classical}} \\
  &\approx\Norm{b}{\dot M^{1,d}(\mu)}\qquad\text{by Theorem \ref{thm:Rupert}},
\end{split}
\end{equation*}
where the first step is valid in general spaces of homogeneous type (and all $p\in(1,\infty)$ in place of $d$), the second one only needs Ahlfors regularity of homogeneous dimension $d$, and the final one the $(1,d)$-Poincar\'e inequality (but no relation between $d$ and the dimensions of the space).
 
Theorem \ref{thm:Rupert} is only relevant for Theorem \ref{thm:main} in the special case of Ahlfors $d$-regular spaces and exponents $p=d$ matching the homogeneous dimension, but Theorem \ref{thm:Rupert} itself is more general. In the said special case, with the additional assumption that $X$ is {\em compact}, a variant of Theorem \ref{thm:Rupert} \eqref{eq:CST} (with different but closely analogous right-hand side) is also contained in \cite[Theorem 1.4]{BS:18}. It seems possible that, with suitable truncation and approximation arguments to reduce the considerations to the compact case, the proof of Theorem \ref{thm:main}\eqref{it:p=d} could also be completed using \cite[Theorem 1.4]{BS:18} in place Theorem \ref{thm:Rupert} from \cite{HK:W1p}, but we prefer the latter, which can be applied directly, without a need of additional reductions to the compact case.

For $X=\R^d$ and $p=d\geq 2$, \eqref{eq:CST} is \cite[Theorem on page 679]{CST}. The extension to all $p\in(1,\infty)$ and the limit relation \eqref{eq:Rupert} are contained in \cite[Theorem 1.1]{Frank}. Beyond $\R^d$, a version of Theorem \ref{thm:Rupert} in Carnot groups was previously obtained in \cite[Theorem A.3]{LXY}. These results have played a similar role in the proofs of the special cases of Theorem \ref{thm:main} in the respective settings, as Theorem \ref{thm:Rupert}  plays in the proof of Theorem \ref{thm:main}.

\section{An overview of methods and organisation of the paper}\label{sec:overview}

Different parts of our main Theorem \ref{thm:main} depend on different sets of assumptions, and thereby also on somewhat different methods. While these components are combined together in the end, substantial parts of the proof are largely independent of each other, allowing alternatives for the order of presentation. 

With the occasional assumption of more particular assumptions, we will be largely working in the general setting of spaces of homogeneous type, and the extensions of many of the Euclidean arguments are less surprising on the local level (say, within a section). Still, we feel that the big picture obtained by the combination of these local elements was not obvious, given the explicit treatment in the recent literature of various special cases that are now subsumed by Theorem \ref{thm:main} or even Proposition \ref{prop:comVsOsc}.

Throughout the paper, the notion of dyadic cubes in spaces of homogeneous type (available since \cite{Christ:90}, with more recent refinements in \cite{AH:13,HK:12}) will be a principal tool. Roughly speaking, Part \ref{part:NWO} covers the part of the theory within reach by an adaptation of the Euclidean methods of \cite{JW:82,RS:NWO}, simply using the dyadic cubes of \cite{Christ:90} in place of the classical ones. With the established properties of the dyadic cubes, the proof of Proposition \ref{prop:Osc=classical} is straightforward. The dyadic cubes also lead to a natural extension of the method of ``nearly weakly orthogonal'' sequences of \cite{RS:NWO}, which allows us to give a proof of the commutator lower bound \eqref{eq:Osc<S}. Together with Proposition \ref{prop:Osc=classical}, we thus obtain
\begin{equation}\label{eq:comLower}
  \Norm{b}{\dot B^p(\mu)}\approx\Norm{b}{\operatorname{Osc}^{p,p}}\lesssim\Norm{[b,T]}{S^p(L^2(\mu))},\qquad p\in(1,\infty).
\end{equation}

In the other direction, a clever interpolation technique of \cite{JW:82} allows us to prove the estimate 
\begin{equation}\label{eq:JWbound}
  \Norm{[b,T]}{S^p(L^2(\mu))}\lesssim\Norm{b}{\dot B^p(\mu)},\qquad p\in(2,\infty).
\end{equation}
Together with \eqref{eq:Bp=const} proved in \cite{HK:W1p}, the estimates \eqref{eq:comLower} and \eqref{eq:JWbound} already give parts \eqref{it:p>d} and \eqref{it:p<d} of Theorem \ref{thm:main} whenever $d\geq 2$. This is a worthwhile observation, since this dimensional assumptions holds in most of the concrete spaces in which the commutator theory has been studied. While \eqref{eq:JWbound} will also follow from the more general methods of Part \ref{part:dyadic} (to be discussed shortly), the simple alternative argument of \cite{JW:82} has independent interest and also provides additional information; see Section \ref{sec:JW} for details.

\begin{remark}
The interpolation behind \eqref{eq:JWbound} is worth a separate comment. The argument of \cite{JW:82} proceeds by first proving the weak-type estimate
\begin{equation*}
   \Norm{[b,T]}{S^{p,\infty}(L^2(\R^d))}\lesssim\Norm{b}{\dot B^{d/p}_{p,p}(\R^d)},
\end{equation*}
and then interpolating. This requires in particular an interpolation result for the Besov spaces on the right, which is classical and well known for Besov spaces on $\R^d$, but less obvious on a general space of homogeneous type. While results {\em of this type} are contained in the literature (see \cite[Theorem 4.4]{GKS:10} and \cite[Theorem 3.1]{Yang:04}), they only deal with interpolation of $B^s_{p,q}$ between variable $(s,q)$ for fixed~$p$, which is not enough for the case at hand. While it is conceivable that the needed interpolation theorem has been proved, and likely that it can be proved, we are not aware of it appearing in the literature. It seems that this detail has been overlooked in some previous adaptations of the argument of \cite{JW:82}.

We avoid this problem by setting up a variant of the argument that only requires standard interpolation of $L^p$ spaces, where the presence of an abstract measure space is no issue. This might be of some interest even for the original result of \cite{JW:82} on $\R^d$, if one wanted to present a self-contained proof, say on an advanced course or in a textbook, without developing a heavy piece of Littlewood--Paley theory of Besov spaces as a sidetrack. Our results and proofs only rely on the definition \eqref{eq:Besov}.
\end{remark}

We finally come to Part \ref{part:dyadic} of the paper. It is here that we prove the commutator upper bound \eqref{eq:S<Osc} for all $p\in(p(\eta),\infty)$ and $q\in[1,\infty]$, extending the result for $p=q\in(2,\infty)$ available by the method of \cite{JW:82} used in Part \ref{part:NWO}, and also obtain the weighted version of this bound, as stated in Theorem \ref{thm:main}\eqref{it:A2}. Even for dimensions $d\geq 2$ (when the methods of Part \ref{part:NWO} were enough for conclusions \eqref{it:p>d} and \eqref{it:p<d} of Theorem \ref{thm:main}), this extension is necessary to cover the end-point estimate with $q=\infty$ as in Theorem \ref{thm:main}\eqref{it:p=d} and the weighted extension in Theorem \ref{thm:main}\eqref{it:A2}. For spaces of lower dimension $d<2$ and/or lacking a Poincar\'e inequality, this allows us to cover the full range of $p$ in Theorem \ref{thm:main}\eqref{it:p>d}, rather than just $p>2$.

The proof of the Euclidean predecessor of \eqref{eq:S<Osc} in \cite{RS:NWO} is based on their method of ``nearly weakly orthogonal'' series produced by local Fourier expansions of the kernel $K$. In some of the recent extensions \cite{FLLX,LXY}, these have been replaced by expansions in terms of so-called Alpert bases \cite{Alpert}, whose characteristic feature is orthogonality to polynomials up to a high degree, and the desired decay of the expansions is achieved via Taylor expansions of a kernel of high-order smoothness. These methods are clearly unavailable in the generality that we consider, and the proof of this remaining bound \eqref{eq:S<Osc} will instead be based on the {\em dyadic representation} of general Calder\'on--Zygmund operators $T$ as averages of series of so-called {\em dyadic shifts} $\S$. By linearity, the estimation of $[b,T]$ is then reduced to the estimation of $[b,\S]$, with control on the shift parameters to ensure the summability of the series.

With significant predecessors, notably \cite{Pet:00,PTV:02}, a dyadic representation theorem in the sense of today was introduced in \cite{Hyt:A2}, originally for proving the so-called $A_2$ conjecture on sharp weighted norm inequalities for Calder\'on--Zygmund operators on $\R^d$. The earlier representation of \cite{Pet:00} had been used for such a result in a special case in \cite{Pet:07}, but its original motivation in \cite{Pet:00} was proving commutator boundedness; thus its more recent role also in related Schatten results is no surprise. The more general version of \cite{Hyt:A2} has also been subsequently applied for this purpose in \cite{DO:16,HLW:17,HW:18}. 

For $S^p$ estimates of commutators, dyadic representation was probably first used in \cite{PS:04} and more recently in \cite{GLW:23,LLW:multi,LLW:2wH,LLWW:2wR}, but all these works only dealt with the Hilbert or Riesz transforms using the special representations of \cite{Pet:00,PTV:02}. It was only very recently that the general dyadic representation of \cite{Hyt:A2} was applied to $S^p$ estimates of commutators on $\R^d$, first for $p\geq 2$ in \cite{Zhang:thesis}, with subsequent extensions to $p<2$ \cite{WZ:24b,WZ:24}. Versions of the dyadic representation theorem on doubling metric spaces were already known \cite{Hyt:Enflo,NRV}, but somewhat implicitly within the proofs of other results, and we supply an explicit statement in Theorem \ref{thm:DRT} in the full generality of spaces of homogeneous type. Due to the many applications of the Euclidean version of this representation, this has independent interest.
Together with commutator estimates for the dyadic shifts $\S$ featuring in this representation obtained in Sections \ref{sec:shift} (unweighted) and \ref{sec:shift-w} (weighted), we obtain the upper bounds for $[b,T]$, which complete the proof of Theorem \ref{thm:main}. Besides the main point of generalising the Euclidean results of \cite{WZ:24b,WZ:24} to spaces of homogeneous type, the weighted aspects of our result are completely new even in $\R^d$, and we can even slightly relax the kernel assumptions of \cite{WZ:24b,WZ:24} also in the unweighted case.

On a technical level, while the results of Part \ref{part:NWO} only require a fixed system of dyadic cubes due to \cite{Christ:90}, the dyadic representation fundamentally depends on the notion of {\em random dyadic systems} constructed in \cite{AH:13,HK:12,HM:12}, which we will also need to revisit in Part \ref{part:dyadic}.

%In summary, the key innovations of this paper are contained in Part \ref{part:Sobolev} and Section \ref{sec:shift-w}.

\section{Dimension concepts}\label{sec:prelim}

Recall that three different dimension concepts featured in Theorem \ref{thm:main}. For completeness we provide some background about them.

\begin{definition}\label{def:dims}
We say that a quasi-metric measure space $(X,\rho,\mu)$ has
\begin{enumerate}
  \item {\em separation dimension} $\Delta$ if every ball $B(x,R)$ can contain at most $C(R/r)^\Delta$ mutually $r$-separated points;
  \item {\em covering dimension} $\Delta'$ if every ball $B(x,R)$ can be covered by at most $C'(R/r)^{\Delta'}$ balls;
  \item  {\em lower dimension} $d$ if
\begin{equation}\label{eq:lowerD}
  \Big(\frac{R}{r}\Big)^d\lesssim\frac{V(x,R)}{V(x,r)};
\end{equation}
 \item {\em upper dimension} $D$ if 
\begin{equation}\label{eq:upperD}
  \frac{V(x,R)}{V(x,r)}\lesssim\Big(\frac{R}{r}\Big)^D
\end{equation}
\end{enumerate}
in each case uniformly in $x\in X$ and $0<r\leq R<\infty$.
\end{definition}

\begin{remark}\label{rem:dims}
In our terminology of Definition \ref{def:dims}, none of the dimensions is a unique number: a space with separation dimension $\Delta$ also has separation dimension $\widetilde\Delta$ for every $\widetilde\Delta\geq\Delta$, and the same applies to the covering dimension and the upper dimension. On the other hand, a space with lower dimension $d$ also has lower dimension $\widetilde d$ for every $\widetilde d\leq d$. Unique values could be obtained by taking the infimum or the supremum of the admissible dimensions; e.g., the infimum of the covering dimensions is called the {\em Assouad dimension} of $X$ \cite[Definition 10.15]{Heinonen:book}. However, we find the terminology of Definition \ref{def:dims} convenient for the present purposes.
\end{remark}

\begin{lemma}\label{lem:sepVsCov}
Let $(X,\rho)$ be a quasi-metric space with quasi-triangle constant $A_0$, and let $E\subset X$.
For $r\in(0,\infty)$, suppose that
\begin{itemize}
  \item $E$ contains at most $M_E(r)$ points that are mutually $r$-separated;
  \item $E$ can be covered by $N_E(r)$ balls of radius $r$.
\end{itemize}
Then $\displaystyle N_E(r)\leq M_E(r)\leq N_E(\frac{r}{2A_0})$.
\end{lemma}

\begin{proof}
Let $\{x_i\}_{i=1}^M$ be a maximal $r$-separated subset of $E$. Let $x\in E$. Since $\{x_i\}_{i=1}^M$ is maximal, the set $\{x_i\}_{i=1}^M\cup\{x\}$ is not $r$-separated, and hence $\rho(x,x_i)<r$ for some $i$. Thus the balls $B(x_i,r)$ cover $E$, and hence $N_E(r)\leq M\leq M_E(r)$.

On the other hand, let $\{B_k:=B(y_k,\frac{r}{2A_0})\}_{k=1}^N$ be some cover of $E$. Thus each $x_i$ is contained in some $B_k$
If two of these points satisfy $x_i,x_j\in B_k$ for the same $k$, then 
\begin{equation*}
  \rho(x_i,x_j)\leq A_0(\rho(x_i,y_k)+\rho(y_k,x_j))<A_0(\frac{r}{2A_0}+\frac{r}{2A_0})=r,
\end{equation*}
which contradicts the $r$-separation. Hence each $x_i$ is contained in a different $B_k$, and thus the number of the balls $B_k$ is at least the number of the points $x_i$. This means that $M\leq N$. Now it is possible to choose a cover by such balls, where $N\leq N_E(\frac{r}{2A_0})$, and since the bound $M\leq N$ holds for every $r$-separated sequence of length $M$, we obtain $M_E(r)\leq N_E(\frac{r}{2A_0})$.
\end{proof}

\begin{proposition}\label{prop:dims}
Let $(X,\rho,\mu)$ be a space of homogeneous type, and let $d,D,\Delta,\Delta'$ be as in Definition \ref{def:dims}. Then
\begin{equation*}
  d\leq\Delta=\Delta'\leq D
\end{equation*}
in the following sense:
\begin{enumerate}[\rm(1)]
  \item $X$ has separation dimension $\Delta$ if and only if it has covering dimension $\Delta$;
  \item if $X$ has upper dimension $D$, then it also has separation dimension $D$;
  \item if $X$ has lower dimension $d$, it cannot have separation dimension below $d$.
\end{enumerate}
\end{proposition}

\begin{proof}
In the notation of Lemma \ref{lem:sepVsCov}, the separation and covering dimensions are defined by the conditions that $M_{B(x,R)}(r)\leq C(R/r)^\Delta$ and $N_{B(x,R)}(r)\leq C'(R/r)^{\Delta'}$ for all $x\in X$ and $0<r<R<\infty$. By the said lemma, we have
\begin{equation*}
  N_{B(x,R)}(r)\leq M_{B(x,R)}(r)\leq C(R/r)^\Delta,
\end{equation*}
and
\begin{equation*}
  M_{B(x,R)}(r)\leq N_{B(x,R)}(\frac{r}{2A_0})\leq C'\Big(\frac{R}{ r/(2A_0)}\Big)^{\Delta'}=
  C'(2A_0)^{\Delta'}(R/r)^{\Delta'},
\end{equation*}
which shows that we always have $\Delta=\Delta'$ and $C'\leq C\leq C'(2A_0)^{\Delta'}$.

Suppose that $B(x,R)$ contains $M$ mutually $r$-separated points $x_i$. Then the balls $B(x_i,\frac{r}{2A_0})$ are mutually disjoint and contained in $B(x,c_1 R)$ which in turn is contained in each $B(x_i,c_2 R)$ for some $c_1,c_2$ depending only on $A_0$.
Thus
\begin{equation*}
\begin{split}
  M\min_i\mu(B(x_i,\frac{r}{2A_0}))
  &\leq\sum_{i=1}^M\mu(B(x_i,\frac{r}{2A_0})) \\
  &\leq\mu(B(x,c_1 R)) \\
  &\leq\min_i\mu(B(x_i,c_2 R)) \\
  &\lesssim\Big(\frac{c_2 R}{\frac{1}{2A_0}r}\Big)^D\min_i\mu(B(x_i,\frac{r}{2A_0})).
\end{split}
\end{equation*}
Cancelling out the common term from both sides and absorbing inessential factors into the implied constant, it follows that $M\lesssim(R/r)^D$. Since $M$ was the size of an arbitrary $r$-separated subset of $B(x,R)$, we get $M_{B(x,R)}\lesssim(R/r)^D$ and hence $\Delta\leq D$.

Finally, suppose that $B(x,R)$ can be covered by $N\leq C'(R/r)^{\Delta'}$ balls $B(x_i,r)$, while $\mu(B(y,r))\lesssim(r/R)^d\mu(B(y,R))$.
We may assume that there is some $z_i\in B(x_i,r)\cap B(x,R)$, for otherwise $B(x_i,r)$ is not needed in the cover, and then it is easy to check that $B(x_i,R)\subset B(x,cR)$ for some $c$ depending only on $A_0$.
\begin{equation*}
\begin{split}
  \mu(B(x,R))
  &\leq\sum_{i=1}^N\mu(B(x_i,r))
  \lesssim\sum_{i=1}^N(r/R)^d\mu(B(x_i,R)) \\
  &\leq\sum_{i=1}^N(r/R)^d\mu(B(x,cR)) 
  \lesssim(R/r)^{\Delta'}(r/R)^d\mu(B(x,cR)).
\end{split}
\end{equation*}
If $d>\Delta'$, the right-hand side converges to zero as $r\to 0$, which is absurd. Hence it is necessary that $d\leq\Delta'$.
\end{proof}

\section{Example: the Bessel setting}\label{sec:Bessel}

This section serves as an illustration of the assumptions of Theorem \ref{thm:main} in the concrete ``Bessel setting'' previously considered in \cite{FLLX}. We detail the claims of Example \ref{ex:Bessel} on getting the results of \cite{FLLX} as special cases of Theorem \ref{thm:main} and also provide some extensions. None of the subsequent sections depend on the current one, which can be skipped by a reader uninterested in this particular setting.

\begin{definition}\label{def:Bessel}
The ``Bessel setting'', with parameters $n\in\{1,2,\ldots\}$,
\begin{equation*}
   \vec\lambda=(\lambda_i)_{i=1}^n\in(-\frac12,\infty)^n,\quad
   \vec\alpha=(\alpha_i)_{i=1}^n\in\{0,1\}^n
\end{equation*}
is the metric measure space $(X,\rho,\mu)=(\R^n_{\vec\lambda,\vec\alpha},\abs{x-y},\mu_{\vec\lambda}^{(n)})$, where $\abs{x-y}$ is the Euclidean metric on
\begin{equation*}
  \R^n_{\vec\lambda,\vec\alpha}:=\prod_{i=1}^n\R_{\lambda_i,\alpha_i},\qquad 
  \R_{\lambda_i,\alpha_i}
  :=\begin{cases} \R, & \text{if}\quad\lambda_i=0=\alpha_i, \\
  [0,\infty), & \text{otherwise}, \end{cases}
\end{equation*}
and
\begin{equation*}
  \ud\mu_{\vec\lambda}^{(n)}(x):= \prod_{i=1}^n\ud\mu_{\lambda_i}^{(1)}(x_i),\qquad
  \ud\mu_{\lambda_i}^{(1)}(x_i):=x_i^{2\lambda_i}\ud x_i.
\end{equation*}
\end{definition}

Note that the parameter $\vec\alpha$ plays a role only if some of the components of $\vec\lambda$ are zero; we have incorporated this parameter to cover slight variants of the setting appearing in the literature. In \cite{CZ:14}, general $\vec\lambda$ with $\vec\alpha=(1,\ldots,1)$ is considered, while \cite{FLLX} deals with $\vec\lambda=(0,\ldots,0,\lambda)$ with $\lambda>0$ and $\vec\alpha=(0,\ldots,0,1)$.

\begin{proposition}\label{prop:Bessel}
The Bessel setting with parameters $n\in\{1,2,\ldots\}$, $\vec\lambda\in(-\frac12,\infty)^n$, and $\vec\alpha\in\{0,1\}^n$ is a complete doubling metric measure space supporting the $(1,1)$-Poincar\'e inequality and having separating dimension $\Delta=n$ and lower and upper dimension $d$ and $D$, where
\begin{equation*}
\begin{split}
  d :=\sum_{i=1}^n d_i\in(0,n],\quad d_i &:=1-2\lambda_{i,-}\in(0,1],\\
  D :=\sum_{i=1}^n D_i\in[n,\infty),\quad D_i &:=1+2\lambda_{i,+}\in[1,\infty).
\end{split}
\end{equation*}
\end{proposition}

\begin{remark}\label{rem:Bessel}
Accepting Proposition \ref{prop:Bessel} for the moment, we can compare the implications of Theorem \ref{thm:main} in the Bessel setting with the results of \cite{FLLX} obtained by working explicitly in this situation. We note that \cite{FLLX} only considers non-negative $\vec\lambda$, and among them only those of the special form $\vec\lambda=(0,\ldots,0,\lambda)$.
\begin{enumerate}[\rm(1)]
  \item If $n\geq 2$ and $\lambda_i\geq 0$ for at least two indices $i$, then $d\geq 2$. In this case, Theorem \ref{thm:main} characterises the $S^p$ properties of $[b,T]$ for all $p\in(0,\infty)$ and all Calder\'on--Zygmund operators of any regularity $\eta\in(0,1]$. This is a significant generalisation of the corresponding results of \cite{FLLX}, where only the Bessel--Riesz transforms are considered. (Note that our $n$ corresponds to $n+1$ in the notation of \cite{FLLX}.)
  \item In particular, if $n\geq 3$, all vectors $\vec\lambda=(0,\ldots,0,\lambda)$ satisfy the previous case, also for negative $\lambda$, Thus, for such dimensions, we have a full extension of the results of \cite{FLLX} to $\lambda\in(-\frac12,0]$, which is new even for the Bessel--Riesz transforms. For $n=2$ and $\vec\lambda=(0,\lambda)$ with $\lambda\in(-\frac12,0]$, we have $d=2+2\lambda>1$ and $\Delta=D=2$. With $\eta=1$, it follows that
\begin{equation*}
  \Big(\frac{\eta}{\Delta}+\frac{1}{2}\Big)^{-1}=1
\end{equation*}
and hence we still achieve a characterisation of the $S^p$ properties, for all $p\in(0,\infty)$, of all Calder\'on-Zygmund operators with regularity $\eta=1$; this covers in particular the Bessel--Riesz transforms considered in \cite{FLLX}.
  \item If $n=1$, then $d\in(0,1]$ and $\Delta=1$. 
Now $(\eta/\Delta+\frac12)^{-1}\leq 1$ provided that $\eta\geq\frac12$, and we characterise the $S^p$ properties of $[b,T]$ for all $p\in(1,\infty)$ and all Calder\'on-Zygmund operators with regularity $\eta\geq\frac12$. For the Bessel--Riesz transforms, this recovers the result of \cite{FLLX} for $\lambda\geq 0$ and is new for $\lambda\in(-\frac12,0)$.
\end{enumerate}
\end{remark}

For the proof of Proposition \ref{prop:Bessel}, we begin with two lemmas related to the Poincar\'e inequality.

\begin{lemma}\label{lem:BesselPoinc1}
Let $\lambda\in(-\frac12,\infty)$ and $0\leq a<b<\infty$, or $\lambda=0$ and $-\infty<a<b<\infty$. If $f$ is a Lipschitz function on $(a,b)$, then
\begin{equation}\label{eq:BesselPoinc1}
  \fint_a^b\fint_a^b\abs{f(x)-f(y)}\ud \mu_\lambda^{(1)}(x)\ud \mu_\lambda^{(1)}(y)
  \leq 2(b-a)\fint_a^b \abs{f'(z)}\ud \mu_\lambda^{(1)}(z).
\end{equation}
\end{lemma}

\begin{proof}
Let $0\leq a<b<\infty$, and note that
\begin{equation}\label{eq:BesselPoinc10}
  \mu_\lambda^{(1)}(a,b)
  =\int_a^b x^{2\lambda}\ud x
  =\frac{1}{2\lambda+1}(b^{2\lambda+1}-a^{2\lambda+1}).
\end{equation}

For $a<x<y<b$, we have
\begin{equation*}
  \abs{f(x)-f(y)}=\Babs{\int_x^y f'(z)\ud z}\leq\int_a^b 1_{\{x<z<y\}}\abs{f'(z)}\ud z.
\end{equation*}
Using the symmetry of the halves of the integral with $x<y$ and $y<x$, it follows that
\begin{equation*}
\begin{split}
  \operatorname{LHS}\eqref{eq:BesselPoinc1}
  &=2\fint_a^b\fint_a^b 1_{\{x<y\}}\abs{f(x)-f(y)}\ud \mu_\lambda^{(1)}(x)\ud\mu_\lambda^{(1)}(y), \\
  &\leq 2\int_a^b\abs{f'(z)} H(z)\ud z,\qquad
  H(z):=\fint_a^b\fint_a^b 1_{\{x<z<y\}}\ud \mu_\lambda^{(1)}(x)\ud\mu_\lambda^{(1)}(y).
\end{split}
\end{equation*}

Making one or the other of the trivial estimates $1_{\{x<z<y\}}\leq 1_{\{x<z\}}$ or $1_{\{x<z<y\}}\leq 1_{\{z<y\}}$, we get the alternative bounds
\begin{equation*}
  H(z)
  \leq\begin{cases}
 \displaystyle   \fint_a^b 1_{\{x<z\}}\ud \mu_\lambda^{(1)}(x)
 =\mu_\lambda^{(1)}(a,z)/\mu_\lambda^{(1)}(b,a), \\
% =\frac{z^{2\lambda+1}-a^{2\lambda+1}}{b^{2\lambda+1}-a^{2\lambda+1}}, \\
 \displaystyle   \fint_a^b 1_{\{z<y\}}\ud \mu_\lambda^{(1)}(y)
  =\mu_\lambda^{(1)}(z,b)/\mu_\lambda^{(1)}(b,a). \\
% =\frac{b^{2\lambda+1}-z^{2\lambda+1}}{b^{2\lambda+1}-a^{2\lambda+1}}.
 \end{cases}
\end{equation*}
By \eqref{eq:BesselPoinc10} and the mean value theorem, for every $z\in(a,b)$ there are $\xi\in(a,z)$ and $\eta\in(z,b)$ such that
\begin{equation*}
\begin{split}
  \mu_\lambda^{(1)}(a,z)=(2\lambda+1)^{-1}(z^{2\lambda+1}-a^{2\lambda+1})
  =\xi^{2\lambda}(z-a)
  \leq z^{2\lambda}(b-a),\quad\text{if}\quad\lambda\geq 0, \\
  \mu_\lambda^{(1)}(z,b)=(2\lambda+1)^{-1}(b^{2\lambda+1}-z^{2\lambda+1})
  =\eta^{2\lambda}(b-z)
  \leq z^{2\lambda}(b-a),\quad\text{if}\quad\lambda\leq 0.
\end{split}
\end{equation*}
Substituting back, and choosing the favourable alternative bound depending on the sign of $\lambda$, we find that, in each case,
\begin{equation*}
  H(z)\leq\frac{z^{2\lambda}(b-a)}{\mu_\lambda^{(1)}(a,b)}.
\end{equation*}
Hence
\begin{equation*}
\begin{split}
  \operatorname{LHS}\eqref{eq:BesselPoinc1}
  &\leq 2\int_a^b\abs{f'(z)}H(z)\ud z \\
  &=2(b-a)\frac{1}{\mu^{(1)}_\lambda(a,b)}\int_a^b\abs{f'(z)}z^{2\lambda}\ud z \\
  &=2(b-a)\fint_a^b\abs{f'(z)}\ud\mu_\lambda^{(1)}(z)=\operatorname{RHS}\eqref{eq:BesselPoinc1},
\end{split}
\end{equation*}
completing the proof of \eqref{eq:BesselPoinc1} in the case that $\lambda\in(-\frac12,\infty)$ and $0\leq a<b<\infty$. When $\lambda=0$, the same computation works for general $-\infty<a<b<\infty$; indeed, this case is just the usual Poincar\'e inequality for the Lebesgue measure.
\end{proof}

\begin{lemma}\label{lem:BesselPoincN}
Let $R=\prod_{i=1}^n I_i \subset\R^n_{\vec\lambda,\vec\alpha}$ be an axis-parallel rectangle, and $f$ be a Lipschitz function of $R$. Then
\begin{equation}\label{eq:BesselPoincN}
  \fint_R\fint_R\abs{f(x)-f(y)}\ud \mu^{(n)}_\lambda(x)\ud \mu^{(n)}_\lambda(y)
  \lesssim\operatorname{diam}(R)\fint_R\abs{\nabla f(z)}\ud \mu^{(n)}_\lambda(z).
\end{equation}
\end{lemma}

\begin{proof}
We express points of $\R^n$ as
\begin{equation*}
  x=(x_{[1,i)},x_i,x_{(i,n]})\in\R^{i-1}\times\R\times\R^{n-i},
\end{equation*}
where $x_{[1,i)}$ resp. $x_{(i,n]}$ is omitted if $i=1$ resp. $i=n$. 
Then
\begin{equation*}
  \abs{f(x)-f(y)}
  \leq\sum_{i=1}^n  \abs{ f(x_{[1,i)}, x_i, y_{(i,n]}) -f(x_{[1,i)}, y_i, y_{(i,n]})}.
\end{equation*}
Applying Lemma \ref{lem:BesselPoinc1} with each fixed $x_{[1,i)}$ and $y_{(i,n]}$,
\begin{equation*}
\begin{split}
  \fint_{I_i}\fint_{I_i} &\abs{ f(x_{[1,i)}, x_i, y_{(i,n]}) -f(x_{[1,i)}, y_i, y_{(i,n]})}\ud\mu_{\lambda_i}^{(1)}(x_i)\ud\mu_{\lambda_i}^{(1)}(y_i) \\
  &\lesssim\ell(I_i)\fint_{I_i}\abs{\partial_i f(x_{[1,i)}, z_i, y_{(i,n]})}\ud\mu_{\lambda_i}^{(1)}(z_i).
\end{split}
\end{equation*}
Integrating over the remaining variables and summing up, it follows that
\begin{equation*}
  \operatorname{LHS}\eqref{eq:BesselPoincN}
  \lesssim\sum_{i=1}^n\ell(I_i) \fint_R\abs{\nabla f(z)}\ud m_\lambda^{(n+1)}(z)
  \approx\operatorname{RHS}\eqref{eq:BesselPoincN},
\end{equation*}
which completes the argument.
\end{proof}

\begin{proof}[Proof of Proposition \ref{prop:Bessel}]
The completeness of $\R^{n}_{\vec\lambda,\vec\alpha}$ with the Euclidean metric is clear.
For $n=1$ and $\alpha_i=1$, it is easy to estimate
\begin{equation*}
  \mu_{\lambda_i}^{(1)}(B(x_i,r))
  =\int_{(x_i-r)_+}^{x_i+r} y^{2\lambda_i}\ud y
  \approx\left.\begin{cases} r\cdot x_i^{2\lambda_i}, & \text{if } x_i\geq 2r, \\
    r^{2\lambda_i+1}, & \text{if } x_i\in(0,2r) \end{cases}\right\}
    \approx r\max(x_i,r)^{2\lambda_i};
\end{equation*}
indeed, the integrand is essentially the constant $x_i^{2\lambda}$ in the first case, and in the second case we may estimate the integral from above by the integral over $[0,3r]$, which can be directly computed, and from below by the integral over  $[x_i+\frac12r,x_i+r]$, where the integrand is essentially the constant $r^{2\lambda}$. The case $\alpha_i=0$ makes a difference only when $\lambda_i=0$ also, and in this case the same result is obvious.
Thus, for $0<r<R<\infty$,
\begin{equation*}
  \frac{\mu_\lambda^{(1)}(B(x_i,R))}{\mu_\lambda^{(1)}(B(x_i,r))}
  \approx\Big(\frac{R}{r}\Big) \Big(\frac{\max(x_i,R)}{\max(x_i,r)}\Big)^{2\lambda_i}
  \begin{cases}\leq (R/r)^{1+2\lambda_{i,+}}, \\ \geq (R/r)^{1-2\lambda_{i,-}},\end{cases}
\end{equation*}

The case of general $n$ follows readily via comparison of balls from inside and outside with rectangles, and we obtain
\begin{equation*}
  \mu_{\vec\lambda}^{(n)}(B(x,r))\approx r^{n}\prod_{i=1}^n \max(x_i ,r)^{2\lambda_i}
\end{equation*}
and
\begin{equation*}
  \prod_{i=1}^n\Big(\frac{R}{r}\Big)^{1-2\lambda_{i,-}}
  \lesssim\frac{\mu_\lambda^{(n)}(B(x,R))}{\mu_\lambda^{(n)}(B(x,r))}
  \lesssim\prod_{i=1}^n\Big(\frac{R}{r}\Big)^{1+2\lambda_{i,+}},
\end{equation*}
which coincides with the claim about the lower and upper dimensions.

It remains to verify the $(1,1)$-Poincar\'e inequality.
This is obtained by dominating the mean oscillation over a ball by the same quantity over a containing rectangle, applying the rectangular Poincar\'e inequality of Lemma \ref{lem:BesselPoincN}, and then dominating the right-hand side by an integral over a containing ball again; this results in a somewhat expanded ball on the right-hand side, but this is permitted by our definition of the Poincar\'e inequality. Let us also note that a Lipschitz function on a domain of $\R^{n}$ is differentiable almost everywhere by Rademacher's theorem, and it is easy to see that $\abs{\nabla f(x)}=\operatorname{lip}f(x)$ at the points of differentiability.
\end{proof}

\begin{remark}
Proposition \ref{prop:Bessel} shows that the characterisation of constants in Proposition \ref{prop:Bp=const} applies, in particular, to the Bessel setting studied in \cite{FLLX}. As explained above, the set-up of \cite{FLLX} corresponds to parameters $n\in\{1,2,\ldots\}$, $\vec\lambda=(0,\ldots,0,\lambda)$ and $\vec\alpha=(0,\ldots,0,1)$. Proposition \ref{prop:Bessel} shows that the associated space has lower dimension $n$ and satisfies the $(1,1)$-Poincar\'e inequality, hence also the $(1,n)$-Poincar\'e inequality. Thus Proposition \ref{prop:Bp=const} guarantees that $f\in \dot B^n(\mu^{(n)}_{\vec\lambda})$ if and only if $f$ is constant almost everywhere. In \cite[Proposition 6.4]{FLLX}, this is obtained under the {\em a priori} assumption that $f\in C^2(\R_+^n)$, which is now seen to be redundant.
\end{remark}

\part{Commutator bounds via classical methods}\label{part:NWO}

\section{The Janson--Wolff method}\label{sec:JW}

We now turn to the proof those parts of Theorem \ref{thm:main} that admit a relatively straightforward adaptation of the classical methods of Janson and Wolff \cite{JW:82}, and Rochberg and Semmes \cite{RS:NWO} to spaces of homogeneous type.

We begin our study of the commutators by providing an extension to spaces of homogeneous type of some of the results of \cite{JW:82}. This method will be limited to Schatten class $S^p$ estimates for $p\in(2,\infty)$, but it has the advantage of requiring minimal prerequisites, and this section can be studied independently of the rest of the paper. Besides extending the classical Euclidean results of \cite{JW:82} to spaces of homogeneous type, we achieve two novelties even on $\R^d$:
\begin{enumerate}[\rm(1)]
  \item On the level of methods, we simplify the required prerequisites from interpolation theory, in that interpolation of Besov spaces spaces is completely avoided in favour of the simpler interpolation of $L^p$ spaces only.
  \item On the level of results, we improve the $S^p$ estimate of the single operator $[b,T]$ to a variational $S^p$ estimate of its truncations in Proposition \ref{prop:JWvar}.
\end{enumerate}

Turning to the details, for functions of two variables $x,y$, we denote
\begin{equation*}
  L^p(L^{q,r})_{\operatorname{symm}}
  :=L^p_x(L^{q,r}_y)\cap L^p_y(L^{q,r}_x).
\end{equation*}

A key lemma of \cite{JW:82} is the following:

\begin{lemma}[\cite{JW:82}, Lemma 1]\label{lem:JWlem1}
For $p\in(2,\infty)$, there is a continuous embedding
\begin{equation*}
    L^p(L^{p',\infty})_{\operatorname{symm}}\subseteq
    (L^\infty(L^1)_{\operatorname{symm}},L^2(L^2))_{\frac{2}{p},\infty},
\end{equation*}
where $(\ ,\ )_{\frac2p,\infty}$ denotes real interpolation.
\end{lemma}

Let us note that this result is valid for general $L^p$ and $L^{q,r}$ spaces on a $\sigma$-finite measure space, so in particular for those on spaces of homogeneous type $(X,\rho,\mu)$ that we deal with.

\begin{corollary}[\cite{JW:82}, Lemma 2]\label{cor:SpWeakJW}
Let $p\in(2,\infty)$ and $K\in L^p(L^{p',\infty})_{\operatorname{symm}}$. Then the integral operator
\begin{equation*}
  \mathcal{I}_K f(x):=\int_X K(x,y)f(y)\ud\mu(y)
\end{equation*}
belongs to the weak Schatten class $S^{p,\infty}(L^2(\mu))$ and
\begin{equation*}
  \Norm{\mathcal{I}_K}{S^{p,\infty}}\lesssim\Norm{K}{L^p(L^{p',\infty})_{\operatorname{symm}}}.
\end{equation*}
\end{corollary}

\begin{proof}
This follows by interpolation between $p=2$ and $p=\infty$:
For the Hilbert--Schmidt norm of integral operators, there is the explicit formula
\begin{equation}\label{eq:IKS2}
  \Norm{\mathcal{I}_K}{S^2}=\Norm{K}{L^2(L^2)}.
\end{equation}
On the other hand, it is also well known that
\begin{equation}\label{eq:IKL}
  \Norm{\mathcal{I}_K}{\bddlin(L^2(\mu))}\leq\big(\Norm{K}{L^\infty_x(L^1_y)}\Norm{K}{L^\infty_y(L^1_x)}\big)^{\frac12}
  \leq\Norm{K}{L^\infty(L^1)_{\operatorname{symm}}}.
\end{equation}
We will then use classical interpolation results for operators on the Hilbert space $H:=L^2(\mu)$. Denoting by $S^\infty(H)=\mathcal K(H)$ the space of all compact operators on $H$, and noting that this is the closure of Hilbert--Schmidt operators in the space $\bddlin(H)$ of all bounded linear operators on $H$, i.e.,
\begin{equation}\label{eq:KHisClosure}
  \mathcal K(H)=\overline{S^2(H)}^{\bddlin(H)}=\overline{S^2(H)\cap\bddlin(H)}^{\bddlin(H)},
\end{equation}
we have the following identities:
\begin{equation}\label{eq:SpInfInterp}
\begin{split}
     S^{p,\infty}(H) &=(S^2(H),\mathcal K(H))_{\frac2p,\infty}\qquad\text{by \cite[Th\'eor\`eme 1]{Merucci}} \\
     &=\big(S^2(H),\overline{S^2(H)\cap\bddlin(H)}^{\bddlin(H)}\big)_{\frac2p,\infty}\qquad\text{by \eqref{eq:KHisClosure}} \\
     &=(S^2(H),\bddlin(H))_{\frac2p,\infty}\qquad\text{by \cite[Theorem 3.4.2(d)]{BL:book}}.
\end{split}
\end{equation}
Thus
\begin{equation*}
\begin{split}
  \Norm{\mathcal{I}_K}{S^{p,\infty}(L^2(\mu))}
  &=\Norm{\mathcal{I}_K}{\big(\bddlin(L^2(\mu)),\ S^2(L^2(\mu))\big)_{\frac2p,\infty}}\qquad\text{by \eqref{eq:SpInfInterp}}   \\
  &\leq\Norm{K}{\big(L^\infty(L^1)_{\operatorname{symm}},\ L^2(L^2)\big)_{\frac2p,\infty}}
  \qquad\text{interpolating \eqref{eq:IKS2}, \eqref{eq:IKL}} \\
  &\lesssim\Norm{K}{ L^p(L^{p',\infty})_{\operatorname{symm}} }\qquad\qquad\qquad\text{by Lemma \ref{lem:JWlem1}}.
\end{split}
\end{equation*}
This is the claimed estimate.
\end{proof}

In order to apply Corollary \ref{cor:SpWeakJW} to concrete kernels featuring the $V(x,y)^{-1}$ type singularity, the following lemma is useful:

\begin{lemma}\label{lem:VinvWeak}
Let $(X,\rho,\mu)$ be a space of homogeneous type. Then
\begin{equation*}
  \Norm{V^{-1}}{L^\infty(L^{1,\infty})_{\operatorname{symm}}}\lesssim 1.
\end{equation*}
\end{lemma}

\begin{proof}
Since we have $V(x,y)\approx V(y,x)$ by doubling, it suffices to estimate the notm $\Norm{V^{-1}}{L^\infty_x(L^1_y)}$.
Since $\mu(B(x,r))\approx\mu(\bar B(x,r))$ by doubling, we may further assume, for this proof, that $V(x,y)=\mu(\bar B(x,\rho(x,y)))$ is defined with the balls $\bar B(x,r):=\{y\in X:\rho(x,y)\leq r\}$.

For each fixed $x\in X$, we observe that
\begin{equation*}
\begin{split}
  \{y\in X:V(x,y)^{-1}>\lambda\}
  &=\{y\in X:\mu(\bar B(x,d(x,y))<\lambda^{-1}\} \\
  &=\bigcup_{\substack{r\in[0,\infty) \\ V(x,r)<\lambda^{-1}}}\bar B(x,r)
  =\bigcup_{n=1}^\infty\bar B(x,r_n)
\end{split}
\end{equation*}
for some $r_n\uparrow\sup\{r\in[0,\infty): V(x,r)<\lambda^{-1}\}$ if the supremum is not achieved, or $r_n$ equal to this supremum otherwise. From the continuity of the measure it then follows that
\begin{equation*}
  \mu(\{y\in X:V(x,y)^{-1}>\lambda\})
  =\lim_{n\to\infty}\mu(\bar B(x,r_n))
  =\lim_{n\to\infty}V(x,r_n)
  \leq\lambda^{-1},
\end{equation*}
so in fact
\begin{equation*}
  \Norm{V(x,y)^{-1}}{L^{1,\infty}_y}\leq 1,\qquad
  \Norm{V^{-1}}{L^\infty_x(L^{1,\infty}_y)}\leq 1,
\end{equation*}
and this estimate did not even require the doubling condition.
\end{proof}

In the following lemma, we obtain $S^p$ estimates for a class of integral operators slightly more general than the commutators that we eventually aim at. This added generality actually simplifies matters, as we will discuss in more detail below.

\begin{proposition}\label{prop:JWtrick}
Let $K$ be a kernel that satisfies \eqref{eq:CZ0}.
Then for any measurable function $B$ on $X\times X$ and exponent $p\in(2,\infty)$, the integral operator $\mathcal{I}_{BK}$ with kernel $BK$ satisfies
\begin{equation*}
  \Norm{{\mathcal I}_{BK}}{S^p(L^2(\mu))}\lesssim\Norm{B}{L^p(V^{-2})}
  :=\Big(\iint_{X\times X}\frac{\abs{B(x,y)}^p}{V(x,y)^2}\ud\mu(x)\ud\mu(y)\Big)^{\frac1p}.
\end{equation*}
\end{proposition}

\begin{proof}
We begin with the estimates
\begin{equation}\label{eq:BKvsBV}
\begin{split}
  \Norm{BK}{ L^p(L^{p',\infty}) }
  &\lesssim\BNorm{\frac{B}{V}}{ L^p(L^{p',\infty}) } \\
  &=\BNorm{\frac{B}{V^{\frac2p}}\frac{1}{V^{\frac1q}}}{ L^p(L^{p',\infty}) },\qquad\frac1q:=1-\frac{2}{p}=\frac{1}{p'}-\frac{1}{p}, \\
  &\lesssim\BNorm{\frac{B}{V^{\frac2p}}}{L^p(L^p)}\BNorm{\frac{1}{V^{\frac1q}}}{ L^\infty(L^{q,\infty}) } \\
  &=\Norm{B}{L^p(V^{-2})}\Norm{V^{-1}}{L^\infty(L^{1,\infty})}^{\frac1q}.
\end{split}
\end{equation}
where we used H\"older's inequality with $\frac1p=\frac1p+\frac{1}{\infty}$ and weak-type H\"older's inequality with $\frac{1}{p'}=\frac{1}{p}+\frac{1}{q}$. The previous computation is valid for either order of the variables in the iterated norms, and hence
\begin{equation}\label{eq:IBKweak}
\begin{split}
  \Norm{ \mathcal{I}_{BK} }{S^{p,\infty}}
  &\lesssim\Norm{BK}{ L^p(L^{p',\infty})_{\operatorname{symm}} }\qquad \text{(Corollary \ref{cor:SpWeakJW})} \\
  &\lesssim\Norm{B}{L^p(V^{-2})}\Norm{V^{-1}}{L^\infty(L^{1,\infty})_{\operatorname{symm}}}^{\frac1q}
  \quad \text{(\eqref{eq:BKvsBV} symmetrized)} \\
  &\lesssim\Norm{B}{L^p(V^{-2})}\qquad \text{(Lemma \ref{lem:VinvWeak})}.
\end{split}\end{equation}

To finish, we apply interpolation again. Given $p\in(2,\infty)$, we choose $2<p_0<p<p_1<\infty$ and $\theta\in(0,1)$ such that $\frac1p=\frac{1-\theta}{p_0}+\frac{\theta}{p_1}$. Then
\begin{equation}\label{eq:JWvariant}
\begin{split}
   \Norm{ \mathcal{I}_{BK} }{S^{p}}
   &=\Norm{ \mathcal{I}_{BK} }{(S^{p_0,\infty},S^{p_1,\infty})_{\theta,p}}
   \qquad\text{by \cite[Th\'eor\`em 2]{Merucci}}\\
   &\lesssim\Norm{B}{(L^{p_0}(V^{-2}),L^{p_1}(V^{-2}))_{\theta,p}}
   \qquad\text{interpolating \eqref{eq:IBKweak}} \\
   &=\Norm{B}{L^{p}(V^{-2})}\qquad\text{by \cite[Theorem 5.2.1]{BL:book}},
\end{split}
\end{equation}
where the last step is just the usual interpolation of $L^p$ spaces.
\end{proof}

The desired commutator estimate is now an immediate consequence:

\begin{corollary}\label{cor:JWSpBp}
Let $K$ be a kernel that satisfies \eqref{eq:CZ0} and $T$ be the (not necessarily bounded) integral operator with kernel $K$. If $b\in B^p$, then the formal commutator
\begin{equation*}
  [b,T]f(x):=\int_X (b(x)-b(y))K(x,y)f(y)\ud\mu(y)
\end{equation*}
defines a bounded operator of Schatten class $S^p$, and
\begin{equation*}
  \Norm{[b,T]}{S^p}\lesssim\Norm{b}{B^p}
  :=\Big(\iint_{X\times X}\frac{\abs{b(x)-b(y)}^p}{V(x,y)^2}\ud\mu(x)\ud\mu(y)\Big)^{\frac1p}.
\end{equation*}
\end{corollary}

\begin{proof}
This is an immediate application of Proposition \ref{prop:JWtrick} to the function $B(x,y):=b(x)-b(y)$.
\end{proof}

\begin{remark}
What we did so far in this section is largely the natural adaptation of the sufficiency part of \cite[Theorem 1]{JW:82} to $(X,\rho,\mu)$ in place of $\R^d$. However, we have chosen to derive the commutator result of Corollary \ref{cor:JWSpBp} from the more general statement of Proposition \ref{prop:JWtrick} involving general functions $B(x,y)$ in place of $b(x)-b(y)$. This generality actually simplifies the argument as follows: In the very last step \eqref{eq:JWvariant} of the proof of Proposition \ref{prop:JWtrick}, we simply apply the usual interpolation of $L^p$ spaces.
By considering commutators only, \cite{JW:82} first prove that
\begin{equation*}
  \Norm{[b,T]}{S^{p,\infty}}\lesssim\Norm{b}{B^p}. 
\end{equation*}
In order to bootstrap this to the desired strong type bound, they argue as in \eqref{eq:JWvariant} but need the interpolation result
\begin{equation*}
  (B^{p_0},B^{p_1})_{\theta,p}=B^p
\end{equation*}
for Besov spaces $B^p$ in the last step. When $B^p=B^{d/p}_{p,p}(\R^d)$ as in \cite{JW:82}, this is an instance of the classical interpolation theorem (e.g. \cite[Theorem 6.4.5(3)]{BL:book})
\begin{equation*}
  (B^{s_0}_{p_0,q_0},B^{s_1}_{p_1,q_1})_{\theta,p}
  =B^s_{p,p},\quad \begin{cases} \theta\in(0,1),\quad s_0\neq s_1,\quad s=(1-\theta)s_0+\theta s_1, \\  \displaystyle
   p_h,q_h\in[1,\infty], \ \frac1p=\frac{1-\theta}{p_0}+\frac{\theta}{p_1}=\frac{1-\theta}{q_0}+\frac{\theta}{q_1} \end{cases}
\end{equation*}
with parameters $s_h=d/p_h$ and $q_h=p_h$ that fit into the above assumptions.

But for Besov spaces on a general doubling space, the issue is more tricky. For example, results in \cite{GKS:10,Yang:04} only deal with interpolation of $B^{s_h}_{p,q_h}$ for fixed $p$, corresponding to the Euclidean result in \cite[Theorem 6.4.5(1)]{BL:book}, and this is not what is needed for the present purposes. One should also note that the proofs of the Besov space interpolation results tend to depend on a sequence of other results in the rich theory of these spaces. Hence, even if a suitable interpolation theorem was available, the reader may agree on the simplicity of our approach that only uses the definition of the Besov space together with classical interpolation of $L^p$ spaces.
\end{remark}

A slight modification of the proof of Corollary \ref{cor:JWSpBp} gives the following elaboration, which seems to be new even for $X=\R^d$.

\begin{proposition}\label{prop:JWvar}
Let $K$ be a kernel that satisfies \eqref{eq:CZ0}.
For each measurable $E\subset X\times X$, the integrals
\begin{equation*}
  [b,T]_E f(x):=\int_X 1_E(x,y)(b(x)-b(y))K(x,y)f(y)\ud\mu(y)
\end{equation*}
define bounded operators of Schatten class $S^p$, which moreover satisfy the variational estimate
\begin{equation}\label{eq:variationalSp}
  \sup_{\mathscr E}\Big(\sum_{E\in\mathscr E}\Norm{[b,T]_E}{S^p}^p\Big)^{\frac1p}\lesssim\Norm{b}{\dot B^p(\mu)},
\end{equation}
where the supremum is over all disjoint collections $\mathscr E$ of measurable subsets $E\subset X\times X$.
\end{proposition}

\begin{proof}
It is immediate that each $K_E(x,y):=1_E(x,y)K(x,y)$ inherits estimate \eqref{eq:CZ0} from $K$, and hence the $S^p$ property of each individual $[b,T]_E$ is just an application of Corollary \ref{cor:JWSpBp} to $K_E$ in place of $K$.

To obtain the variational estimate, we again apply Proposition \ref{prop:JWtrick}. We may equally well interpret $[b,T]_E=\mathcal{I}_{B_E K}$, where
\begin{equation*}
  B_E(x,y):=1_E(x,y)(b(x)-b(y)).
\end{equation*}
Hence, by Proposition \ref{prop:JWtrick},
\begin{equation*}
  \Norm{[b,T]_E}{S^p}=\Norm{\mathcal{I}_{B_E K}}{S^p}
  \lesssim\Norm{B_E}{L^p(V^{-2})},
\end{equation*}
and thus
\begin{equation*}
\begin{split}
  \sum_{E\in\mathscr E}  \Norm{[b,T]_E}{S^p}^p
  &\lesssim \sum_{E\in\mathscr E}\int_X\int_X 1_E(x,y)\frac{\abs{b(x)-b(y)}^p}{V(x,y)^2}\ud\mu(x)\ud\mu(y) \\
  &= \int_X\int_X  \sum_{E\in\mathscr E} 1_E(x,y)\frac{\abs{b(x)-b(y)}^p}{V(x,y)^2}\ud\mu(x)\ud\mu(y) \\
  &\leq \int_X\int_X \frac{\abs{b(x)-b(y)}^p}{V(x,y)^2}\ud\mu(x)\ud\mu(y) =\Norm{b}{\dot B^p(\mu)}^p,
\end{split}  
\end{equation*}
where the last inequality used the fact that
\begin{equation*}
  \sum_{E\in\mathscr E}1_E\leq 1
\end{equation*}
by the assumed disjointness of $E\in\mathscr E$.
\end{proof}

\begin{remark}
For sets of the form
\begin{equation*}
  E=E_I=\{(x,y)\in X\times X: \rho(x,y)\in I\},
\end{equation*}
where $I$ are (open, half-open, or closed) subintervals of $(0,\infty)$, the operators $[b,T]_{E_I}$ are simply commutators $[b,T_I]$ of the truncated singular integrals
\begin{equation*}
  T_I f(x)=\int_X 1_I(\rho(x,y))K(x,y)f(y)\ud\mu(y),
\end{equation*}
whose variational estimates have been widely studied in different contexts; see \cite{CJRW:00,CJRW:03} and the many subsequent works citing these papers.
 Compared to most such results in the literature, Proposition \ref{prop:JWvar} is surprisingly simple. This is also reflected in the fact that the estimate is obtained for general sets $E$, instead of those of the special form $E=E_I$.
\end{remark}

\section{Dyadic cubes and Haar functions}\label{sec:cubes}

The considerations in the rest of this paper will rely on the notion of dyadic cubes. Thus we interrupt the development of the commutator theory to recall the relevant definitions and results.

\begin{definition}\label{def:cubes}
A {\em system of dyadic cubes} $\mathscr D$, or just a {\em dyadic system}, on the space of homogeneous type $(X,\rho,\mu)$ is a collection
\begin{equation*}
  \mathscr D=\bigcup_{k\in\Z}\mathscr D_k,
\end{equation*}
where
\begin{enumerate}
  \item\label{it:Dpartition} each $\mathscr D_k$ is a measurable partition of $X$;
  \item\label{it:Drefine} each $\mathscr D_{k+1}$ refines the previous $\mathscr D_k$;
  \item\label{it:DlikeB} for parameters $\delta\in(0,1)$ and $0<c_0\leq C_0<\infty$, each $Q\in\mathscr D_k$ is essentially a ball of size $\delta^k$, in the sense that, for some ``centre'' $z_Q\in X$,
\begin{equation}\label{eq:cubeVsBalls}
  B(z_Q,c_0\delta^k)\subseteq Q\subseteq B(z_Q,C_0\delta^k).
\end{equation}
\end{enumerate}
The dyadic system is said to be {\em connected} if, in addition,
\begin{enumerate}\setcounter{enumi}{3}
  \item\label{it:Dconnected} for any two $P,Q\in\mathscr D$, there is some $R\in\mathscr D$ such that $P\cup Q\subseteq R$.
\end{enumerate}
\end{definition}

The original construction of such a system, for any space of homogeneous type, is due to Christ \cite{Christ:90}, except that the partition and refinement properties are only obtained up to sets of measure zero; a construction as stated can be found in \cite{HK:12}, including the possibility of requiring the system to be connected. Note that the standard dyadic system $\mathscr D=\{2^{-k}([0,1)^d+m):k\in\Z,m\in\Z^d\}$ of $\R^d$ is not connected, but it is easy to modify the construction to achieve this additional property.

At a later point, we will also need the notion of random dyadic systems (see Section \ref{sec:random}); however, we wish to stress the point that one fixed dyadic system is quite sufficient for a substantial part of the commutator theory, and we will thus introduce the additional random ingredients only when they are needed.

These dyadic cubes share most essential features of the familiar dyadic cubes of $\R^d$, with a notable exception related to property \eqref{it:Drefine}. In $\R^d$, this refinement is strict; every cube of a finer scale is a strict subset of some cube of a coarser scale. This remain true for all spaces with a positive lower dimension $d>0$ provided that the parameter $\delta>0$, which describes the change of scale between two consecutive levels of cubes, is chosen small enough, but it need not be the case for general spaces of homogeneous type. In other words, the same cube $Q$, as a set, may appear as a member of several different $\mathscr D_k$.

\begin{lemma}\label{lem:DuniqueLevel}
If $(X,\rho,\mu)$ has positive lower dimension $d>0$, then the dyadic systems $\mathscr D$ can be constructed so that each $Q\in\mathscr D$ appears in only one level $\mathscr D_k$.
\end{lemma}

\begin{proof}
We show that the claim is achieved as soon as the parameter $\delta\in(0,1)$ is chosen small enough, noting that $c_0,C_0$ in \eqref{eq:cubeVsBalls} are independent of $\delta\in(0,1)$. Thus, if $Q\in\mathscr D_{k+1}$ and $R\in\mathscr D_k$ satisfy $Q\subseteq R$ then, by the quasi-triangle inequality with constant $A_0$,
\begin{equation}\label{eq:BzQzR}
  B(z_Q,C_0\delta^k)\subseteq B(z_R,2A_0\delta^k)
\end{equation}
and hence
\begin{equation*}
\begin{split}
    \mu(Q) &\leq\mu(B(z_Q,C_0\delta^{k+1}))\qquad\text{by \eqref{eq:cubeVsBalls}} \\
    &\lesssim\Big(\frac{C_0\delta^{k+1}}{C_0\delta^k}\Big)^d \mu(B(z_Q,C_0\delta^{k}))
    \qquad\text{by the assumed lower dimension }d \\
    &\leq\delta^d \mu(B(z_R,2A_0 C_0\delta^{k}))\qquad\text{by \eqref{eq:BzQzR}} \\
    &\lesssim\delta^d\Big(\frac{2A_0 C_0\delta^k}{c_0\delta^k}\Big)^D\mu(B(z_R,c_0\delta^k)
    \qquad\text{by finite upper dimension }D \\
    &\lesssim \delta^d\mu(R)\qquad\text{by \eqref{eq:cubeVsBalls}},
\end{split}
\end{equation*}
where we used the fact that every space of homogeneous type has some finite upper dimension. Fixing $\delta\in(0,1)$ small enough, depending only on the parameters of the space in the implied constants, we find that $\mu(Q)<\mu(R)$, and hence every $R\in\mathscr D_k$ must be a union of more than one cube $Q\in\mathscr D_{k+1}$.
\end{proof}

While Lemma \ref{lem:DuniqueLevel} covers many concrete spaces of interest, and in particular all those for which we state conclusions \eqref{it:p<d} and \eqref{it:p=d} of Theorem \ref{thm:main}, a positive lower dimension is not needed for conclusion \eqref{it:p>d} of Theorem \ref{thm:main}, and thus we take the slight additional trouble to deal with the possible ambiguity of the level of dyadic cubes in general spaces of homogeneous type. At times, we wish to view a cube simply as a set, but there will also be occasions, where it is useful to understand it as the pair $(Q,k)$, also containing the information about the level $k$.

For $Q\in\mathscr D_k$ (where we understand $Q$ as the pair $(Q,k)$), we denote by $Q^{(1)}=(Q,k)^
{(1)}$ the unique $R\in\mathscr D_{k-1}$ such that $Q\subseteq R$, and by
\begin{equation*}
  \operatorname{ch}(Q):=\operatorname{ch}(Q,k):=\{Q'\in\mathscr D_{k+1}:(Q')^{(1)}=Q\}.
\end{equation*}
We refer to $Q^{(1)}$ and $\operatorname{ch}(Q)$ as the (dyadic) parent and the (dyadic) children of $Q$, respectively.
Recursively, we define $Q^{(j)}:=(Q^{(j-1)})^{(1)}$ and $\operatorname{ch}^{(j)}(Q):=\operatorname{ch}(\operatorname{ch}^{(j-1))}(Q))$.
It follows from the doubling property that
\begin{equation*}
  M_Q:=\#\operatorname{ch}(Q)\leq M<\infty
\end{equation*}
for some finite constant $M$ independent of $Q\in\mathscr D$.

For $Q\in\mathscr D$ (where we do not specify the level $k$), we denote by $Q^{[1]}$ the minimal $R\in\mathscr D$ such that $Q\subsetneq R$, and by
\begin{equation*}
  [\operatorname{ch}](Q):=\{Q'\in\mathscr D:(Q')^{[1]}=Q\}.
\end{equation*}
We refer to $Q^{[1]}$ and $[\operatorname{ch}](Q)$ as the strict parent and the strict children of $Q$, respectively.
Note that $Q^{[1]}$ exists unless $Q=X$, and $[\operatorname{ch}](Q)\neq\varnothing$ unless $Q=\{x\}$ for some $x\in X$.
Recursively, we define $Q^{(j)}:=(Q^{(j-1)})^{(1)}$ and $\operatorname{ch}^{(j)}(Q):=\operatorname{ch}(\operatorname{ch}^{(j-1))}(Q))$.

For $Q\in\mathscr D$, let
\begin{equation}\label{eq:kMinMaxQ}
\begin{split}
   k_{\min}(Q)&:=\inf\{ k\in\Z: Q\in\mathscr D_k\}\in\Z\cup\{-\infty\},\\
   k_{\max}(Q)&:=\sup\{ k\in\Z: Q\in\mathscr D_k\}\in\Z\cup\{\infty\},
\end{split}
\end{equation}
where $k_{\min}(Q)=-\infty$ if and only if $Q=X$, and $k_{\max}(Q)=\infty$ if and only if $Q=\{x\}$ for some $x\in X$.
Then we observe that
\begin{equation*}
  Q^{[1]}=(Q,k_{\min}(Q))^{(1)},\qquad
  [\operatorname{ch}]Q=\operatorname{ch}(Q,k_{\max}(Q))
\end{equation*}
whenever $k_{\min}(Q)$ resp. $k_{\max}(Q)$ is finite. Whenever $[\operatorname{ch}](Q)\neq\varnothing$, it follows that also
\begin{equation}\label{eq:strChQ>1}
  2\leq\#[\operatorname{ch}(Q)]\leq M<\infty.
\end{equation}

\begin{proof}[Proof of \eqref{eq:strChQ>1}]
Let $Q\in\mathscr D$ satisfy $[\operatorname{ch}](Q)\neq\varnothing$. Thus, there is at least one $Q'\in[\operatorname{ch}](Q)$. By definition, this means that $(Q')^{[1]}=Q$, thus $Q'\subsetneq Q$. If $Q'\in\mathscr D_{k'}$ and $Q\in\mathscr D_k$ for some $k',k\in\Z$, it follows from properties \eqref{it:Dpartition} and \eqref{it:Drefine} of dyadic cubes from Definition \ref{def:cubes} that $k'>k$. Hence both $k_{\min(Q')},k_{\max(Q)}$ are finite and $k_{\min(Q')}>k_{\max(Q)}$.

Suppose for contradiction that there exists $\hat k\in\Z$ with $k_{\min(Q')}>\hat k>k_{\max(Q)}$. By properties \eqref{it:Dpartition} and \eqref{it:Drefine} of dyadic cubes, there exists $\hat Q\in\mathscr D_{\hat k}$ with $Q'\subseteq\hat Q\subseteq Q$. Since $Q',Q\notin\mathscr D_{\hat k}$ by definition of $k_{\min(Q')}$ and $k_{\max(Q)}$, both containments are strict. But then $Q$ would not be a minimal dyadic cube strictly containing $Q'$, which is the contradiction that we wanted. Thus $k_{\min(Q')}=k_{\max(Q)}+1=:k_0+1$.

Now $Q'\subsetneq Q$, where $Q'\in\mathscr D_{k_0+1}$ and $Q\in\mathscr D_{k_0}$. 
By properties \eqref{it:Dpartition} and \eqref{it:Drefine} of dyadic cubes again, $Q$ is a disjoint union of the cubes $Q_i\in\operatorname{ch}(Q,k_0)\subseteq\mathscr D_{k_0+1}$, one of them being $Q'$. Since $Q'\neq Q$, there are at least two such cubes $Q_1\neq Q_2$. Moreover, all these $Q_i$ satisfy $(Q_i)^{(1)}=Q\neq Q_i$, thus $(Q_i)^{[1]}=Q$, and hence $Q_i\in[\operatorname{ch}](Q)$. 

On the other hand, each $Q''\in[\operatorname{ch}](Q)$ must be contained in some $Q_i\in\operatorname{ch}(Q,k_0)$. Since $Q_i\subsetneq Q$, it cannot be that $Q''\subsetneq Q_i$, since this would contradict the property that $Q=(Q'')^{[1]}$ is the dyadic cube strictly containing $Q''$. Thus $Q''=Q_i\in\operatorname{ch}(Q)$.

Altogether, we have checked that $\{Q_1,Q_2\}\subseteq[\operatorname{ch}](Q)=\operatorname{ch}(Q,k_0)$. Recalling that $\operatorname{ch}(Q,k_0)\leq M$, both bounds in \eqref{eq:strChQ>1} follow.
\end{proof}

\begin{lemma}\label{lem:strRD}
There are constants $1<c\leq C<\infty$ such that for all $Q\in\mathscr D\setminus\{X\}$,
\begin{equation}\label{eq:strRD}
   c \mu(Q)\leq \mu(Q^{[1]})\leq C\mu(Q).
\end{equation}
\end{lemma}

\begin{proof}
The upper bound is immediate from doubling. Let us then pick some $Q'\in[\operatorname{ch}](Q^{[1]})\setminus\{Q\}$, as we may by \eqref{eq:strChQ>1}. 
Using this upper bound with $Q'$ in place of $Q$, observing that $(Q')^{[1]}=Q^{[1]}$, we find that
\begin{equation*}
  \mu(Q^{[1]})
  \geq\mu(Q)+\mu(Q')
  \geq\mu(Q)+\frac{1}{C}\mu((Q')^{[1]})  
  \geq\mu(Q)+\frac{1}{C}\mu(Q^{[1]}).
\end{equation*}
Rearranging terms gives the claimed lower bound with $c=\frac{C}{C-1}>1$.
\end{proof}

The lower bound in \eqref{eq:strRD} has the following useful consequence: For every $\alpha>0$,
\begin{equation}\label{eq:strRDsum}
  \sum_{k=0}^\infty\frac{1}{\mu(Q^{[k]})^\alpha}
  \leq\sum_{k=0}^\infty \frac{1}{c^{k\alpha}\mu(Q)^\alpha}
  =\frac{c^\alpha}{c^\alpha-1}\frac{1}{\mu(Q)^\alpha}.
\end{equation}
Here and in what follows, we make the implicit convention that, if $Q^{[k]}$ is undefined some $k$ (i.e., $Q^{[j]}=X$ for some $j<k$), then we simply omit the corresponding terms in a sum like the one above.

\begin{remark}\label{rem:Haar}
Corresponding to a system of dyadic cubes, it is easy to construct a related orthonormal system of Haar wavelets.
Since it naturally fits with the concepts just introduced, we do this here, although it will only be needed in Part~\ref{part:dyadic}.

For each $Q\in\mathscr D$, we have the non-cancellative wavelet
\begin{equation*}
  h^0_Q:=\mu(Q)^{-1/2} 1_Q.
\end{equation*}
A possible way of constructing $M_Q-1$ cancellative wavelets is as follows. (Note that there are none if $M_Q=1$.) Fix some enumeration $\{Q_i\}_{i=1}^{M_Q}$ of $\operatorname{ch}(Q)$. Let
\begin{equation*}
  \hat Q_i:=\bigcup_{j=i}^{M_Q} Q_j
\end{equation*}
and
\begin{equation*}
  h_Q^i:=\sqrt{\frac{\mu(Q_i)\mu(\hat Q_{i+1})}{\mu(\hat Q_i)}}\Big(\frac{1_{Q_i}}{\mu(Q_i)}-\frac{1_{\hat Q_{i+1}}}{\mu(\hat Q_{i+1})}\Big),\qquad
  i=1,\ldots,M_Q-1.
\end{equation*}
Then it is easy to check that
\begin{enumerate}
  \item $\Norm{h_Q^i}{L^2(\mu)}=1$ for $i=0,\ldots,M_Q-1$;
  \item $h_Q^i$ is constant on the support of $h_Q^j$ for $0\leq i<j\leq M_Q-1$;
  \item $\int h_Q^j\ud\mu=0$ for $j=1,\ldots,M_Q-1$;
  \item $\{h_Q^i\}_{i=0}^{M_Q-1}$ is an orthonormal basis of
\begin{equation*}
  L^2(\operatorname{ch}(Q)):=\{f\in L^2(\mu):f\text{ is constant on each }Q'\in\operatorname{ch}(Q), 1_{Q^c}f=0\};
\end{equation*}
  \item $\{h_Q^i\}_{i=1}^{M_Q-1}$ is an orthonormal basis of
\begin{equation*}
  L^2_0(\operatorname{ch}(Q)):=\Big\{f\in L^2(\operatorname{ch}(Q)):\int f\ud\mu=0\Big\}.
\end{equation*}  
\end{enumerate}
Let
\begin{equation*}
  \E_Qf:=\ave{f}_Q 1_Q=\pair{f}{h_Q^0}h_Q^0
\end{equation*}
and let $\D_Q$ be the orthogonal projection of $L^2(\mu)$ onto $L^2_0(\operatorname{ch}(Q))$. Then
\begin{equation*}
  \D_Qf=\Big(\sum_{Q'\in\operatorname{ch}(Q)}\E_{Q'}-\E_Q\Big)=\sum_{i=1}^{M_Q-1}\pair{f}{h_Q^i}h_Q^i.
\end{equation*}
\end{remark}

\section{The Rochberg--Semmes method}\label{sec:RS}

The essence of this section goes back to \cite{RS:BMO}, where the results and proofs for $X=\R$ are expressed in a complex-variable notation of the upper half-place $\C_+=\{x+iy:x\in\R,y>0\}$. The results are restated and partially reproved in $\R^n$ in \cite{RS:NWO}. Once the dyadic cubes in a doubling metric space are available, adapting these results to a general $X$ is straightforward. Since the Euclidean version is already a little scattered over the two papers \cite{RS:BMO,RS:NWO}, we include a treatment here for the sake of completeness and easy reference.

Slightly adapting the ideas of \cite{RS:NWO}, for a sequence $\mathcal E=(e_Q,h_Q)_{Q\in\mathscr D}$ of pairs of functions, we consider the bi-sublinear maximal operator
\begin{equation}\label{eq:ME}
   M_{\mathcal E}(f,g)(x):=\sup_{Q\in\mathscr D}1_Q(x)\frac{\abs{\pair{f}{e_Q}\pair{g}{h_Q}}}{\mu(Q)}.
\end{equation}
Considering this one bi-sublinear operator, instead of the two operators
\begin{equation*}
  M_{\mathcal E}^1(f)(x):=\sup_{Q\in\mathscr D}1_Q(x)\frac{\abs{\pair{f}{e_Q}}}{\mu(Q)^{\frac12}},\quad
  M_{\mathcal E}^2(g)(x):=\sup_{Q\in\mathscr D}1_Q(x)\frac{\abs{\pair{g}{h_Q}}}{\mu(Q)^{\frac12}}
\end{equation*}
featuring in \cite{RS:NWO} and several subsequent works, is a minor technical difference, in principle, but it also brings a certain conceptual advantage: with $M_{\mathcal E}$, we only need to take care of the scaling of one object instead of two, and the scaling in \eqref{eq:ME}, emphasising the averaging nature of $M_{\mathcal E}$, is perhaps more natural than the scalings of the $M_{\mathcal E}^i$. This is even more apparent in comparison to the $L^p$ variants in \cite[(7.38)--(7.39)]{RS:NWO}, where a different $L^{p}$-adapted scaling is used for $M_{\mathcal E}^1$ and for $M_{\mathcal E}^2$, but they cancel out to produce the same $M_{\mathcal E}$ independent of $p$. Although we will not develop this direction here, it seems worth recording for possible future extensions. The use of $M_{\mathcal E}$ will also simplify the weighted extension of these considerations that we take up in the following Section \ref{sec:RSweighted}.

For a sequence $\lambda=(\lambda_Q)_{Q\in\mathscr D}$ of numbers, we define the new sequence $\operatorname{Car}\lambda$ (denoted by $M\lambda$ in \cite{RS:BMO,RS:NWO}) by
\begin{equation*}
  \operatorname{Car}\lambda(P):=\frac{1}{\mu(P)}\sum_{Q\subseteq P}\abs{\lambda_Q}\mu(Q).
\end{equation*}
We say that $\lambda$ is a Carleson sequence if $\operatorname{Car}\lambda\in\ell^\infty(\mathscr D)$.

\begin{proposition}\label{prop:NWO}
For sequences $\mathcal E=(e_Q,h_Q)_{Q\in\mathscr D}$ and $\lambda=(\lambda_Q)_{Q\in\mathscr D}$, we have the estimates
\begin{equation*}
\begin{split}
  \BNorm{\sum_{Q\in\mathscr D}\lambda_Q e_Q\otimes h_Q }{\bddlin(L^2(\mu))}
  &\leq\Norm{M_{\mathcal E}}{L^2(\mu)\times L^2(\mu)\to L^1(\mu)}\Norm{\operatorname{Car}\lambda}{\infty}, \\
  a_n\Big(\sum_{Q\in\mathscr D}\lambda_Q e_Q \otimes h_Q \Big)
  &\leq\Norm{M_{\mathcal E}}{L^2(\mu)\times L^2(\mu)\to L^1(\mu)}(\operatorname{Car}\lambda)^*(n),
\end{split}
\end{equation*}
where
\begin{equation*}
  a_n(T):=\inf\Big\{\Norm{T-F}{\bddlin(L^2(\mu))}:F\in\bddlin(L^2(\mu)),\ \operatorname{rank}(F)<n\Big\}
\end{equation*}
is the $n$th approximation number of $T\in\bddlin(L^2(\mu))$ and
 $(\operatorname{Car}\lambda)^*$ is the decreasing rearrangement of $\operatorname{Car}\lambda$.
\end{proposition}

\begin{proof}
For two functions $f,g\in L^2(\mu)$ and a number $t>0$, let
\begin{equation*}
   \mathscr Q_t:=\Big\{Q\in\mathscr D:  \frac{ \abs{\pair{f}{e_Q}\pair{g}{h_Q} } }{\mu(Q)} > t\Big\},\ 
   \mathscr Q_t^*:=\Big\{Q\in\mathscr Q_t\text{ is maximal in }\mathscr Q_t\Big\}.
\end{equation*}
Letting the operator act on one of these functions, and dualising with the other one, we are led to estimate
\begin{equation}\label{eq:NWOestStart}
\begin{split}
  \sum_{Q\in\mathscr D}&\abs{\lambda_Q}\abs{\pair{f}{e_Q}\pair{h_Q}{g}}
  =\sum_{Q\in\mathscr D}\abs{\lambda_Q}\mu(Q)\int_0^{\frac{\abs{\pair{f}{e_Q}\pair{h_Q}{g}}}{\mu(Q)}}\ud t \\
  &=\int_0^\infty \sum_{Q\in\mathscr Q_t}\abs{\lambda_Q}\mu(Q)\ud t 
  \leq\int_0^\infty \sum_{P\in\mathscr Q_t^*}\sum_{Q\subseteq P}\abs{\lambda_Q}\mu(Q)\ud t \\
  &=\int_0^\infty \sum_{P\in\mathscr Q_t^*}\operatorname{Car}\lambda(P)\mu(P)\ud t  
  \leq\Norm{\operatorname{Car}\lambda}{\infty}\int_0^\infty \mu\Big(\bigcup_{P\in\mathscr Q_t^*}P\Big)\ud t \\
  &\leq\Norm{\operatorname{Car}\lambda}{\infty}\int_0^\infty \mu(\{M_{\mathcal E}(f,g)>t\})\ud t  \\
  &=\Norm{\operatorname{Car}\lambda}{\infty}\Norm{M_{\mathcal E}(f,g)}{L^1(\mu)}  \\
  &\leq\Norm{\operatorname{Car}\lambda}{\infty}\Norm{M_{\mathcal E}}{L^2(\mu)\times L^2(\mu)\to L^1(\mu)} 
  \Norm{f}{L^2(\mu)}\Norm{g}{L^2(\mu)}.
\end{split}
\end{equation}
This proves the first claim.

Let $(P_n)_{n=1}^\infty$ be an enumeration of $\mathscr D$ such that $\operatorname{Car}\lambda(P_n)=(\operatorname{Car}\lambda)^*(n)$. Then
\begin{equation*}
  \sum_{i=1}^{n-1}\lambda_{P_i} e_{P_i}\otimes f_{P_i}
\end{equation*}
has rank less than $n$. Denoting $\mathscr P_n:=\mathscr D\setminus\{P_1,\ldots,P_{n-1}\}$, it follows that
\begin{equation*}
  a_n\Big(\sum_{Q\in\mathscr D}\lambda_Q e_Q\otimes f_Q\Big)
  \leq\BNorm{\sum_{Q\in\mathscr P_{n}}\lambda_Q e_Q\otimes f_Q}{} 
  \leq\Norm{M_{\mathcal E}}{}
  \Norm{\operatorname{Car}(1_{\mathscr P_{n}}\lambda)}{\infty}
\end{equation*}
by an application of the first part of the lemma to the truncated sequence $(1_{\mathscr P_{n}}(Q)\lambda_Q)_{Q\in\mathscr D}$ in place of $\lambda$. Let $P\in\mathscr D$, and let $\mathscr P_n^*(P)$ be the collection of maximal $R\in\mathscr P_n$ contained in $P$. Then
\begin{equation*}
\begin{split}
  \mu(P)\operatorname{Car}(1_{\mathscr P_n}\lambda)(P)
  &=\sum_{Q\subseteq P}\mu(Q)\abs{\lambda_Q}1_{\mathscr P_n}(Q)
  \leq\sum_{R\in\mathscr P_n^*(P)}\sum_{Q\subseteq R}\mu(Q)\abs{\lambda_Q} \\
  &=\sum_{R\in\mathscr P_n^*(P)}\mu(R)(\operatorname{Car}\lambda)(R) \\
  &\leq\Norm{1_{\mathscr P_n}\operatorname{Car}\lambda}{\infty}\sum_{R\in\mathscr P_n^*(P)}\mu(R)
  \leq (\operatorname{Car}\lambda)^*(n)\mu(P).
\end{split}
\end{equation*}
Hence
\begin{equation*}
  a_n\Big(\sum_{Q\in\mathscr D}\lambda_Q e_Q\otimes f_Q\Big)
  \leq\Norm{M_{\mathcal E}}{}
  \Norm{\operatorname{Car}(1_{\mathscr P_{n}}\lambda)}{\infty}
  \leq \Norm{M_{\mathcal E}}{}
  (\operatorname{Car}\lambda)^*(n),
\end{equation*}
which is the second claim of the lemma.
\end{proof}

\begin{proposition}\label{prop:CarBd}
For all $p\in(0,\infty)$ and $q\in(0,\infty]$, the operator $\operatorname{Car}$ is bounded on $\ell^{p,q}(\mathscr D)$.
\end{proposition}

\begin{proof}
It suffices to consider $q=p\in(0,\infty)$, since the rest then follows by interpolation of Lorentz spaces,
\begin{equation*}
   \ell^{p,q}=(\ell^{p_0},\ell^{p_1})_{\theta,q},\qquad\frac{1}{p}=\frac{1-\theta}{p_0}+\frac{\theta}{p_1},\qquad p_0\neq p_1,\qquad 0<\theta<1;
\end{equation*}
see \cite[Theorem 5.3.1]{BL:book}.

If $p\in(0,1]$, then
\begin{equation*}
\begin{split}
  \sum_{P\in\mathscr D}(\operatorname{Car}\lambda)^p(P)
  &\leq\sum_{P\in\mathscr D}\sum_{Q\subseteq P} \Big( \frac{\mu(Q)}{\mu(P))}\Big)^p\abs{\lambda_Q}^p \\
  &=\sum_{Q\in\mathscr D}\abs{\lambda_Q}^p\sum_{P\supseteq Q}\Big( \frac{\mu(Q)}{\mu(P))}\Big)^p
  =\sum_{Q\in\mathscr D}\abs{\lambda_Q}^p\sum_{k=0}^\infty\Big( \frac{\mu(Q)}{\mu(Q^{[k]})}\Big)^p,
\end{split}
\end{equation*}
where $Q^{[k]}$ is the $k$th strict dyadic ancestor. From \eqref{eq:strRDsum}, it then follows that
\begin{equation*}
  \sum_{P\in\mathscr D}(\operatorname{Car}\lambda)^p(P)
  \leq\sum_{Q\in\mathscr D}\abs{\lambda_Q}^p\frac{c^p}{c^p-1}
\end{equation*}
and hence
\begin{equation*}
  \Norm{\operatorname{Car}\lambda}{\ell^p}\leq\frac{c}{(c^p-1)^{\frac1p}}\Norm{\lambda}{\ell^p}.
\end{equation*}

If $p\in(1,\infty)$, then for a positive sequence $(a_k)_{k=0}^\infty$ to be chosen,
\begin{equation*}
\begin{split}
   \sum_{P\in\mathscr D}(\operatorname{Car}\lambda)^p(P)
   &=\sum_{P\in\mathscr D}\Big(\sum_{k=0}^\infty \frac{a_k}{a_k}
   \sum_{Q\in\operatorname{ch}^k(P)}\frac{\mu(Q)}{\mu(P)}\abs{\lambda_Q}\Big)^p \\
   &\leq\sum_{P\in\mathscr D}\Big(\sum_{i=0}^\infty a_i^{p'}\Big)^{p-1}
      \sum_{k=0}^\infty a_k^{-p}\Big(\sum_{Q\in\operatorname{ch}^k(P)}\frac{\mu(Q)}{\mu(P)}\abs{\lambda_Q}\Big)^p \\
   &\leq \Big(\sum_{i=0}^\infty a_i^{p'}\Big)^{p-1}\sum_{P\in\mathscr D}
      \sum_{k=0}^\infty a_k^{-p}\sum_{Q\in\operatorname{ch}^k(P)}\frac{\mu(Q)}{\mu(P)}\abs{\lambda_Q}^p \\
   &= \Big(\sum_{i=0}^\infty a_i^{p'}\Big)^{p-1}\sum_{Q\in\mathscr D}\abs{\lambda_Q}^p
      \sum_{k=0}^\infty a_k^{-p}\frac{\mu(Q)}{\mu(Q^{[k]})}.
\end{split}
\end{equation*}
By Lemma \ref{lem:strRD}, $\mu(Q^{[k]})/\mu(Q)\geq c^k$ for some $c>1$. We then choose $a_k$ so that $a_k^{p'}=a_k^{-p}c^{-k}$, i.e., $a_k=c^{-kp/(p+p')}=c^{-k/p'}$. Then
\begin{equation*}
  \sum_{k=0}^\infty a_k^{p'}=\sum_{k=0}^\infty a_k^{-p}c^{-k}=\sum_{k=0}^\infty c^{-k}=\frac{c}{c-1},
\end{equation*}
and we obtain
\begin{equation*}
  \Norm{\operatorname{Car}\lambda}{\ell^p}\leq\frac{c}{c-1}\Norm{\lambda}{\ell^p}.
\end{equation*}
This completes the proof.
\end{proof}

\begin{remark}
The proof of Proposition \ref{prop:CarBd} gives the explicit estimate
\begin{equation*}
  \Norm{\operatorname{Car}\lambda}{\ell^p}\leq\frac{c}{(c^{\hat p}-1)^{\frac{1}{\hat p}}}\Norm{\lambda}{\ell^p},\qquad\hat p:=\min(p,1),
\end{equation*}
where $c>1$ is the constant from Lemma \ref{lem:strRD}. For $X=\R^d$ with the usual dyadic cubes and the Lebesgue measure, $c=2^d$.
\end{remark}

\begin{corollary}\label{cor:Spq<ellpq}
Let $p\in(0,\infty)$ and $q\in(0,\infty]$.
For sequences $\mathcal E=(e_Q,h_Q)_{Q\in\mathscr D}$ and $\lambda=(\lambda_Q)_{Q\in\mathscr D}$, we have the estimate
\begin{equation*}
  \BNorm{\sum_{Q\in\mathscr D}\lambda_Q e_Q\otimes h_Q }{S^{p,q}(L^2(\mu))}
  \lesssim\Norm{M_{\mathcal E}}{L^2(\mu)\times L^2(\mu)\to L^1(\mu)}\Norm{\lambda}{\ell^{p,q}}.
\end{equation*}
\end{corollary}

\begin{proof}
By the definition of the norm of $S^{p,q}$ in terms of the singular values, and the similar definition of the norm of $\ell^{p,q}$ in terms of the decreasing rearrangement, the estimate
\begin{equation*}
    \BNorm{\sum_{Q\in\mathscr D}\lambda_Q e_Q\otimes h_Q }{S^{p,q}(L^2(\mu))}
  \lesssim\Norm{M_{\mathcal E}}{L^2(\mu)\times L^2(\mu)\to L^1(\mu)}\Norm{\operatorname{Car}\lambda}{\ell^{p,q}}
\end{equation*}
is immediate from Proposition \ref{prop:NWO}. The claim of the corollary then follows from the boundedness of $\operatorname{Car}$ guaranteed by Proposition \ref{prop:CarBd}.
\end{proof}

\begin{corollary}\label{cor:ellpq<Spq}
Let $p\in(1,\infty)$ and $q\in[1,\infty]$. Let $A\in S^{p,q}(L^2(\mu))$.
Then for any sequence $\mathcal E=(e_Q,h_Q)_{Q\in\mathscr D}$, we have the estimate
\begin{equation*}
  \Norm{\{\pair{Ae_Q}{h_Q}\}_{Q\in\mathscr D}}{\ell^{p,q}(\mathscr D)}
  \lesssim\Norm{M_{\mathcal E}}{L^2(\mu)\times L^2(\mu)\to L^1(\mu)}\Norm{A}{S^{p,q}(L^2(\mu))}.
\end{equation*}
\end{corollary}

\begin{proof}
For any finitely non-zero sequence $\lambda=(\lambda_Q)_{Q\in\mathscr D}$, we have
\begin{equation}\label{eq:ellpq<Spq}
\begin{split}
  \Babs{\sum_{Q\in\mathscr D}\lambda_Q\pair{Ae_Q}{h_Q}}
  &=\Babs{\operatorname{trace}\Big(A\sum_{Q\in\mathscr D}\lambda_Q e_Q\otimes h_Q\Big)} \\
  &\leq\Norm{A}{S^{p,q}}\BNorm{\sum_{Q\in\mathscr D}\lambda_Q e_Q\otimes h_Q}{S^{p',q'}} \\
  &\lesssim\Norm{A}{S^{p,q}}\Norm{M_{\mathcal E}}{L^2\times L^2\to L^1}\Norm{\lambda}{\ell^{p',q'}(\mathscr D)}
\end{split}
\end{equation}
by the duality of the $S^{p,q}$ spaces in the second step and Corollary \ref{cor:Spq<ellpq} in the last one. Taking the supremum over all such sequences with $\Norm{\lambda}{\ell^{p',q'}(\mathscr D)}\leq 1$, and using the duality of the $\ell^{p,q}$ spaces, the claim follows.
\end{proof}

\begin{remark}\label{rem:NWOsuff}
The boundedness of $M_{\mathcal E}:L^2(\mu)\times L^2(\mu)\to L^1(\mu)$ follows in particular if both maximal operators
\begin{equation*}
  M_{\mathcal E}^1 f:= \sup_{Q\in\mathscr D}1_Q\frac{\abs{\pair{f}{e_Q}}}{\mu(Q)^{\frac12}},\qquad
  M_{\mathcal E}^2 g \sup_{Q\in\mathscr D}1_Q\frac{\abs{\pair{g}{h_Q}}}{\mu(Q)^{\frac12}}
\end{equation*}
are bounded on $L^2(\mu)$. When $M_{\mathcal E}^1$ is bounded, the sequence $(e_Q)_{Q\in\mathscr D}$ is called {\em nearly weakly orthogonal} (NWO) in \cite{RS:NWO}. (A seemingly different definition of NWO already appears in \cite{RS:BMO}; their equivalence is proved in \cite{RS:NWO}.) If
\begin{equation}\label{eq:NWOsuffBasic}
   \abs{e_Q}\lesssim \mu(Q)^{-\frac12}1_{B_Q^*}, 
\end{equation}
where $B_Q^*=B(x_Q,c\ell(Q))$ for any fixed constant $c$ independent of $Q$,
then $M_{\mathcal E}^1f\lesssim M f$, where $M$ is the Hardy--Littlewood maximal operator. More generally, if $e_Q$ is supported on $B_Q^*$ and $\Norm{e_Q}{r}\lesssim\mu(Q)^{\frac1r-\frac12}$ for some $r>2$, then
\begin{equation*}
  \frac{\abs{\pair{f}{e_Q}}}{\mu(Q)^{\frac12}}
  \leq\frac{\Norm{1_{B_Q^*}f}{r'}\Norm{e_Q}{r}}{\mu(Q)^{\frac12}}
  \lesssim\Norm{1_{B_Q^*}f}{r'}\mu(Q)^{\frac1r-1}=\Big(\fint_{B_Q^*}\abs{f}^{r'}\Big)^{\frac{1}{r'}},
\end{equation*}
and
\begin{equation*}
  M_{\mathcal E}^1 f\lesssim (M\abs{f}^{r'})^{\frac{1}{r'}}
\end{equation*}
is dominated by the rescaled dyadic maximal operator, which is still bounded on $L^2(\mu)$ for $r'<2$.
\end{remark}

\section{A weighted extension of the Rochberg--Semmes method}\label{sec:RSweighted}

The first extension of the method of nearly weakly orthogonal sequences of \cite{RS:NWO} to weighted spaces $L^2(w)$ on $\R^n$ was already given in the same paper \cite[Proposition 7.25]{RS:NWO}, and it has been recently revisited in \cite{GLW:23}. In this section, we present a version of this extension tailored for our needs in spaces of homogeneous type, and making use of the bi-sublinear maximal operator \eqref{eq:ME}.

We consider the weighted spaces
\begin{equation*}
\begin{split}
    L^2(w)=L^2(w\ud\mu):=\Big\{ & f\text{ measurable function on }X: \\
    &\Norm{f}{L^2(w)}:=\Big(\int_X\abs{f(x)}^2 w(x)\ud\mu(x)\Big)^{\frac12}<\infty\Big\},
\end{split}
\end{equation*}
where the weight $w$ is a measurable function with $w(x)\in(0,\infty)$ at $\mu$-almost every $x\in X$. With respect to the unweighted duality
\begin{equation*}
   \pair{f}{g}:=\int_X f(x)g(x)\ud\mu(x),
\end{equation*}
the dual of $L^2(w)$ is another weighted space $L^2(\sigma)$ with $\sigma(x):=w^{-1}(x):=1/w(x)$ (the reciprocal, not the inverse function).

As a weighted version of Proposition \ref{prop:NWO}, we have the following, where all that we need to know about the weight $w$ is neatly encoded in the boundedness properties of the maximal operator $M_{\mathcal E}$.

\begin{proposition}\label{prop:NWOw}
For sequences $\mathcal E=(e_Q,h_Q)_{Q\in\mathscr D}$ and $\lambda=(\lambda_Q)_{Q\in\mathscr D}$, we have the estimates
\begin{equation*}
\begin{split}
  \BNorm{\sum_{Q\in\mathscr D}\lambda_Q e_Q\otimes h_Q }{\bddlin(L^2(w))}
  &\leq\Norm{M_{\mathcal E}}{L^2(w)\times L^2(\sigma)\to L^1(\mu)}\Norm{\operatorname{Car}\lambda}{\infty}, \\
  a_n\Big(\sum_{Q\in\mathscr D}\lambda_Q e_Q \otimes h_Q\Big)_{\bddlin(L^2(w))}
  &\leq\Norm{M_{\mathcal E}}{L^2(w)\times L^2(\sigma)\to L^1(\mu)}(\operatorname{Car}\lambda)^*(n),
\end{split}
\end{equation*}
where $w$ is a weight, $\sigma=w^{-1}$,
\begin{equation*}
  a_n(T)_{\bddlin(L^2(w))}:=\inf\Big\{\Norm{T-F}{\bddlin(L^2(w))}:F\in\bddlin(L^2(w)),\ \operatorname{rank}(F)<n\Big\}
\end{equation*}
is the $n$th approximation number of $T\in\bddlin(L^2(w))$ and
 $(\operatorname{Car}\lambda)^*$ is the decreasing rearrangement of $\operatorname{Car}\lambda$.
\end{proposition}

\begin{proof}
For $f\in L^2(w)$ and $g\in L^2(\sigma)$, we have the same sequence of estimates \eqref{eq:NWOestStart} as in the proof of Proposition \ref{prop:NWO}, except in the last step we use the weighted boundedness $M_{\mathcal E}:L^2(w)\times L^2(\sigma)\to L^1(\mu)$. This gives
\begin{equation*}
  \sum_{Q\in\mathscr D}\abs{\lambda_Q}\abs{\pair{f}{e_Q}\pair{h_Q}{g}}
  \leq\Norm{\operatorname{Car}\lambda}{\infty}\Norm{M_{\mathcal E}}{L^2(w)\times L^2(\sigma)\to L^1(\mu)}
  \Norm{f}{L^2(w)}\Norm{g}{L^2(\sigma)}.
\end{equation*}
The left-hand side of the previous bound dominates the action of the operator $\sum_{Q\in\mathscr D}\lambda_Q e_Q\otimes h_Q$ on $f\in L^2(w)$, paired with $g\in L^2(\sigma)$. Taking the supremum over the unit balls of these spaces, and recalling that $L^2(\sigma)$ is the dual of $L^2(w)$, we obtain the first estimate of the proposition. Using this first estimate, the proof of the second estimate is obtained by repeating the proof of Proposition \ref{prop:NWO} verbatim.
\end{proof}

\begin{corollary}\label{cor:Spq<ellpq-w}
Let $p\in(0,\infty)$ and $q\in(0,\infty]$.
For sequences $\mathcal E=(e_Q,h_Q)_{Q\in\mathscr D}$ and $\lambda=(\lambda_Q)_{Q\in\mathscr D}$, and all weights $w$ and $\sigma=w^{-1}$, we have the estimate
\begin{equation*}
  \BNorm{\sum_{Q\in\mathscr D}\lambda_Q e_Q\otimes h_Q }{S^{p,q}(L^2(w))}
  \lesssim\Norm{M_{\mathcal E}}{L^2(w)\times L^2(\sigma)\to L^1(\mu)}\Norm{\lambda}{\ell^{p,q}}.
\end{equation*}
\end{corollary}

\begin{proof}
Just like in proof of the unweighted Corollary \ref{cor:Spq<ellpq},
by the definition of the norm of $S^{p,q}(L^2(w))$ in terms of the singular values $a_n(\ )_{\bddlin(L^2(w))}$, and the similar definition of the norm of $\ell^{p,q}$ in terms of the decreasing rearrangement, the estimate
\begin{equation*}
    \BNorm{\sum_{Q\in\mathscr D}\lambda_Q e_Q\otimes h_Q }{S^{p,q}(L^2(\mu))}
  \lesssim\Norm{M_{\mathcal E}}{L^2(w)\times L^2(\sigma)\to L^1(\mu)}\Norm{\operatorname{Car}\lambda}{\ell^{p,q}}
\end{equation*}
is immediate from Proposition \ref{prop:NWOw}. The claim of the corollary then follows from the boundedness of $\operatorname{Car}$ guaranteed by Proposition \ref{prop:CarBd}.
\end{proof}

\begin{corollary}\label{cor:ellpq<Spq-w}
Let $p\in(1,\infty)$ and $q\in[1,\infty]$. Let $w$ and $\sigma=w^{-1}$ be weights and $A\in S^{p,q}(L^2(w))$.
Then for any sequence $\mathcal E=(e_Q,h_Q)_{Q\in\mathscr D}$, we have the estimate
\begin{equation*}
  \Norm{\{\pair{Ae_Q}{h_Q}\}_{Q\in\mathscr D}}{\ell^{p,q}(\mathscr D)}
  \lesssim\Norm{M_{\mathcal E}}{L^2(w)\times L^2(\sigma)\to L^1(\mu)}\Norm{A}{S^{p,q}(L^2(w))}.
\end{equation*}
\end{corollary}

\begin{proof}
This follows by repeating the proof of Corollary \ref{cor:ellpq<Spq} but using the weighted Corollary \ref{cor:Spq<ellpq-w} in the last step of \eqref{eq:ellpq<Spq}. For any finitely non-zero sequence $\lambda=(\lambda_Q)_{Q\in\mathscr D}$, this gives
\begin{equation*}
  \Babs{\sum_{Q\in\mathscr D}\lambda_Q\pair{Ae_Q}{h_Q}}
  \lesssim\Norm{A}{S^{p,q}(L^2(w))}\Norm{M_{\mathcal E}}{L^2(w)\times L^2(\sigma)\to L^1(\mu)}\Norm{\lambda}{\ell^{p',q'}(\mathscr D)}
\end{equation*}
Taking the supremum over all such sequences with $\Norm{\lambda}{\ell^{p',q'}(\mathscr D)}\leq 1$, and using the duality of the $\ell^{p,q}$ spaces, the claim follows.
\end{proof}

\begin{remark}\label{rem:NWOsuff-w}
The boundedness of $M_{\mathcal E}:L^2(w)\times L^2(\sigma)\to L^1(\mu)$ follows in particular if the maximal operators
\begin{equation}\label{eq:ME1}
  M_{\mathcal E}^1 f:= \sup_{Q\in\mathscr D}1_Q\frac{\abs{\pair{f}{e_Q}}}{\mu(Q)^{\frac12}},\qquad
  M_{\mathcal E}^2 g \sup_{Q\in\mathscr D}1_Q\frac{\abs{\pair{g}{h_Q}}}{\mu(Q)^{\frac12}}
\end{equation}
are bounded on $L^2(w)$ and $L^2(\sigma)$, respectively. Under the basic estimate \eqref{eq:NWOsuffBasic}, we have $M_{\mathcal E}^1f\lesssim Mf$, where $M$ is the Hardy--Littlewood maximal operator, which is known to be bounded on $L^2(w)$ if and only if $w$ belongs to the Muckenhoupt class $A_2$ defined by the finiteness of
\begin{equation}\label{eq:A2dyad}
  [w]_{A_2}:=\sup_{B \text{ ball}}\fint_B w\ud\mu\fint_B \sigma\ud\mu=[\sigma]_{A_2}.
\end{equation}
\end{remark}

More generally, we have the following:

\begin{lemma}\label{lem:NWOsuff-w}
Let $(X,\rho,\mu)$ be a space of homogeneous type and $w\in A_2$. Let $(e_Q)_{\mathscr D}$ is a sequence of functions that satisfies
\begin{equation}\label{eq:NWOrEst}
  \supp e_Q\subseteq B_Q^*,\qquad \Norm{e_Q}{r}\lesssim\mu(Q)^{\frac1r-\frac12}
\end{equation}
for every $r\in(2,\infty)$ (with implied constant allowed to depend on $r$),
where $B_Q^*=B(z_Q,c\ell(Q))$ for some fixed constant independent of $Q$.

Then $M_{\mathcal E}^1$ as in \eqref{eq:ME1} is bounded on $L^2(w)$. If $(h_Q)_{Q\in\mathscr D}$ satisfies the same properties, then $M_{\mathcal E}$ as in \eqref{eq:ME} is bounded from $L^2(w)\times L^2(\sigma)$ to $L^1(\mu)$.
\end{lemma}

\begin{proof}
The computation in Remark \ref{rem:NWOsuff} shows that
\begin{equation}\label{eq:ME1vsMr}
   M_{\mathcal E}^1 f\lesssim (M\abs{f}^{r'})^{\frac{1}{r'}}
\end{equation}
under the assumptions that we made. Note that
 \begin{equation}\label{eq:MrwBd}
  \Norm{(M\abs{f}^{r'})^{\frac{1}{r'}}}{L^2(w)}
  =\Norm{M\abs{f}^{r'}}{L^{\frac{2}{r'}}(w)}^{\frac{1}{r'}}
  \lesssim \Norm{\abs{f}^{r'}}{L^{\frac{2}{r'}}(w)}^{\frac{1}{r'}}=\Norm{f}{L^2(w)}
\end{equation}
by the boundedness of the maximal operator on $L^{\frac{2}{r'}}(w)$, which holds {\em if} $w\in A_{\frac{2}{r'}}$.

By the self-improvement property of Muckenhoupt weights, which is also well known over space of homogeneous type (see \cite[Lemma I.8]{ST:89}), every $w\in A_2$ satisfies $w\in A_{2-\eps}$ for some $\eps>0$, or equivalently, $w\in A_{\frac{2}{r'}}=A_{2(1-\frac1r)}$ for some finite $r$. Since the assumptions allow us to pick $r$ as large as we like, we obtain $w\in A_{\frac{2}{r'}}$ and hence the boundedness \eqref{eq:MrwBd}. Together with \eqref{eq:ME1vsMr}, this gives the claimed boundedness of $M_{\mathcal E}^1$ on $L^2(w)$.

Recall from \eqref{eq:A2dyad} that $w\in A_2$ if and only if $\sigma=w^{-1}\in A_2$. Hence, if $(h_Q)_{Q\in\mathscr D}$, we also obtain the boundedness of $M_{\mathcal E}^2$ on $L^2(\sigma)$. Since
\begin{equation*}
\begin{split}
    \Norm{M_{\mathcal E}(f,g)}{L^1(\mu)} &\leq \Norm{M_{\mathcal E}^1(f)M_{\mathcal E}^2(g)}{L^1(\mu)} \\
    &\leq \Norm{M_{\mathcal E}^1(f)}{L^2(w)}\Norm{M_{\mathcal E}^2(g)}{L^1(\sigma)},
\end{split}
\end{equation*}
the claimed boundedness of $M_{\mathcal E}:L^2(w)\to L^2(\sigma)\to L^1(\mu)$ follows.
\end{proof}

 \begin{remark}
 In \cite[Proposition 7.25]{RS:NWO}, the condition $w\in A_2$ is implicit in the statement, but explicit in the proof. In \cite{GLW:23}, this assumption is explicit throughout.
 \end{remark}

\section{The lower bound for the commutators}\label{sec:osc<com}

In this section, we prove that the membership of the commutator $[b,T]$ in the Schatten--Lorentz class $S^{p,q}$ implies an oscillatory norm bound on the symbol $b$. This is an analogue of \cite[Theorem 3.4]{RS:NWO}, and it will be proved with the help their technique of nearly weakly orthogonal sequences as formulated in Corollary \ref{cor:ellpq<Spq-w}. The level of originality of this section is rather modest; it is included for the sake of completeness to show that the existing techniques for the lower bounds of commutators also apply to the setting at hand.

With Corollary \ref{cor:ellpq<Spq-w} at hand, the remaining issue is coming up with sequences $\mathcal E=(e_Q,h_Q)_{Q\in\mathscr D}$ such that $m_b(Q)$ can be dominated by $\pair{[b,T]e_Q}{h_Q}$. In \cite{RS:NWO}, such sequences were produced by local Fourier expansions of the inverse kernel $K(x,y)^{-1}$. In the recent extensions, a variety of alternatives have been employed.
The ``median method'' was introduced for similar lower bounds in the context of boundedness of commutators in \cite[Proposition 3.1]{LOR:19}, and seems to have been first employed for $S^p$ estimates in \cite[Lemma 4.1]{FLL:23} on the Heisenberg group, and then in \cite[Section 2.1]{LXY} on more general Carnot groups. In the Bessel setting \cite[Section 5]{FLLX}, Alpert wavelet expansions were used a bit like Fourier expansions in \cite{RS:NWO}.

A drawback of the median method, until recently, was its restriction to real-valued functions only. A ``complex median method'' was introduced in \cite{WZ:24} and could be most likely applied here as well, but we will take a different route based on the alternative method of ``approximate weak factorisation'', which was introduced on $\R^d$ in \cite{Hyt:commu} and adapted to spaces of homogeneous type in \cite{GLS}. However, we will actually only use this method implicitly, given that we can simply apply, essentially as a black box, a consequence of this method from \cite{GLS}. Before stating the result, let us elaborate a little on our key assumption, which is a minor generalisation of the related assumptions in \cite{GLS,Hyt:commu}.

\begin{definition}\label{def:CZomega}
A function $K:X\times X\setminus\{(x,x):x\in X\}\to\C$ is said to be an $\omega$-Calder\'on--Zygmund kernel if its satisfies \eqref{eq:CZ0} and
\begin{equation}
  \abs{K(x,y)-K(x',y)}+\abs{K(y,x)-K(y,x')}
  \leq\omega\Big(\frac{\rho(x,x')}{\rho(x,y)}\Big)\frac{1}{V(x,y)}  
\end{equation}
for all $x,x',y\in X$ with $\rho(x,x')\ll\rho(x,y)$.
\end{definition}

\begin{definition}\label{def:nondeg}
A Calder\'on--Zygmund kernel is said to be non-degenerate at point $x\in X$ and scale $r>0$ if there exists $y\in X$ with $\rho(x,y)\approx r$ such that 
\begin{equation}\label{eq:nondeg}
  \abs{K(x,y)}+\abs{K(y,x)}\gtrsim\frac{1}{V(x,y)}.
\end{equation}
A Calder\'on--Zygmund kernel is said to be non-degenerate if it is non-degenerate at all points and scales, with implied constants independent of $x\in X$ and $r>0$.
\end{definition}

Under this non-degeneracy condition on the kernel of an $\omega$-Calder\'on--Zygmund operator, we have the following estimate (which essentially a combination of \cite[Proposition 4.3 and 4.12]{GLS}, as explained below) for the oscillations of $b$ in terms of the commutator $[b,T]$. Note that condition \eqref{eq:omega0} below is much weaker than any condition of the form $\omega(t)=t^\eta$ as used in Theorem \ref{thm:main}.

\begin{proposition}\label{prop:GLS}
Let $b\in L^1_{\loc}(\mu)$, let $K$ be an $\omega$-Calder\'on--Zygmund kernel with
\begin{equation}\label{eq:omega0}
  \lim_{t\to 0}\omega(t)=0,
\end{equation}
and let $[b,T]$ be an operator with kernel $(b(x)-b(y))K(x,y)$.
Let $x\in X$ and $r>0$, and 
suppose that $K$ is non-degenerate at point $x\in X$ and scale $Ar$,
where $A$ is a large but fixed constant depending only on the parameters of the space and the kernel $K$, not on the specific $x$ and $r$.
Then there exists a point $y\in X$ such that
the balls $B_1:=B(x,r)$ and $B_2:=B(y,r)$ satisfy $\rho(B_1,B_2)\approx r$
and for all $E_i\subset B_i$ with $\mu(E_i)\approx\mu(B_i)$, there exist functions $f_j,g_j$ such that
\begin{enumerate}[\rm(1)]
  \item $\Norm{f_j}{\infty}+\Norm{g_j}{\infty}\lesssim 1$ for both $j=1,2$;
  \item\label{it:GLSsupp} $\supp f_j\times\supp g_j$ is contained in either $E_1\times E_2$ or $E_2\times E_1$; 
  \item the oscillation of $b$ over $E_1$ has the following estimate:
\begin{equation}\label{eq:oscVsCom}
  \int_{E_1}\abs{b-\ave{b}_{E_1}}\ud\mu\lesssim\sum_{j=1}^2\abs{\pair{[b,T]f_j}{g_j}}.
\end{equation}
\end{enumerate}
In particular, if $K$ is non-degenerate, then these conclusions hold for all $B=B(x,r)\subset X$.
\end{proposition}

\begin{proof}
The assumption \eqref{eq:nondeg} clearly implies that
\begin{equation}\label{eq:nondeg1}
  \abs{K(x,y)}\gtrsim\frac{1}{V(x,y)}
\end{equation}
or
\begin{equation}\label{eq:nondeg2}
  \abs{K(y,x)}\gtrsim\frac{1}{V(x,y)}.
\end{equation}
In \cite[Proposition 4.3 resp. 4.12]{GLS}, the conclusions of Proposition \ref{prop:GLS} are proved under the assumption \eqref{eq:nondeg2} resp. \eqref{eq:nondeg1}. Strictly speaking the said propositions of \cite{GLS}, as stated there, assume that \eqref{eq:nondeg2} resp. \eqref{eq:nondeg1} holds for all $x$ and $r$. However, an inspection of their proofs shows that, 
to get the conclusion for a given ball $B=B(x,r)$, the non-degeneracy assumption is only used for the corresponding
point $x$ and a scale $Ar$, as in the assumptions; essentially, one chooses $A$ large enough so that $\omega(A^{-1})$ is small enough, which uses assumption \eqref{eq:omega0}. In particular, it follows from this observation that, to have the conclusions of Proposition \ref{prop:GLS} for all $x$ and $r$, whether \eqref{eq:nondeg1} or \eqref{eq:nondeg2} holds may vary from one $(x,r)$ to another, and hence it is enough to have the symmetric condition \eqref{eq:nondeg}.
\end{proof}

\begin{proposition}\label{prop:Osc<S}
Let $K$ be a non-degenerate $\omega$-Calder\'on--Zygmund kernel with \eqref{eq:omega0}, and let $b\in L^1_{\loc}(\mu)$.
Let $p\in(1,\infty)$ and $q\in[1,\infty]$.
Suppose that there exists an operator $[b,T]\in S^{p,q}(L^2(\mu))$ with kernel $(b(x)-b(y))K(x,y)$. Then
\begin{equation*}
  \Norm{b}{\operatorname{Osc}^{p,q}}
  :=\Norm{\{m_b(B_Q)\}_{Q\in\mathscr D}}{\ell^{p,q}}
  \lesssim\Norm{[b,T]}{S^{p,q}(L^2(\mu))}.
\end{equation*}
More generally, if $w\in A_2$ and $[b,T]\in S^{p,q}(L^2(w))$, then
\begin{equation*}
  \Norm{b}{\operatorname{Osc}^{p,q}}
  \lesssim\Norm{[b,T]}{S^{p,q}(L^2(w))}.
\end{equation*}
\end{proposition}

\begin{proof}
For each $Q\in\mathscr D$, we apply Proposition \ref{prop:GLS} with $E_1=B_1=B_Q$ to produce functions $f_Q^j$ and $g_Q^j$ such that
\begin{equation}\label{eq:use-GLS}
  m_b(B_Q)
  =\frac{1}{\mu(B_Q)}\int_{B_Q}\abs{b-\ave{b}_Q}\ud\mu
  \lesssim\sum_{j=1}^2\abs{ \pair{[b,T]f_Q^j}{g_Q^j}},
\end{equation}
and
\begin{equation}\label{eq:GLS-NWO}
  \abs{f_Q^j}+\abs{g_Q^j}\lesssim\frac{1_{B_Q^*}}{\mu(Q)^{\frac12}},
\end{equation}
where $B_Q^*= c\cdot B_Q$ is a concentric extension by some factor $c$ that depends only on the implied constants in Proposition \ref{prop:GLS}; thus $B_Q^*$ contains both $E_1=B_1=B_Q$ and $E_2\subseteq B_2$, and there is no need to distinguish between the two cases in Proposition \ref{prop:GLS}\eqref{it:GLSsupp}. Note also that we have absorbed the factor $\mu(B_Q)^{-1}$ into the functions $f_Q^j,g_Q^j$, resulting in a normalisation different from that in Proposition \ref{prop:GLS}, where we estimated the ``plain'' integral $\int_{E_1}$ instead of the average integrals $\fint_{B_Q}$ in the case at hand.

Under the condition \eqref{eq:GLS-NWO}, it is immediate that the maximal operators
\begin{equation}\label{eq:Mjfg}
  \mathcal M^j:(f,g)\mapsto \sup_{Q\in\mathscr D} 1_Q\frac{\abs{\pair{f^j_Q}{f}\pair{g^j_Q}{g}}}{\mu(Q)}
\end{equation}
are bounded $L^2(\mu)\times L^2(\mu)\to L^1(\mu)$. Hence, combining \eqref{eq:use-GLS} with the quasi-triangle inequality in $\ell^{p,q}$ and Corollary \ref{cor:ellpq<Spq}, we obtain
\begin{equation}\label{eq:osc<SpqFinish}
\begin{split}
   &\Norm{\{m_b(B_Q)\}_{Q\in\mathscr D}}{\ell^{p,q}}
  \lesssim\sum_{j=1}^2\Norm{ \{ \pair{[b,T]f_Q^j}{g_Q^j} \}_{Q\in\mathscr D}}{\ell^{p,q}} \\
  &\qquad\lesssim\sum_{j=1}^2\Norm{\mathcal M^j}{L^2(\mu)\times L^2(\mu)\to L^1(\mu)} \Norm{ [b,T]}{S^{p,q}} 
  \lesssim \Norm{ [b,T] }{S^{p,q}}.
\end{split}
\end{equation}
This is the first claimed estimate.

For the weighted version, the proof is identical except for minor modifications in the previous paragraph. Under condition \eqref{eq:GLS-NWO}, the maximal operators \eqref{eq:Mjfg} are also bounded $L^2(w)\times L^2(\sigma)\to L^1(\mu)$ for every $w\in A_2$ and $\sigma=w^{-1}$ by Remark \ref{rem:NWOsuff-w}. Then, using Corollary \ref{cor:ellpq<Spq-w} in place of Corollary \ref{cor:ellpq<Spq} in \eqref{eq:osc<SpqFinish}, this estimate becomes
\begin{equation*}
\begin{split}
   \Norm{\{m_b(B_Q)\}_{Q\in\mathscr D}}{\ell^{p,q}}
  &\lesssim\sum_{j=1}^2\Norm{\mathcal M^j}{L^2(w)\times L^2(\sigma)\to L^1(\mu)} \Norm{ [b,T]}{S^{p,q}(L^2(w))}  \\
  &\lesssim \Norm{ [b,T] }{S^{p,q}(L^2(w))}.
\end{split}
\end{equation*}
This is the second claimed estimate and completes the proof.
\end{proof}

Note that Proposition \ref{prop:Osc<S} is identical to \eqref{eq:Osc<S} of Proposition \ref{prop:comVsOsc}, which we have proved now.

\section{On spaces of finite diameter}\label{sec:finite}

In this section we discuss a modification of Proposition \ref{prop:Osc<S} in the case of spaces $(X,\rho,\mu)$ with $\operatorname{diam}(X)<\infty$. (For spaces of homogeneous type, it is well known that this is equivalent to $\mu(X)<\infty$; see \cite[Lemma 1.9]{BC:96}.) For such spaces, non-degeneracy in the sense of Definition \ref{def:nondeg} can never hold for all scales, simply because there are no pairs of points $x,y\in X$ with $\rho(x,y)\approx r$ when $r\gg\operatorname{diam}(X)$. In such a space, it might be more reasonable to assume that a kernel is non-degenerate at all points $x\in X$ and scales $r\in(0,\operatorname{diam}(X))$.

In Proposition \ref{prop:GLS}, we require non-degeneracy at scale $Ar$ as an assumption, which can only reasonably hold for $r\in(0,A^{-1}\operatorname{diam}(X))$. As a result, in Proposition \ref{prop:Osc<S}, we may only consider cubes $Q$ with $\ell(Q)\lesssim A^{-1}\operatorname{diam}(X)$. Note that cubes with $\ell(Q)\gg \operatorname{diam}(X)$ are of no concern anyway, since they reduce to the single largest cube $Q=X$, and we are only counting cubes as different sets in the norm $\Norm{\{m_b(B_Q)\}_{Q\in\mathscr D}}{\ell^{p,q}}$. Nevertheless, there may be some cubes $Q$ with $\ell(Q)\approx\operatorname{diam}(X)$ that we are missing by restricting to $\ell(Q)\lesssim A^{-1}\operatorname{diam}(X)$. These observations lead to the following variant of Proposition \ref{prop:Osc<S}:

\begin{proposition}\label{prop:Osc<S-fin}
Let $K$ be an $\omega$-Calder\'on--Zygmund kernel with \eqref{eq:omega0} such that $K$ is non-denegerate at all points $x\in X$ and all scales $r\in(0,\operatorname{diam}(X))$, and let $b\in L^1_{\loc}(\mu)$.
Let $p\in(1,\infty)$ and $q\in[1,\infty]$.
Suppose that there exists an operator $[b,T]\in S^{p,q}(L^2(\mu))$ with kernel $(b(x)-b(y))K(x,y)$. Then
\begin{equation}\label{eq:Osc<S-fin}
  \BNorm{\Big\{1_{\{\ell(Q)<\alpha\operatorname{diam}(X)\}}m_b(B_Q)\}_{Q\in\mathscr D} }{\ell^{p,q}}
  \lesssim\Norm{[b,T]}{S^{p,q}(L^2(\mu))},
\end{equation}
for some $\alpha\in(0,1)$ depending only on the parameters of the space and the kernel~$K$.
More generally, if $w\in A_2$ and $[b,T]\in S^{p,q}(L^2(w))$, then we also have \eqref{eq:Osc<S-fin} with $S^{p,q}(L^2(w))$ in place of $S^{p,q}(L^2(\mu))$.

Under these assumptions, we still have $b\in\operatorname{Osc}^{p,q}$ qualitatively, but without a control of its norm.
\end{proposition}

\begin{proof}
All claims except for the last one were already explained in the discussion before the statement of the proposition. The last claim about $\in\operatorname{Osc}^{p,q}$ only requires an argument if $\operatorname{diam}(X)<\infty$, since otherwise the condition $\ell(Q)<\operatorname{diam}(X)$ is always satisfied.

It is easy to see that there are only boundedly many different cubes (undertstood as sets) $Q\in\mathscr D$ with $\ell(Q)\geq\alpha\operatorname{diam}(X)$. The {\em a priori} assumption that $b\in L^1_{\loc}(\mu)$ implies that $m_b(B_Q)<\infty$ for each $Q\in\mathscr D$. Thus, by the quasi-triangle inequality in $\ell^{p,q}$, it follows that
\begin{equation}\label{eq:OscQualit}
\begin{split}
  &\Norm{b}{\operatorname{Osc}^{p,q}}
  :=\Norm{\{m_b(B_Q)\}_{Q\in\mathscr D}}{\ell^{p,q}} \\
  &\quad\lesssim
  \BNorm{\Big\{1_{\{\ell(Q)<\alpha\operatorname{diam}(X)\}}m_b(B_Q)\}_{Q\in\mathscr D} }{\ell^{p,q}}
  +\!\!\!\!\!\!\sum_{\substack{Q\in\mathscr D \\ \ell(Q)\geq\alpha\operatorname{diam}(X)}}\!\!\!\!\!\!m_b(Q)<\infty
\end{split}
\end{equation}
by \eqref{eq:Osc<S-fin}, the boundedness of each $m_b(Q)$, and the boundedness of terms in the summation.
\end{proof}

The qualitative conclusion that $b\in\operatorname{Osc}^{p,q}$ is optimal in general:

\begin{example}\label{ex:Riesz-fin}
Consider a non-degenerate Calder\'on--Zygmund kernel on $\C\simeq\R^2$, e.g., the Beurling--Ahlfors kernel $B(x,y)=(x-y)^{-2}$. Let $Q_i$, $i=1,2,3,4$ be four separated unit cubes, say $[0,1]^2+v_i$, where $v_i\in\{0,3\}^2$. Let $X:=\bigcup_{i=1}^4 Q_i$ with the Euclidean distance and Lebesgue measure, and consider the kernel
\begin{equation*}
  K(x,y):=B(x,y)\Big(\sum_{i=1}^4 1_{Q_i\times Q_i}+1_{Q_1\times Q_3\cup Q_3\times Q_1}
    +1_{Q_2\times Q_4\cup Q_4\times Q_2}\Big)(x,y).
\end{equation*}
It is easy to check that this is an $\omega$-Calder\'on--Zygmund kernel with $\omega(t)=O(t)$, and it is also non-degenerate at all points and scales $r\leq\operatorname{diam}(X)$: For $x\in Q_i$ and $r\ll\operatorname{diam}(X)$, we can pick $y$ from the same $Q_i$; for $r\approx\operatorname{diam}(X)$, we can pick $y\in Q_j$ where $j\equiv i+2\mod 4$.

The kernel $K$ gives rise to an $L^2(X)$-bounded operator by the $L^2(\C)$-boundedness of the Beurling--Ahlfors transform and restriction. If $b\in L^1_{\loc}(X)$, and the commutator $[b,T]$ is in $S^{p,q}$, Proposition \ref{prop:Osc<S-fin} gives the estimate \eqref{eq:Osc<S-fin} and the qualitative property $b\in\operatorname{Osc}^{p,q}$. However, we claim that there is no quantitative control of $m_b(X)=\fint_X\abs{b-\ave{b}_X}$. To see this, it suffices to consider a function $b$ that is piecewise constant with $b(x)\equiv\beta_i$ for all $x\in Q_i\cup Q_{i+2}$, where $i=1,2$. Then one immediate checks that $[b,T]=0$, whereas $m_b(X)\approx\abs{\beta_1-\beta_2}$ can be arbitrarily large.
\end{example}

In an important class of spaces, the anomaly of Example \ref{ex:Riesz-fin} can be avoided. Note that, in contrast to most other results of this paper involving the Poincar\'e inequality in the assumptions, the current one does not assume any relation between the Poincar\'e exponent and other parameters of the problem; a Poincar\'e inequality of any exponent will do.

\begin{proposition}\label{prop:OscPoi}
Let $(X,\rho,\mu)$ be a doubling metric measure space supporting a $(1,t)$-Poincar\'e inequality for some $t\in(1,\infty)$.
Let $b\in L^1_{\loc}(\mu)$, let $p\in(1,\infty)$ and $q\in[1,\infty]$.
For any $\alpha\in(0,1)$, we have the estimate
\begin{equation*}
  \Norm{b}{\operatorname{Osc}^{p,q}}
  \lesssim \BNorm{\{1_{\{\ell(P)<\alpha\operatorname{diam}(X)\}}m_b(P)\}_{P\in\mathscr D}}{\operatorname{Osc}^{p,q}}.
\end{equation*}
\end{proposition}

\begin{proof}
By the definition of the norm on the left, we need to estimate $m_b(B_Q)$ for $Q\in\mathscr D$ with $\ell(Q)\geq\alpha\operatorname{diam}(X)$. By doubling, it is easy to see that $m_b(B_Q)\lesssim m_b(X)$ for such cubes, so it suffices to estimate the latter. To this end, we invoke the assumption of $(1,t)$-Poincar\'e inequality via the following consequence obtained in \cite[Theorem 3.4]{HK:W1p}:
\begin{equation*}
  m_b(B(x_0,R))
  \lesssim\frac{R}{r}
  \Big(\fint_{B(x_0,\max(2,\lambda)R}m_b(B(x,r))^t\ud\mu(x)\Big)^{\frac1t},\quad
  0<r<R<\infty;
\end{equation*}
with $0<r<R:=\operatorname{diam}(X)$, we obtain
\begin{equation}\label{eq:applyMacroPoi}
  m_b(X)
  \lesssim\frac{\operatorname{diam}(X)}{r}
  \Big(\fint_{X}m_b(B(x,r))^t\ud\mu(x)\Big)^{\frac1t}.
\end{equation}
Now, let us consider the collection $\mathscr D_k$ of cubes of maximal sidelength $\ell(Q)=\delta^{k_0}<\alpha\operatorname{diam}(X)$, and pick $r=\delta^{k_0}$. Given $x\in X$, let $Q_x\in\mathscr D_k$ be the unique cube of this generation that contains $x$. Then $B(x,r)\subseteq B_{Q_x}$ (assuming that the expansion constant defining $B_Q$ is large enough), yet $\mu(B_{Q_x})\lesssim\mu(B(x,r))$, and hence
\begin{equation*}
  m_b(B(x,r))\lesssim m_b(B_{Q_x})
  \leq\BNorm{\{1_{\{\ell(P)<\alpha\operatorname{diam}(X)\}}m_b(P)\}_{P\in\mathscr D}}{\operatorname{Osc}^{p,q}}.
\end{equation*}
Substituting this into \eqref{eq:applyMacroPoi} and recalling that $r=\delta^{k_0}\approx\operatorname{diam}(X)$, it follows that
\begin{equation*}
  m_b(Q)\lesssim  m_b(X)\lesssim\BNorm{\{1_{\{\ell(P)<\alpha\operatorname{diam}(X)\}}m_b(P)\}_{P\in\mathscr D}}{\operatorname{Osc}^{p,q}}
\end{equation*}
for all the boundedly many $Q\in\mathscr D$ with $\ell(Q)\geq\alpha\operatorname{diam}(X)$.

Substituting this into \eqref{eq:OscQualit}, we obtain the claim.
\end{proof}

\begin{corollary}\label{cor:Osc<S-Poi}
Let $(X,\rho,\mu)$ be a complete doubling metric measure space supporting a $(1,t)$-Poincar\'e inequality for some $t\in(1,\infty)$.
Let $K$ be an $\omega$-Calder\'on--Zygmund kernel with \eqref{eq:omega0} such that $K$ is non-denegerate at all points $x\in X$ and all scales $r\in(0,\operatorname{diam}(X))$, and let $b\in L^1_{\loc}(\mu)$.
Let $p\in(1,\infty)$ and $q\in[1,\infty]$.
Suppose that there exists an operator $[b,T]\in S^{p,q}(L^2(\mu))$ with kernel $(b(x)-b(y))K(x,y)$. Then
\begin{equation}\label{eq:Osc<S-Poi}
  \Norm{b}{\operatorname{Osc}^{p,q}}
  \lesssim\Norm{[b,T]}{S^{p,q}(L^2(\mu))},
\end{equation}
More generally, if $w\in A_2$ and $[b,T]\in S^{p,q}(L^2(w))$, then we also have \eqref{eq:Osc<S-Poi} with $S^{p,q}(L^2(w))$ in place of $S^{p,q}(L^2(\mu))$.
\end{corollary}

\begin{proof}
We have
\begin{equation*}
\begin{split}
  \Norm{b}{\operatorname{Osc}^{p,q}}
  &\lesssim \BNorm{\{1_{\{\ell(P)<\alpha\operatorname{diam}(X)\}}m_b(P)\}_{P\in\mathscr D}}{\operatorname{Osc}^{p,q}}
  \qquad\text{by Proposition \ref{prop:OscPoi}} \\
  &\lesssim \Norm{[b,T]}{S^{p,q}(L^2(w))}\qquad\text{by Proposition \ref{prop:Osc<S-fin}},
\end{split}
\end{equation*}
and this completes the proof.
\end{proof}

\section{Local oscillation characterisations of Besov spaces}\label{sec:Osc=Besov}

In Section \ref{sec:osc<com}, we established a connection between $S^{p,q}$ norms of commutators and the $\operatorname{Osc}^{p,q}$ norms of their symbol; however, Theorem \ref{thm:main} is stated in terms of the Besov norms instead. Hence, it is in order to clarify the relation between the two last-mentioned spaces. As a preparation, we record the following variant of a pointwise local oscillation formula originating from \cite{Lerner:formula}.

\begin{lemma}\label{lem:Lerner}
Let $f$ be a measurable function on $X$. For every dyadic $Q\in\mathscr D$, there exists a constant $c_Q$ and a collection $\mathscr S\subset\mathscr D$ of subcubes $S\subseteq Q$ such that
\begin{enumerate}[\rm(1)]
  \item\label{it:sparse} there are pairwise disjoint sets $E(S)\subseteq S$ with $\mu(E(S))\geq\frac12\mu(S)$, and
  \item\label{it:Lerner} for every $s\in(0,\infty)$ and almost every $x\in X$,
\begin{equation}\label{eq:Lerner}
  1_Q(x)\abs{f(x)-c_Q}
  \lesssim\sum_{S\in\mathscr S}\inf_c\Norm{f-c}{\avL^s(S)}1_S(x),
\end{equation}
\end{enumerate}
where the implied constant depends on the parameters of the space $X$ and its dyadic system and on the exponent $s\in(0,\infty)$ only.
\end{lemma}

\begin{proof}
This is essentially a combination of \cite[Theorem 11.1.12 and Lemma 11.1.1]{HNVW3}. These results are formulated on $\R^d$, but they extend  {\em mutatis mutandis} to dyadic systems on a space of homogeneous type. We only note that the value $\lambda=2^{-2-d}$ in \cite[Theorem 11.1.12]{HNVW3} needs to be replaced by a small number depending on (the doubling constant of) the space $X$ and the parameters of the dyadic system $\mathscr D$ only. This modification of  \cite[Theorem 11.1.12]{HNVW3} guarantees the existence of $\mathscr S\subset\mathscr D$ as in the lemma and such that 
\begin{equation*}
  1_Q\abs{f-c_Q}
  \lesssim\sum_{S\in\mathscr S}\operatorname{osc}_\lambda(f;S)1_S,
\end{equation*}
where $c_Q$ is any {\em $\lambda$-pseudomedian} of $f$ on $Q$ (defined and shown to exist in \cite[Definition 11.1.3 and Lemma 11.1.4]{HNVW3}, respectively) and $\operatorname{osc}_\lambda(f;S)$ is a measure of oscillation of $f$ on $S$ (defined on \cite[page 2]{HNVW3}) for which we only need the upper bound \cite[Lemma 11.1.1]{HNVW3}
\begin{equation*}
   \operatorname{osc}_\lambda(f;S)\lesssim\inf_c\Norm{f-c}{\avL^s(S)}
\end{equation*}
with the implied constant depending only on $\lambda$ (hence on $X$ and $\mathscr D$) and $s\in(0,\infty)$.
The last two displayed bounds clearly combine to give the one asserted in the lemma.
\end{proof}

\begin{proposition}\label{prop:Osc}
For $p,r\in(0,\infty)$ and $q\in(0,\infty]$, for all measurable functions $f$, let
\begin{equation*}
 \Norm{f}{\operatorname{Osc}^{p,q}_r}:= \BNorm{\Big\{\inf_c\Big(\fint_Q\abs{f-c}^r\ud\mu\Big)^{\frac1r}\Big\}_{Q\in\mathscr D}}{\ell^{p,q}}.
\end{equation*}
Then for all $p,r,s\in(0,\infty)$ and $q\in(0,\infty]$, we have
\begin{equation*}
   \Norm{f}{\operatorname{Osc}^{p,q}_r}
   \approx\Norm{f}{\operatorname{Osc}^{p,q}_s},
\end{equation*}
i.e., changing the value of the third index $r$ results in an equivalent norm.
\end{proposition}

\begin{proof}
We only need to prove $\lesssim$ when $r>s$, since the other direction is immediate.
For a given $Q\in\mathscr D$, taking the $\avL^r(Q)$ norms of \eqref{eq:Lerner} and using a basic estimate about $L^r$ norms over sparse collections $\mathscr S$ as in Lemma \ref{lem:Lerner}\eqref{it:sparse}, we obtain
\begin{equation*}
\begin{split}
  \lambda_r(Q) &:=\inf_c\Norm{f-c}{\avL^r(Q)}
  \leq\Norm{f-c_Q}{\avL^r(Q)} \\
  &\lesssim\BNorm{\sum_{S\in\mathscr S}\inf_c\Norm{f-c}{\avL^s(S)} 1_S}{\avL^r(Q)}
   \quad\text{by Lemma \ref{lem:Lerner}}\\
  &=\BNorm{\sum_{S\in\mathscr S}\lambda_s(S) 1_S}{\avL^r(Q)} \\
  &\lesssim\BNorm{\sum_{S\in\mathscr S}\lambda_s(S) 1_{E(S)}}{\avL^r(Q)}
    \quad\text{by \cite[Proposition 11.1.11]{HNVW3}} \\
  &=\Big(\sum_{S\in\mathscr S} \lambda_s(S)^r\frac{\mu(E(S))}{\mu(Q)}\Big)^{\frac1r} \\
  &\leq\Big(\sum_{S\subseteq Q} \lambda_s(S)^r\frac{\mu(S)}{\mu(Q)}\Big)^{\frac1r} 
  =\Big(\operatorname{Car}(\lambda_s)^r\Big)^{\frac1r}(Q).
\end{split}
\end{equation*}
Hence,
\begin{equation*}
\begin{split}
  \Norm{f}{\operatorname{Osc}^{p,q}_r}
  &=\Norm{\{\lambda_r(Q)\}_{Q\in\mathscr D}}{\ell^{p,q}} \\
  &\lesssim\Norm{\{\big(\operatorname{Car}(\lambda_s)^r(Q)\big)^{\frac1r}\}_{Q\in\mathscr D}}{\ell^{p,q}} 
  =\Norm{\{\operatorname{Car}(\lambda_s)^r(Q)\}_{Q\in\mathscr D}}{\ell^{\frac pr,\frac qr}}^{\frac 1r} \\
  &\lesssim\Norm{\{(\lambda_s(Q))^r\}_{Q\in\mathscr D}}{\ell^{\frac pr,\frac qr}}^{\frac 1r} 
  =\Norm{\{\lambda_s(Q)\}_{Q\in\mathscr D}}{\ell^{p,q}}
  =\Norm{f}{\operatorname{Osc}^{p,q}_s},
\end{split}
\end{equation*}
by the boundedness of the Carleson operator on $\ell^{\frac{p}{r},\frac{q}{r}}$ (Proposition \ref{prop:CarBd}) in the last inequality.
\end{proof}

\begin{lemma}\label{lem:LpOsc}
For $p\in(0,\infty)$,
\begin{equation*}
\begin{split}
  \inf_c\Big(\fint_B\abs{f-c}^p\ud\mu\Big)^{\frac1p}
  &\approx \Big(\fint_B\fint_B\abs{f(x)-f(y)}^p\ud\mu(x)\ud\mu(y)\Big)^{\frac1p} \\
  &\approx \Big(\fint_B\abs{f(x)-\ave{f}_B}^p\ud\mu(x)\Big)^{\frac1p},\quad\text{if}\quad p\in[1,\infty).
\end{split}
\end{equation*}
\end{lemma}

\begin{proof}
For each $y\in B$,
\begin{equation*}
  \inf_c\fint_B\abs{f-c}^p\ud\mu
  \leq\fint_B\abs{f(x)-f(y)}^p\ud\mu(x),
\end{equation*}
and integration over $y\in B$ proves the first ``$\leq$''.

On the other hand, $f(x)-f(y)=(f(x)-c)-(f(y)-c)$, and hence by the quasi-triangle inequality
\begin{equation*}
\begin{split}
   \Big( &\fint_B \fint_B\abs{f(x)-f(y)}^p\ud\mu(x)\ud\mu(y)\Big)^{\frac1p} \\
   &\lesssim\Big(\fint_B\fint_B\abs{f(x)-c}^p\ud\mu(x)\ud\mu(y)\Big)^{\frac1p}
   +\Big(\fint_B\fint_B\abs{f(y)-c}^p\ud\mu(x)\ud\mu(y)\Big)^{\frac1p} \\
   &=2\Big(\fint_B\abs{f(x)-c}^p\ud\mu(x)\Big)^{\frac1p}
\end{split}
\end{equation*}
by integrating out the variable that does not appear in the integrand. Taking the infimum over $c$ proves the first ``$\gtrsim$''.

The second ``$\leq$'' is again obvious. Using $f-\ave{f}_B=(f-c)-\ave{f-c}_B$, we get
\begin{equation*}
  \Big(\fint_B\abs{f-\ave{f}_B}^p\ud\mu\Big)^{\frac1p}
  \leq\Big(\fint_B\abs{f-c}^p\ud\mu\Big)^{\frac1p}+\abs{\ave{f-c}_B} 
  \leq 2\Big(\fint_B\abs{f-c}^p\ud\mu\Big)^{\frac1p},
\end{equation*}
and taking the infimum over $c$ completes the estimate.
\end{proof}

\begin{proposition}\label{prop:Bp=Osc}
Let $p\in(0,\infty)$.
For all measurable $b$ on $X$, we have
\begin{equation*}
\begin{split}
  \Norm{b}{\dot B^p(\mu)}
  &:=\Big(\int_X\int_X\frac{\abs{b(x)-b(y)}^p}{V(x,y)^2}\ud\mu(x)\ud\mu(y)\Big)^{\frac1p} \\
  &\approx\BNorm{\Big\{\inf_c\Big(\fint_{B_Q}\abs{b-c}^p\ud\mu\Big)^{\frac1p}\Big\}_{Q\in\mathscr D}}{\ell^p}.
\end{split}
\end{equation*}
If this norm is finite, then in fact $b\in L^r_{\loc}(\mu)$ for all $r\in[1,\infty)$, and
\begin{equation*}
  \Norm{b}{\dot B^p(\mu)}
  \approx\BNorm{\Big\{\Big(\fint_{B_Q}\abs{b-\ave{b}_{B_Q}}^r\ud\mu\Big)^{\frac1r}\Big\}_{Q\in\mathscr D}}{\ell^p}.
\end{equation*}
\end{proposition}

\begin{proof}
By Lemma \ref{lem:LpOsc},
\begin{equation}\label{eq:BesovNorms}
\begin{split}
  &\BNorm{\Big\{\inf_c\Big(\fint_{B_Q}\abs{b-c}^p\ud\mu\Big)^{\frac1p}\Big\}_{Q\in\mathscr D}}{\ell^p}^p \\
  &\approx\sum_{Q\in\mathscr D}\fint_{B_Q}\fint_{B_Q}\abs{b(x)-b(y)}^p\ud\mu(x)\ud\mu(y) \\
  &=\int_X\int_X\abs{b(x)-b(y)}^p\sum_{Q\in\mathscr D}\frac{1_{B_Q}(x)1_{B_Q}(y)}{\mu(B_Q)^2}
    \ud\mu(x)\ud\mu(y).
\end{split}
\end{equation}
If $Q$ is a dyadic cube of sidelength $\ell(Q)\approx\rho(x,y)$ that contains $x$, then $B_Q$ contains both $x$ and $y$, and $\mu(B_Q)\approx V(x,y)$. For any pair $(x,y)\in X\times X$, the sum over $Q\in\mathscr D$ contains at least this term, and hence the sum is at least $V(x,y)^{-2}$.

If $x\neq y$, and both $x,y\in B_Q$, then $\ell(Q)\gtrsim\rho(x,y)$. Hence every $Q$ with this property contains a minimal cube with the same property, and doubling implies that there are at most boundedly many such minimal cubes. Then every $Q$ appearing in the sum is a dyadic ancestor $P^{[k]}$ of some minimal cube $P$. For some $c\in(0,1)$ depending only on $X$ and $\mathscr D$, we have
\begin{equation*}
  \sum_{k=0}^\infty\frac{1}{\mu(P^{[k]})^2}
  \lesssim\sum_{k=0}^\infty\frac{c^{2k}}{\mu(P)^2}\lesssim\frac{1}{\mu(P)^2}\approx\frac{1}{V(x,y)^2},
\end{equation*}
and summing over the boundedly many minimal $P$, we derive
\begin{equation*}
  \sum_{Q\in\mathscr D}\frac{1_{B_Q}(x)1_{B_Q}(y)}{\mu(B_Q)^2}\lesssim\frac{1}{V(x,y)^2}.
\end{equation*}
Recall that we already verified the matching lower bound, and thus the proof of the first norm comparison is complete.

It is then an immediate application of earlier results that
\begin{equation*}
\begin{split}
  \Norm{b}{\dot B^p(\mu)}
  &\approx\BNorm{\Big\{\inf_c\Big(\fint_{B_Q}\abs{b-c}^p\ud\mu\Big)^{\frac1p}\Big\}_{Q\in\mathscr D}}{\ell^p}
  \qquad\text{(just explained)} \\
  &\approx\BNorm{\Big\{\inf_c\Big(\fint_{B_Q}\abs{b-c}^r\ud\mu\Big)^{\frac1r}\Big\}_{Q\in\mathscr D}}{\ell^p}
  \qquad\text{(Proposition \ref{prop:Osc})} \\
  &\approx\BNorm{\Big\{\Big(\fint_{B_Q}\abs{b-\ave{b}_B}^r\ud\mu\Big)^{\frac1r}\Big\}_{Q\in\mathscr D}}{\ell^p}
  \qquad\text{(Lemma \ref{lem:LpOsc})}.
\end{split}
\end{equation*}
This completes the proof.
\end{proof}

Note that Proposition \ref{prop:Bp=Osc} coincides with \eqref{eq:Osc=Besov} of Proposition \ref{prop:Osc=classical}.

\begin{remark}
Let us summarise some results proved so far. We just proved that
\begin{equation}\label{eq:remBp<Sp}
\begin{split}
  \Norm{b}{\dot B^p(\mu)} &\approx\Norm{b}{\operatorname{Osc}^{p,p}(\mu)}
  \qquad\text{(Proposition \ref{prop:Bp=Osc})} \\
  &\lesssim\Norm{[b,T]}{S^p}\qquad\text{(Proposition \ref{prop:Osc<S})},\qquad\forall p\in(1,\infty).
\end{split}
\end{equation}
Moreover, we have
\begin{equation*}%\label{eq:remSp<Bp}
  \Norm{[b,T]}{S^p}\lesssim\Norm{b}{\dot B^p(\mu)}\qquad\text{(Corollary \ref{cor:JWSpBp})},\qquad \forall p\in(2,\infty),
\end{equation*}
and, assuming that $(X,\rho,\mu)$ has lower dimension $d\in(1,\infty)$ and supports the $(1,d)$-Poincar\'e inequality,
\begin{equation}\label{eq:remBd0}
  \Norm{b}{\dot B^d(\mu)}<\infty\quad\Leftrightarrow\quad b=\text{const}\qquad\text{(Proposition \ref{prop:Bp=const})}.
%  Corollary \ref{cor:constBesov})}\qquad\forall p\in(0,d].
\end{equation}

Now suppose that $d\geq 2$. Then a combination of these results shows that
\begin{equation}\label{eq:remSp=Bp}
  [b,T]\in S^p\quad\Leftrightarrow\quad\begin{cases} b\in\dot B^p(\mu), & \text{if }p\in(d,\infty), \\ b=\text{const}, & \text{if }p\in(0,d].\end{cases}
\end{equation}
Indeed, case $p\in(d,\infty)$ is immediate by combining the upper and lower bounds for $[b,T]$, noting that $(d,\infty)\subseteq(2,\infty)\subseteq(1,\infty)$ when $d\geq 2$. As for case $p\in(0,d]$, note that ``$\Leftarrow$'' is trivial, while ``$\Rightarrow$'' follows from
\begin{equation*}
\begin{split}
  [b,T]\in S^p\quad &\Rightarrow\quad [b,T]\in S^d\quad\text{(by the containment of Schatten classes)} \\
  &\Rightarrow\quad b\in \dot B^d\quad\text{(by \eqref{eq:remBp<Sp})} \\
  &\Rightarrow\quad b=\text{constant}\quad\text{(by \eqref{eq:remBd0})}.
\end{split}
\end{equation*}

The equivalence \eqref{eq:remSp=Bp} consists of conclusions \eqref{it:p>d} and \eqref{it:p<d} of Theorem \ref{thm:main} in the special case of $d\geq 2$. As far as Theorem \ref{thm:main} is concerned, the remainder of the paper, and in particular Part \ref{part:dyadic}, is only needed to cover the lower dimensional spaces $(d<2)$ and the end-point result in \eqref{it:p=d} of Theorem \ref{thm:main}. Note that most of the concrete situations previously studied satisfy the dimensional condition $d\geq 2$.

It is also worth noting that these results for $d\geq 2$ hold under even weaker Calder\'on--Zygmund regularity than what was stated in Theorem \ref{thm:main}. The commutator upper bound of Corollary \ref{cor:JWSpBp} only required the kernel bound \eqref{eq:CZ0} but no regularity, and the commutator lower bound of Proposition \ref{prop:Osc<S} only the minimal regularity \eqref{eq:omega0}.
\end{remark}

\section{The weak-type oscillation space}\label{sec:weakSpace}

In the previous section, we identified the oscillation space $\operatorname{Osc}^{p,p}$ with the Besov spaces $\dot B^p(\mu)$ featuring in conclusion \eqref{it:p>d} of Theorem \ref{thm:main}. We now complement this result by identifying the weak-type oscillation space $\operatorname{Osc}^{d,\infty}$ with the other weak type space $L^{d,\infty}(\nu_d)$, which in turn was identified with the Haj\l{}asz--Sobolev in Theorem \ref{thm:Rupert} that we borrowed from \cite{HK:W1p}. This is relevant for conclusion \ref{it:p=d} of Theorem \ref{thm:main}.

Let us recall the mean oscillation function
\begin{equation*}
  m_b(x,t):=m_b(B(x,t)):=\fint_{B(x,t)}\abs{b-\ave{b}_{B(x,t)}}\ud\mu,\qquad
  (x,t)\in X\times(0,\infty),
\end{equation*}
and the two related norms
\begin{equation*}
\begin{split}
  \Norm{m_b}{L^{d,\infty}(\nu_d)}
  &:=\sup_{\kappa>0}\kappa\nu_d(\{m_b>\kappa\})^{\frac1d},\qquad
  \ud\nu_d(x,t):=\frac{\ud\mu(x)\ud t}{t^{d+1}}, \\
   \Norm{b}{\operatorname{Osc^{d,\infty}}}
   &:=\Norm{\{m_b(B_Q)\}_{Q\in\mathscr D}}{\ell^{d,\infty}}
   :=\sup_{\kappa>0}\kappa (\#\{Q\in\mathscr D: m_b(B_Q)>\kappa\})^{\frac1d}.
\end{split}
\end{equation*}
that we identify in the result that follows:

\begin{proposition}\label{prop:Osc=LdWeak}
Let $(X,\rho,\mu)$ be a space of homogeneous type that is Ahlfors $d$-regular for some $d\in(0,\infty)$. Then for all  $b\in L^1_{\loc}(\mu)$, we have
\begin{equation*}
  \Norm{b}{\operatorname{Osc}^{d,\infty}}\approx\Norm{m_b}{L^{d,\infty}(\nu_d)}.
\end{equation*}
\end{proposition}

\begin{proof}
The assumed Ahlfors $d$-regularity with $d\in(0,\infty)$ implies in particular a positive lower dimension, and hence Lemma \ref{lem:DuniqueLevel} shows that each $Q\in\mathscr D$ belongs to a unique $\mathscr D_k$ under the assumptions of the proposition. We denote by $\ell(Q)=\delta^k$ the ``side-length'' of $Q$.

If $B_1\subseteq B_2$ are two balls with $\mu(B_2)\lesssim\mu(B_1)$, then
\begin{equation*}
  m_b(B_1)\lesssim m_b(B_2).
\end{equation*}

From the quasi-triangle inequality and doubling, 
it follows that
\begin{equation*}
  \begin{cases}B_Q\subseteq B(x,t), \\ V(x,t)\lesssim\mu(B_Q)\end{cases}\quad
  \text{for all}\quad\begin{cases} x\in B(z_Q,c\ell(Q)),\\ t\in[ C\delta\ell(Q),C\ell(Q)),\end{cases}
\end{equation*}
for all small enough $c$ and large enough $C$ depending only on the parameters of the space and the dyadic system. Thus, if $m_b(B_Q)>\kappa$ for some $Q\in\mathscr D$, then
\begin{equation}\label{eq:mbxt>ck}
  m_b(x,t)\gtrsim\kappa\quad\forall(x,t)\in E(Q):=B(z_Q,c\ell(Q))\times[C\delta\ell(Q),C\ell(Q)),
\end{equation}
where the sets $E(Q)\subseteq X\times(0,\infty)$ are disjoint for small enough $c$m, since the balls $B(z_Q,c\ell(Q))$ are disjoint among cubes $Q$ of the same size, and the intervals $[C\delta\ell(Q),C\ell(Q))$ are disjoint among cubes of different sizes. Moreover,
\begin{equation}\label{eq:nudEQ}
  \nu_d(E(Q))=V(z_Q,c\ell(Q))\int_{C\delta\ell(Q)}^{C\ell(Q)}\frac{\ud t}{t^{d+1}}
  \approx V(z_Q,c\ell(Q))\frac{1}{\ell(Q)^d}\approx 1
\end{equation}
by Ahlfors $d$-regularity. Thus
\begin{equation*}
\begin{split}
  \#\{Q\in\mathscr D: m_b(B_Q)>\kappa\}
  &=\sum_{\substack{Q\in\mathscr D \\ m_b(B_Q)>\kappa}}1 \\
  &\approx \sum_{\substack{Q\in\mathscr D \\ m_b(B_Q)>\kappa}}\nu_d(E(Q))
  \qquad\text{by \eqref{eq:nudEQ}} \\
  &=\nu_d\Big( \bigcup_{\substack{Q\in\mathscr D \\ m_b(B_Q)>\kappa}} E(Q) \Big)   \qquad\text{by disjointness} \\
  &\leq \nu_d(\{(x,t): m_b(x,t)>c\kappa\})\qquad\text{by \eqref{eq:mbxt>ck}}.
\end{split}
\end{equation*}
This implies
\begin{equation}\label{eq:mbell<mbLnu}
  \Norm{\{m_b(B_Q)\}_{Q\in\mathscr D}}{\ell^{d,\infty}}
  \lesssim\Norm{m_b}{L^{d,\infty}(\nu_d)}.
\end{equation}

For the converse inequality, given $(x,t)\in X\times(0,\infty)$, let $Q(x,t)\in\mathscr D$ be the unique dyadic cube such that $x\in Q(x,t)$ and $\ell(Q(x,t))\in[C\delta t,Ct)$. If $C$ is large enough, depending only on the parameters of the space and the dyadic system, it follows that
\begin{equation*}
  B(x,t)\subseteq B_{Q(x,t)},\qquad  \mu(B_{Q(x,t)})\lesssim \mu(B(x,t)),
\end{equation*}
and thus
\begin{equation}\label{eq:mbxt<mbBQ}
  m_b(x,t)\lesssim m_b(B_{Q(x,t)})
\end{equation}
For $Q\in\mathscr D$, let
\begin{equation}\label{eq:defFQ}
    F(Q) :=\{(x,t)\in X\times(0,\infty):Q(x,t)=Q\} 
    =Q\times\Big(\frac{\ell(Q)}{C},\frac{\ell(Q)}{C\delta}\Big].
\end{equation}
These sets are disjoint, since the cubes $Q$ of the same side-length are disjoint, and the intervals $(\frac{\ell(Q)}{C},\frac{\ell(Q)}{C\delta}]$ are disjoint for cubes of different side-lengths $\ell(Q)=\delta^k$. Moreover,
\begin{equation}\label{eq:nudFQ}
  \nu_d(F(Q))=\mu(Q)\int_{\frac{\ell(Q)}{C}}^{\frac{\ell(Q)}{C\delta}}\frac{\ud t}{t^{d+1}}\approx\mu(Q)\frac{1}{\ell(Q)^d}\approx 1
\end{equation}
by Ahlfors $d$-regularity. Thus
\begin{equation*}
\begin{split}
  \nu_d(\{(x,t):m_b(x,t)>\kappa\})
  &\leq \nu_d(\{(x,t):m_b(B_{Q(x,t)})>c\kappa\})
  \qquad\text{by \eqref{eq:mbxt<mbBQ}} \\
  &=\sum_{\substack{Q\in\mathscr D \\ m_b(B_Q)>c\kappa}}\nu_d(\{(x,t): Q=Q(x,t)\}) \\
  &=\sum_{\substack{Q\in\mathscr D \\ m_b(B_Q)>c\kappa}}\nu_d(F(Q))\qquad\text{by definition \eqref{eq:defFQ}} \\
  &\approx \sum_{\substack{Q\in\mathscr D \\ m_b(B_Q)>c\kappa}} 1\qquad\text{by \eqref{eq:nudFQ}} \\
  &=\#\{Q\in\mathscr D:m_b(B_Q)>c\kappa\}.
\end{split}
\end{equation*}
This implies that
\begin{equation}\label{eq:mbLnu<mbell}
  \Norm{m_b}{L^{d,\infty}(\nu_d)}\lesssim\Norm{\{m_b(B_Q)\}_{Q\in\mathscr D}}{\ell^{d,\infty}}
\end{equation}
Combining \eqref{eq:mbell<mbLnu} and \eqref{eq:mbLnu<mbell}, we have proved the claimed equivalence.
\end{proof}

Note that Proposition \ref{prop:Osc=LdWeak} coincides with \eqref{eq:Osc=weakLp} of Proposition \ref{prop:Osc=classical}. Together with Proposition \ref{prop:Bp=Osc}, which settled \eqref{eq:Osc=Besov} of Proposition \ref{prop:Osc=classical}, we have now completed the proof of Proposition \eqref{prop:Osc=classical}.

\section{Dyadic paraproducts and their commutators}\label{sec:parap}

Dyadic paraproducts were considered by \cite{RS:NWO} as the first model example of their method already in their introductory discussion \cite[pp.~242--243]{RS:NWO}. %Hence, these objects very clearly fit within the bounds available by classical methods. 
In this short section, we provide an extension of these results to spaces of homogeneous type, as they will play a role in the analysis of general Calder\'on--Zygmund operators in Part \ref{part:dyadic}. This extension is not difficult, but it requires some care with the fact that the same set may appear as a dyadic cube of several different generations in our general setting.

A different approach (motivated by operator-valued extensions that we do not consider) to the Schatten properties of dyadic paraproducts has been developed in \cite{PS:04} and \cite{WZ:24,Zhang:thesis}. However, these papers deal with paraproducts related to a symmetric $d$-adic filtrations, where each cube has exactly $d$ children of equal measure (with $d=2$ in \cite{PS:04}, and $d$ being an arbitrary but fixed integer $d\geq 2$ in \cite{WZ:24,Zhang:thesis}). In this case, one can also choose a symmetric basis of the Haar functions involving $d$th roots of unity (see \cite[Definition 2.2]{WZ:24} or \cite[Definition 1.4.2]{Zhang:thesis}), and the approach of \cite{PS:04,WZ:24,Zhang:thesis} makes at least some use of this in their explicit computations. While it is possible that these considerations could be adapted to our less symmetric situation (where both the number of children and their relative sizes may vary from one cube to another), our treatment is more closely related to the original approach of \cite{RS:NWO}.

For each dyadic cube $Q\in\mathscr D$, we define the (martingale) difference operators
\begin{equation*}
  \D_Q:=\sum_{Q'\in\operatorname{ch}(Q)}\E_{Q'}-\E_Q;\qquad
  [\D]_Q:=\sum_{Q'\in [\operatorname{ch}](Q)}\E_{Q'}-\E_Q,
\end{equation*}
where the first form implicitly assumes that $Q=(Q,k)$ also carries information about its level $k$ in the dyadic system, while the second one only depends on $Q$ as a set. Note that $\D_Q=0$ if and only if $Q$ has only one dyadic child (itself!), while $[\D]_Q=0$ if and only if $Q=\{x\}$ for some $x\in X$.

For $\beta\in\BMO(\mathscr D)$, the associated dyadic paraproduct is then defined by
\begin{equation}\label{eq:paradef}
  \Pi_\beta:=\sum_{k\in\Z}\sum_{Q\in\mathscr D_k}\D_Q\beta\otimes\frac{1_Q}{\mu(Q)}=\sum_{Q\in\mathscr D}[\D]_Q\beta\otimes\frac{1_Q}{\mu(Q)}.
\end{equation}
In the latter summation, each set $Q\in\mathscr D$ is counted only once, while in the double summation, it appears according to its multiplicity among the different dyadic levels $\mathscr D_k$. The equality is explained by the fact $\D_Q=\D_{(Q,k)}$ is non-zero only if and only if $k=k_{\max}(Q)$.

\begin{lemma}\label{lem:DQbd}
\begin{equation*}
  \Norm{[\D]_Qb}{\infty}\lesssim\Norm{b-\ave{b}_Q}{\avL^1(Q)}
  :=\fint_Q\abs{b-\ave{b}_Q}\ud\mu.
\end{equation*}
\end{lemma}

\begin{proof}
On $Q'\in[\operatorname{ch}](Q)$, the value of $\abs{[\D]_Q b}$ is
\begin{equation*}
  \abs{\ave{b}_{Q'}-\ave{b}_Q}
  =\Babs{\fint_{Q'}(b-\ave{b}_Q)}
  \leq\fint_{Q'}\abs{b-\ave{b}_Q}\lesssim\fint_Q\abs{b-\ave{b}_Q},
\end{equation*}
since $Q'\subset Q$ and $\mu(Q)\lesssim\mu(Q')$ by doubling.
\end{proof}

\begin{proposition}\label{prop:paracomest}
For $p\in(0,\infty)$ and $q\in(0,\infty]$, we have the estimates 
\begin{equation*}
\begin{split}
  \Norm{\Pi_b}{S^{p,q}(L^2(\mu))} &\lesssim\Norm{b}{\operatorname{Osc}^{p,q}(\mathscr D)}, \\
  \Norm{[b,\Pi_\beta]}{S^{p,q}(L^2(\mu))} &\lesssim\Norm{b}{\operatorname{Osc}^{p,q}(\mathscr D)}\Norm{\beta}{\BMO(\mathscr D)}.
\end{split}
\end{equation*}
More generally, for every $w\in A_2$, we have
\begin{equation*}
\begin{split}
  \Norm{\Pi_b}{S^{p,q}(L^2(w))} &\lesssim\Norm{b}{\operatorname{Osc}^{p,q}(\mathscr D)}, \\
  \Norm{[b,\Pi_\beta]}{S^{p,q}(L^2(w))} &\lesssim\Norm{b}{\operatorname{Osc}^{p,q}(\mathscr D)}\Norm{\beta}{\BMO(\mathscr D)}.
\end{split}
\end{equation*}
The same estimates are also valid for the adjoint paraproducts $\Pi_b^*$ and $\Pi_\beta^*$ in place of $\Pi_b$ and $\Pi_\beta$, respectively.
\end{proposition}

\begin{proof}
For any $s\in(1,\infty)$, we can write
\begin{equation}\label{eq:pararep}
  \Pi_b 
  =\sum_{Q\in\mathscr D}[\D]_Q b\otimes\frac{1_Q}{\mu(Q)}
  =\sum_{Q\in\mathscr D}\Norm{[\D]_Q b}{\infty} f_Q\otimes 1_Q,
\end{equation}
where
\begin{equation*}
  f_Q:= \frac{\D_Q b}{\Norm{[\D]_Q b}{\infty}\mu(Q)}
\end{equation*}
Then
\begin{equation*}
  \frac{\abs{\pair{f}{f_Q}\pair{g}{1_Q}}}{\mu(Q)}\leq\Norm{f}{\avL^1(Q)}\Norm{g}{\avL^1(Q)},
\end{equation*}
and hence the corresponding maximal operator is clearly bounded $L^2(w)\times L^2(\sigma)\to L^1(\mu)$ by Remark \ref{rem:NWOsuff-w}. Hence the representation \eqref{eq:pararep} together with Lemma \ref{lem:DQbd} implies that
\begin{equation*}
\begin{split}
  \Norm{\Pi_b}{S^{p,q}}
  &\lesssim \BNorm{\Big\{ \Norm{[\D]_Q b}{\infty} \Big\}_{Q\in\mathscr D}}{\ell^{p,q}}
  \qquad\text{by Corollary \ref{cor:Spq<ellpq-w}}    \\
  &\lesssim \BNorm{\Big\{ \Norm{b-\ave{b}_Q}{\avL^{1}(Q)} \Big\}_{Q\in\mathscr D}}{\ell^{p,q}}
  =\Norm{b}{\operatorname{Osc}^{p,q}(\mathscr D)} \qquad\text{by Lemma \ref{lem:DQbd}}.
\end{split}
\end{equation*}
This is the first claimed bound.

To estimate the commutator, for any $s\in(1,\infty)$, we can write
\begin{equation}\label{eq:paracomrep}
\begin{split}
  [b,\Pi_\beta]
  &=\sum_{Q\in\mathscr D}\Big(b[\D]_Q\beta\otimes\frac{1_Q}{\mu(Q)}-[\D]_Q\beta\otimes\frac{1_Q b}{\mu(Q)}\Big) \\
  &=\sum_{Q\in\mathscr D}\Norm{b-\ave{b}_Q}{\avL^{s'}(Q)}\Big(e_Q\otimes 1_Q-[\D]_Q\beta\otimes h_Q\Big),
\end{split}  
\end{equation}
where
\begin{equation*}
  e_Q:= \frac{(b-\ave{b}_Q)[\D]_Q\beta}{\Norm{b-\ave{b}_Q}{\avL^{s'}(Q)}\mu(Q)},\qquad
  h_Q:= \frac{(b-\ave{b}_Q)1_Q}{\Norm{b-\ave{b}_Q}{\avL^{s'}(Q)}\mu(Q)}.
\end{equation*}
Then
\begin{equation}\label{eq:paraComME}
\begin{split}
  \frac{\abs{\pair{e_Q}{f}\pair{1_Q}{g} }}{\mu(Q)}
  &\leq\Norm{[\D]_Q\beta}{\infty}\Norm{f}{\avL^s(Q)}\Norm{g}{\avL^1(Q)},\\
  \frac{\abs{\pair{[\D]_Q\beta}{f}\pair{h_Q}{g} }}{\mu(Q)}
  &\leq\Norm{[\D]_Q\beta}{\infty}\Norm{f}{\avL^1(Q)}\Norm{g}{\avL^s(Q)}.
\end{split}
\end{equation}

For any $s\in(1,2)$, the related bi-sublinear maximal operators are bounded $L^2(\mu)\times L^2(\mu)\to L^1(\mu)$ with norm $\lesssim\sup_{Q\in\mathscr D}\Norm{[\D]_Q\beta}{\infty}\lesssim\Norm{\beta}{\BMO(\mathscr D)}$.
Hence
\begin{equation}\label{eq:paraComUnw}
\begin{split}
  \Norm{[b,\Pi_\beta]}{S^{p,q}(L^2(\mu))}
  &\lesssim \BNorm{\Big\{ \Norm{b-\ave{b}_Q}{L^{s'}(Q)} \Big\}_{Q\in\mathscr D}}{\ell^{p,q}}\Norm{\beta}{\BMO} \\
  &\qquad\qquad\text{by \eqref{eq:paracomrep} and Corollary \ref{cor:Spq<ellpq}} \\
  &=\Norm{b}{\operatorname{Osc}^{p,q}_{s'}(\mathscr D)}\Norm{\beta}{\BMO}
  \approx\Norm{b}{\operatorname{Osc}^{p,q}(\mathscr D)}\Norm{\beta}{\BMO} \\
  &\qquad\qquad\text{by Proposition \ref{prop:Osc}}.
\end{split}
\end{equation}
which is the claimed commutator bound in the unweighted case.

In the weighted case, we need to use the fact that we can choose $s>1$ as close to $1$ as we like. Then the bi-sublinear maximal operators related to \eqref{eq:paraComME} are still bounded on $L^2(w)\times L^2(\sigma)\to L^1(\mu)$ by Lemma \ref{lem:NWOsuff-w}. Thus the first step in \eqref{eq:paraComUnw} also hold in the form
\begin{equation*}
\begin{split}
  &\Norm{[b,\Pi_\beta]}{S^{p,q}(L^2(w))}
  \lesssim \BNorm{\Big\{ \Norm{b-\ave{b}_Q}{L^{s'}(Q)} \Big\}_{Q\in\mathscr D}}{\ell^{p,q}}\Norm{\beta}{\BMO} \\
  &\qquad\qquad\text{by \eqref{eq:paracomrep} and Corollary \ref{cor:Spq<ellpq-w}},
\end{split}
\end{equation*}
and the rest of \eqref{eq:paraComUnw} can be repeated verbatim.

For the adjoint paraproducts, we simply interchange the roles of $f$ and $g$ in the argument. Since the argument is essentially symmetric, the analogous conclusions follow.
\end{proof}

\begin{remark}
The proof of Proposition \ref{prop:paracomest} shows that, in the upper bound of the Schatten norms of commutators of $\Pi_b$, the factor $\Norm{\beta}{\BMO(\mathscr D)}\approx\Norm{\Pi_\beta}{\bddlin(L^2(\mu))}$ may be replaced by the (in general much smaller) quantity $\sup_{Q\in\mathscr D}\Norm{[\D]_Q\beta}{\infty}$. In particular, these commutators may be bounded and of Schatten class even if the operator $\Pi_\beta$ itself is unbounded.
\end{remark}

\part{Commutator bounds via dyadic representation}\label{part:dyadic}

\section{Commutators of dyadic shifts}\label{sec:shift}

In this final part of the paper, we complete the proof of Theorem \ref{thm:main} by proving the general commutator upper bound \eqref{eq:S<Osc} in the full range of parameters as stated. This will be accomplished by using the dyadic representation theorem to express the Calder\'on--Zygmund operator $T$ as a series of model operators called dyadic shifts, and proving the corresponding upper bound for these shift operators in place of $T$. For operators expressible in terms of the special shifts of \cite{Pet:00,PTV:02}, this approach has been previously employed in \cite{GLW:23,LLW:2wH,LLWW:2wR,LLW:multi}; the extension to general Calder\'on--Zygmund operators and general dyadic shifts of \cite{Hyt:A2} on $\R^d$ is due to \cite{WZ:24}.

In this section, we study the dyadic shifts and their commutators in their own right, leaving the connection to Calder\'on--Zygmund operators for the subsequent sections. We make the following definition:

\begin{definition}\label{def:shift}
A dyadic shift of complexity $(i,j)$ is an operator of the form
\begin{equation}\label{eq:shift}
  \S=\sum_{k\in\Z}\sum_{R\in\mathscr D_k}A_{(R,k)},\qquad
  A_{(R,k)}=\sum_{\substack{P\in\operatorname{ch}^i(R,k) \\ Q\in\operatorname{ch}^j(R,k)}}a_{PQR}h_P\otimes h_Q,
\end{equation}
where $h_P,h_Q$ are cancellative Haar functions on their respective cubes $P,Q$, and $a_{PQR}$ are coefficients. (If there are no cancellative Haar functions on some cube $Q$, then we simply interpret $h_Q=0$, with a similar convention for $h_P$.) The shift is said to be {\em normalised} if
\begin{equation*}
  \abs{a_{PQR}}\lesssim\frac{\sqrt{\mu(P)\mu(Q)}}{\mu(R)}.
\end{equation*}
\end{definition}

The main result of this section is the following:

\begin{theorem}\label{thm:shift}
Let $(X,\rho,\mu)$ be a space of homogeneous type of separation dimension $\Delta$ in the sense of Definition \ref{def:dims}. Let $\S$ be a normalised dyadic shift of complexity $(i,j)$ on $L^2(\mu)$, and let $b\in L^1_{\loc}(\mu)$. Then for all $p\in(0,\infty)$ and $q\in(0,\infty]$, we have
\begin{equation*}
  \Norm{[b,\S]}{S^{p,q}(L^2(\mu))}
  \lesssim \delta^{-(i\wedge j)\Delta(\frac1p-\frac12)_+}(1+i\wedge j)^{\frac1p}\Norm{b}{\operatorname{Osc}^{p,q}},
\end{equation*}
where $\delta\in(0,1)$ is the parameter from the dyadic system $\mathscr D$.
The factor $(1+i\wedge j)^{\frac1p}$ can be omitted, if $(X,\rho,\mu)$ has positive lower dimension $d>0$.
\end{theorem}

\begin{remark}\label{rem:WZshiftBd}
In the special case $X=\R^d$ and $\delta=\frac12$ (corresponding to the usual dyadic cubes) as well as $p=q$, versions of Theorem \ref{thm:shift} are contained as part of the proofs of \cite[Theorem 1.6]{WZ:24} and \cite[Theorems 1.3 and 1.4]{WZ:24b}, although not formulated as independent statements. With slightly weaker bounds in \cite{WZ:24}, the results of \cite{WZ:24b} read as 
\begin{equation}\label{eq:WZshiftBd}
\begin{split}
   \Norm{[b,\S]}{S^{p}(L^2(\R^d))}
    &\lesssim\Norm{b}{B^p(\R^d)}
    \begin{cases} (1+i+j), & p\in[2,\infty), \\
      2^{i(\frac1p-\frac12)}, & p\in(1,2), d=1, \\
    \end{cases}
\end{split}
\end{equation}
with analogous estimates for the $S^{p,\infty}$ norm. To compare with Theorem \ref{thm:shift}, we recall that $\Norm{b}{B^p}\approx\Norm{b}{\operatorname{Osc}^{p,p}}$ by Proposition \ref{prop:Bp=Osc} and note that the Euclidean space $\R^d$ satisfies $d=\Delta>0$, so that the factor $(1+i\wedge j)^{\frac1p}$ can be omitted in Theorem \ref{thm:shift}, making its statement marginally sharper than \eqref{eq:WZshiftBd}.
On a technical level, the proofs of \cite{WZ:24b,WZ:24} proceeds via the computation of $\Phi^*\Phi$ for $\Phi=[b,\S]$, while we work with the original commutator, in some sense more directly.
\end{remark}

The proof of Theorem \ref{thm:shift} will occupy the rest of this section. It will be divided into several intermediate results, some of which may have independent interest.

In the definition of dyadic shifts as in \eqref{eq:shift}, we intentionally sum over all pairs $(R,k)$ of dyadic cubes and their generations. It means that the same cube as a set may appear multiple times in the summation, and this is an issue that will require some care.
Note that  notation like ``$P\in\operatorname{ch}^i(R)$'' implies that $R$ is understood as the pair $R=(R,k)$, in which case $\operatorname{ch}^i(R)\subseteq\mathscr D_{k+i}$; in \eqref{eq:shift}, we have indicated this by using the pedantic notation $\operatorname{ch}^i(R,k)$.

The consequences of this multiple counting are relatively mild:

\begin{lemma}\label{lem:atMostMin}
Let $\S$ be a dyadic shift of complexity $(i,j)$ as in \eqref{eq:shift}.
For every set $R\in\mathscr D$, the term $A_{(R,k)}$ is nonzero only if
\begin{equation}\label{eq:KijR}
  k\in \mathcal K_{i,j}(R):=\Big[k_{\min} \vee (k_{\max}(R)-\min(i,j)),k_{\max(R)}\Big]\cap\Z,
\end{equation}
where $\#\mathcal K_{i,j}(R)\leq\min(i,j)+1$.
\end{lemma}

\begin{proof}
Recalling the notation \eqref{eq:kMinMaxQ},
consider a set $R\in\mathscr D$ and
\begin{equation*}
  k\in[k_{\min}(R),k_{\max}(R)-(i+1)]\cap\Z,
\end{equation*}
if this range is nonempty. Then
\begin{equation*}
   k_{\min}(R)\leq k+i< k+i+1\leq k_{\max(R)},
\end{equation*}
and hence both $\operatorname{ch}^i(R,k)$ and $\operatorname{ch}^{i+1}(R,k)$ consists of the single cube $P=R$ only. Thus $\operatorname{ch}(P)=\{P\}$, and there are no cancellative $h_P$ for $P\in\operatorname{ch}^i(R,k)$, which means that $A_{(R,k)}=0$.

Similarly, if $k\in[k_{\min}(R),k_{\max}(R)-(j+1)]\cap\Z$, we also get $A_{(R,k)}=0$.

Thus $A_{(R,k)}$ can be non-zero only if $k\in[k_{\min}(R),k_{\max}(R)]$ is outside the two ranges just considered, and this leaves exactly the range $\mathcal K_{i,j}(R)$ stated in the lemma. It is clear that this range has at most $\min(i,j)+1$ elements.
\end{proof}

For analysing the commutators of pointwise multipliers with dyadic shifts, it is convenient to begin by expressing the pointwise multipliers in terms of dyadic operators, notably including the dyadic paraproducts that we already encountered in Section \ref{sec:parap}. On $X=\R$, the following lemma goes back to \cite{Pet:00}; versions to $\R^d$ are well known, and the following extension to spaces of homogeneous type is straightforward. Let us introduce the dyadic test function spaces
\begin{equation*}
\begin{split}
   S_0(\mathscr D) &:=\lspan\{h_P^\alpha:P\in\mathscr D,1\leq\alpha<M_P\}, \\
   S(\mathscr D) &:=\begin{cases} S_0(\mathscr D), & \text{if}\quad\mu(X)=\infty, \\
    \lspan(S_0(\mathscr D)\cup\{1\}), & \text{if}\quad\mu(X)<\infty,\end{cases}
\end{split}
\end{equation*}
where ``$1$'' above is the constant function on $X$. Note that $S(\mathscr D)$ is dense in $L^2(\mu)$ in either case; the constants are only needed when $\mu(X)<\infty$.

\begin{lemma}\label{lem:bf}
For all $b\in L^1_{\loc}(\mu)$ and $f\in S(\mathscr D)$,
we can decompose their pointwise product as
\begin{equation*}
  bf=\Pi_b f+\sum_{1\leq\alpha<M}\Gamma_b^\alpha f+\mathfrak H_{\ave{b}}f
  +\ave{b}_X\ave{f}_X
\end{equation*}
where
\begin{equation*}
  \Pi_b=\sum_{P\in\mathscr D}\D_P b\otimes\frac{1_P}{\mu(P)},\quad
  \Gamma_b^\alpha=\sum_{P\in\mathscr D}(\D_P b)h_P^\alpha\otimes h_P^\alpha,\quad
  \mathfrak H_{\ave{b}}  =\sum_{P\in\mathscr D}\ave{b}_P\D_P,
\end{equation*}
and
\begin{equation*}
  \ave{b}_X:=\begin{cases} 0, & \text{if}\quad\mu(X)=\infty, \\ \mu(X)^{-1}\int_X b\ud\mu, & \text{if}\quad\mu(X)<\infty.\end{cases}
\end{equation*}
\end{lemma}

\begin{proof}
For $P\in\mathscr D$, we can write
\begin{equation*}
\begin{split}
  b\D_P f
  &=\ave{b}_P\D_P f
  +(b-\ave{b}_P)\D_P f, \\
  (b-\ave{b}_P)\D_P f
  &=\sum_{Q\subseteq P}\D_Q b\D_P f
  =\sum_{Q\subsetneq P}\D_Q b \D_P f+\D_P b\D_P f.
\end{split}
\end{equation*}
Since $\D_P f$ is constant on each $Q\subsetneq P$,
\begin{equation*}
  \sum_{Q\subsetneq P}\D_Q b \D_P f
  =\sum_{Q\subsetneq P}\D_Q b \ave{\D_P f}_Q
  =\Pi_b\D_P f.
\end{equation*}
Moreover,
\begin{equation*}
  \D_P b\D_P f
  =\sum_{1\leq\alpha<M_P}
  ((\D_P b) h_P^\alpha\otimes h_P^\alpha)\D_P f
  =\sum_{1\leq\alpha<M}\Gamma_b^\alpha\D_P f.
\end{equation*}
Noting that $\ave{\D_P f}_X=0$ (either by definition, if $\mu(X)=\infty$, or by the vanishing integral of $\D_P f$ otherwise), it follows that
\begin{equation*}
  b\D_P f=\ave{b}_P\D_P f+\Pi_b\D_P f+\sum_{1\leq\alpha<M}\Gamma_b^\alpha\D_P f+\ave{\D_P f}_X,
\end{equation*}
which is the claimed identity for $\D_P f$ in place of $f$.

If $\mu(X)<\infty$, we also need to consider the constant function $f=1$. Note that this is annihilated by both $\Gamma_b^\alpha$ and $\mathfrak H_{\ave{b}}$. As for $\Pi_b$, we have
\begin{equation*}
  \Pi_b 1=\sum_{P\in\mathscr D}\D_P b\ave{1}_P
  =\sum_{P\in\mathscr D}\D_P b=b-\ave{b}_X
  =b\cdot 1-\ave{b}_X\ave{1}_X,
\end{equation*}
and hence
\begin{equation*}
  b\cdot 1=\Pi_b 1+\ave{b}_X 1=\Pi_b 1+\sum_{1\leq\alpha<M}\Gamma_b^\alpha 1+\mathfrak H_{\ave{b}}1+\ave{b}_X \ave{1}_X,
\end{equation*}
which is the claimed identity with $1$ in place of $f$.

The proof is completed by observing that
\begin{equation*}
 f=\sum_{P\in\mathscr D}\D_P f+\ave{f}_X
\end{equation*}
for all $f\in S(\mathscr D)$.
\end{proof}

\begin{lemma}\label{lem:bScom}
For a dyadic shift $\S$ as in \eqref{eq:shift}, its commutator with a pointwise multiplier $b$ is given by
\begin{equation}\label{eq:bS}
  [b,\S]=[\Pi_b,\S]+\sum_{1\leq\alpha<M}[\Gamma_b^\alpha,\S]+\S_{\ave{b}},
\end{equation}
where $\S_{\ave{b}}$ has the same form as $\S$ but with $a_{PQR}(\ave{b}_P-\ave{b}_Q)$ in place of $a_{PQR}$.
\end{lemma}

\begin{proof}
From Lemma \ref{lem:bf} it is immediate that
\begin{equation*}
  [b,\S]=[\Pi_b,\S]+\sum_{1\leq\alpha<M}[\Gamma_b^\alpha,\S]+[\mathfrak H_{\ave{b}},\S]+[\ave{b}_X\E_X,\S],\quad
  \E_X f:=\ave{f}_X.
\end{equation*}
Since the multiplication by the constant $\ave{b}_X$ commutes with every operator,
the action of the last term on a function $f$ can be written as
\begin{equation*}
  [\ave{b}_X\E_X,\S]f
  =\ave{b}_X[\E_X,\S]f
  =\ave{b}_X(\ave{\S f}_X-\S\ave{f}_X).
\end{equation*}
As a linear combination of elementary tensors of the form $h_P\otimes h_Q$, the dyadic shift $\S$ both annihilates constants (since $\int h_Q=0$) and is annihilated by the average $\ave{\S f}_X$ (since $\int h_P=0$).  Thus both terms above, and hence the last term in the decomposition of $[b,\S]$, vanish.

It remains to check that $[\mathfrak H_{\ave{b}},\S]=\S_{\ave{b}}$, but this is immediate.
\end{proof}

\begin{corollary}\label{cor:bScom}
For all $p\in(0,\infty)$ and $q\in(0,\infty]$, we have
\begin{equation*}
  \Norm{[b,\S]}{S^{p,q}(L^2(w))}
  \lesssim\Norm{\Pi_b}{S^{p,q}(L^2(w))}
  +\sum_{1\leq\alpha<M}\Norm{\Gamma_b^\alpha}{S^{p,q}(L^2(w))}
  +\Norm{S_{\ave{b}}}{S^{p,q}(L^2(w))}.
\end{equation*}
\end{corollary}

\begin{proof}
The Schatten--Lorentz classes $S^{p,q}(L^2(w))$ are quasi-normed spaces, so we can apply the quasi-triangle inequality on the identity established in Lemma \ref{lem:bScom}.
\end{proof}

We already estimated $S^{p,q}(L^2(w))$ norms of the dyadic paraproducts $\Pi_b$ in Proposition \ref{prop:paracomest}, and the estimation of the operators $\Gamma_b^\alpha$ is similar. Both operators are covered by the next lemma:

\begin{lemma}\label{lem:PiGammaCom}
For all $p\in(0,\infty)$, $q\in(0,\infty]$, and $w\in A_2$, we have
\begin{equation*}
      \Norm{\Pi_b}{S^{p,q}(L^2(w))} +\Norm{\Gamma_b^\alpha}{S^{p,q}(L^2(w))} 
      \lesssim\Norm{\{m_b(Q)\}_{Q\in\mathscr D}}{\ell^{p,q}(\mathscr D)},
\end{equation*}
\end{lemma}

\begin{proof}
Expanding $\D_P b$ in terms of Haar functions, we have
\begin{equation*}
\begin{split}
  \Pi_b &=\sum_{1\leq\beta<M}\sum_{P\in\mathscr D}\frac{\pair{b}{h_P^\beta}}{\mu(P)^{\frac12}} h_P^\beta\otimes\frac{1_P}{\mu(P)^{\frac12}}, \\
  \Gamma_b^\alpha &=\sum_{1\leq\beta<M}\sum_{P\in\mathscr D}  \frac{\pair{b}{h_P^\beta}}{\mu(P)^{\frac12}} 
     \mu(P)^{\frac12}h_P^\beta h_P^\alpha\otimes h_P^\alpha.
\end{split}
\end{equation*}
From the presence of the cancellative Haar functions, we note that the terms corresponding to a dyadic cube $P=(P,k)$ for at most one value of $k$, so the issue of multiplicity of generations of dyadic cubes does not play a role here.

From the basic estimate $\abs{h_P^\beta}\lesssim 1_P/\mu(P)^{\frac12}$ of the Haar functions, it is immediate that all four families of functions
\begin{equation*}
  e_P\in\{h_P^\beta,\frac{1_P}{\mu(P)^{\frac12}}, \mu(P)^{\frac12}h_P^\beta h_P^\alpha,h_P^\alpha\}
\end{equation*}
satisfy the basic sufficient condition \eqref{eq:NWOsuffBasic} of Remark \ref{rem:NWOsuff} for the $L^2(\mu)\times L^2(\mu)\to L^1(\mu)$ boundedness of the associate bi-sublinear maximal operator \eqref{eq:ME}. By Remark \ref{rem:NWOsuff-w}, this same condition even implies the weighted  $L^2(w)\times L^2(\sigma)\to L^1(\mu)$ boundedness for every $w\in A_2$.

Hence Corollaries \ref{cor:Spq<ellpq} and \ref{cor:Spq<ellpq-w} imply that
\begin{equation*}
  \Norm{\Pi_b}{S^{p,q}(L^2(w))}+\Norm{\Gamma_b^\alpha}{S^{p,q}(L^2(w))}
  \lesssim\sum_{1\leq\beta<M}\Norm{\{\abs{\pair{b}{h_P^\beta}}/\mu(P)^{-\frac12}\}_{Q\in\mathscr D}}{\ell^{p,q}},
\end{equation*}
where
\begin{equation*}
  \Babs{\frac{\pair{b}{h_P^\beta}}{\mu(P)^{\frac12}}}
  =\Babs{\int_P(b-\ave{b}_P)\frac{h_P^\beta}{\mu(P)^{\frac12}}\ud\mu}
  \lesssim\fint_P\abs{b-\ave{b}_P}\ud\mu=m_b(P).
\end{equation*}
This completes the proof.
\end{proof}

In order to estimate the Schatten norm of the special shifts $\S_{\ave{b}}$, it is convenient to study a slightly more general problem. For fixed complexity parameters $(i,j)$, let us view the shift $\S=\S_a$ as a linear function of the parameter vector
\begin{equation*}
  a=\{a_{PQR}:R\in\mathscr D, k\in\mathcal K_{i,j}(R),P\in\operatorname{ch}^i(R,k),Q\in\operatorname{ch}^j(R,k)\},
\end{equation*}
where $\mathcal K_{i,j}(R)$ is as in \eqref{eq:KijR}. By Lemma \ref{lem:atMostMin}, these are precisely the shift coefficients that non-trivially contribute to $\S_a$; otherwise, the terms $A_R$ where they appear simply vanish by the vanishing of the related Haar functions by convention.

For these parameter vectors, we consider the norms
\begin{equation}\label{eq:apq2Norm}
  \Norm{a}{\ell^{p,q}(\ell^2)}
  :=\BNorm{\Big\{ \Big(\sum_{\substack{P\in\operatorname{ch}^i(R,k) \\ Q\in\operatorname{ch}^j(R,k)}}\abs{a_{PQR}}^2\Big)^{\frac12}
    \Big\}_{(R,k)\in\mathscr D^*_{i,j}}}{\ell^{p,q}},
\end{equation}
where
\begin{equation*}
  \mathscr D^*_{i,j}
  :=\Big\{ (R,k):R\in\mathscr D, k\in\mathcal K_{i,j}(R) \Big\}
\end{equation*}

{\em A priori}, the $\ell^2$ space inside could have different dimensions for different cubes $R$; however, introducing some extra coefficient $a_{PQR}$ that do not affect the shift $\S_a$, we may assume for convenience that we have the same $\ell^2$ space for each $R$. 

We will need the following growth bound for the number of dyadic descendants in the $i$th generation. It is through this lemma that the separation dimension $\Delta$ enters the conclusions of Theorem \ref{thm:shift} and eventually Theorem \ref{thm:main}.

\begin{lemma}\label{lem:childCount}
Let $(X,\rho,\mu)$ be a space of homogeneous type of separation dimension $\Delta$ in the sense of Definition \ref{def:dims}. Then for all $k\in\Z$, dyadic cubes $R\in\mathscr D_k$ and all $i\in\N$, we have
\begin{equation*}
   \#\operatorname{ch}^i(R,k)\lesssim \delta^{-i\Delta}.
\end{equation*}
\end{lemma}

\begin{proof}
For all $R\in\mathscr D_k$ and $P\in\operatorname{ch}^i(R)\subseteq\mathscr D_{i+k}$, we have
\begin{equation*}
  B(z_P,c_0\delta^{i+k})\subseteq P\subseteq R\subseteq B(z_R,C_0\delta^k).
\end{equation*}
The cubes $P$, and hence the balls on the left, are disjoint. Thus the points $z_P$ are $c_0\delta^{i+k}$-separated and contained in $B(z_R,C_0\delta^i)$. By the dimension assumption, the number of such points, and hence the number of $P\in\operatorname{ch}^i(R)$, is at most
\begin{equation*}
  C\Big(\frac{C_0\delta^k}{c_0\delta^{k+i}}\Big)^\Delta\approx \delta^{-i\Delta},
\end{equation*}
which is the claim bound.
\end{proof}

\begin{proposition}\label{prop:SaBd}
Let $(X,\rho,\mu)$ be a space of homogeneous type of separation dimension $\Delta$ in the sense of Definition \ref{def:dims}.
Let $\S_a$ be a dyadic shift of complexity $(i,j)$ on $L^2(\mu)$. For all $p\in(0,\infty)$ and $q\in(0,\infty]$, its Schatten norm satisfies
\begin{equation*}
  \Norm{\S_a}{S^{p,q}}\leq
  (C\delta^{-(i\wedge j)d})^{(\frac1p-\frac12)_+}\Norm{a}{\ell^{p,q}(\ell^2)}.
\end{equation*}
\end{proposition}

\begin{proof}
Let us first consider $p=q\in(0,\infty)$. From the pairwise orthogonality of the domains of the shift components
\begin{equation}\label{eq:ARk}
  A_{(R,k)}=\sum_{\substack{P\in\operatorname{ch}^i(R,k) \\ Q\in\operatorname{ch}^j(R,k)}}
  a_{PQR}h_P\otimes h_Q
\end{equation}
on the one hand, and of their ranges on the other hand, as well as the vanishing of some $A_{(R,k)}$ by Lemma \ref{lem:atMostMin}, it follows that $\S_a$ can be seen as a direct sum $\bigoplus_{(R,k)\in\mathscr D^*_{i,j}}A_{(R,k)}$ acting from an orthogonal direct sum decomposition of $L^2(\mu)$ into another one. Thus, the unordered sequence (including multiplicities) of singular values of $\S_a$ is the union of the corresponding sequences of singular values of each $A_R$. Taking the $\ell^p$ sums, it follows that
\begin{equation*}
  \Norm{\S_a}{S^p}^p
  =\sum_{(R,k)\in\mathscr D_{i,j}^*}\Norm{A_{(R,k)}}{S^p}^p.
\end{equation*}
When $p\geq 2$, for each component $A_{(R,k)}$, we simply dominate by the Hilbert--Schmidt norm:
\begin{equation*}
  \Norm{A_{(R,k)}}{S^p}
  \leq\Norm{A_{(R,k)}}{S^2}
  =\Big(\sum_{\substack{P\in\operatorname{ch}^i(R,k) \\ Q\in\operatorname{ch}^j(R,k)}} \abs{a_{PQR}}^2\Big)^{\frac12}
  =:\Norm{a_{(R,k)}}{\ell^2},\qquad p\geq 2.
\end{equation*}
For $p<2$, we note from the very formula of $A_{(R,k)}$ that it is an operator of rank at most $\min\{\#\operatorname{ch}^i(R,k),\#\operatorname{ch}^j(R,k)\}=\#\operatorname{ch}^{i\wedge j}(R,k)\leq C\delta^{-(i\wedge j)\Delta}$. Denoting the singular values of $A_{(R,k)}$ by $\sigma_h$, it hence follows that
\begin{equation*}
\begin{split}
    \Norm{A_{(R,k)}}{S^p}
  &=\Big(\sum_{h=1}^{\operatorname{rank}(A_{(R,k)})}\sigma_h^p\Big)^{\frac1p}
  \leq\Big(\sum_{h=1}^{\operatorname{rank}(A_{(R,k)})}\sigma_h^2\Big)^{\frac12}
  \Big(\sum_{h=1}^{\operatorname{rank}(A_{(R,k)})}1\Big)^{\frac1p-\frac12} \\
  &=\Norm{A_{(R,k)}}{S^2}(\operatorname{rank}A_{(R,k)})^{\frac1p-\frac12}
  \leq\Norm{a_{(R,k)}}{\ell^2} (C\delta^{-(i\wedge j)\Delta})^{\frac1p-\frac12}
\end{split}
\end{equation*}
for $p\in(0,2)$.
We may combine the two estimates to the statement 
\begin{equation*}
  \Norm{A_{(R,k)}}{S^p}\leq\Norm{a_{(R,k)}}{\ell^2} (C\delta^{-(i\wedge j)\Delta})^{(\frac1p-\frac12)_+}.
\end{equation*}
Taking the $p$-th powers and summing over $(R,k)\in\mathscr D^*_{i,j}$, this gives
\begin{equation}\label{eq:SaSp<ap2}
  \Norm{\S_a}{S^p}\leq \Norm{a}{\ell^p(\ell^2)}(C\delta^{-(i\wedge j)\Delta})^{(\frac1p-\frac12)_+}
  =:C_p\Norm{a}{\ell^p(\ell^2)}
\end{equation}

We will complete the argument by interpolation.
Let
\begin{equation}\label{eq:interpAss}
  0<p_0<p<p_1<\infty,\quad 0<\theta<1,\quad\frac{1}{p}=\frac{1-\theta}{p_0}+\frac{\theta}{p_1},\quad
  0<q\leq\infty,
\end{equation}
and let $(\ ,\ )_{\theta,q}$ denote the $K$-method of real interpolation. Then,
\begin{align}
  S^{p,q}&=(S^{p_0},S^{p_1})_{\theta,q},\qquad\text{by \cite[Th\'eor\`eme 1]{Merucci}},\label{eq:Merucci} \\
  L^{p,q}(B)&=(L^{p_0}(B),L^{p_1}(B))_{\theta,q}\qquad\text{by \cite[Proposition 9]{Kree}}\label{eq:Kree}
\end{align}
for any Banach space $B$, thus in particular for $B=\ell^2$ and $\ell^{p_h},\ell^{p,q}$ as special cases of $L^{p_h},L^{p,q}$. (While \cite[Proposition 9]{Kree} only states the seemingly scalar-valued result $L^{p,q}=(L^{p_0,q_0},L^{p_1,q_1})_{\theta,q}$, it is explained in the preamble of \cite[Part II]{Kree} that, ``up to and including \cite[Section 6]{Kree} (where \cite[Proposition 9]{Kree} appears), `functions on $\Omega$' are understood as $\mu$-measurable functions $\Omega\to B$, where $B$ is a fixed Banach space.'')

Moreover, by \cite[Theorem 3.11.2]{BL:book}, the interpolation method $(\ ,\ )_{\theta,q}$ is exact of exponent $\theta$. Combining the mentioned interpolation results with \eqref{eq:SaSp<ap2}, it follows that
\begin{equation*}
\begin{split}
  \Norm{\S_a}{S^{p,q}}
  &=\Norm{\S_a}{(S^{p_0},S^{p_1})_{\theta,q}}\qquad\text{by \eqref{eq:Merucci}} \\
  &\leq C_{p_0}^{1-\theta}C_{p_1}^\theta \Norm{a}{ (\ell^{p_0}(\ell^2),\ell^{p_1}(\ell^2))_{\theta,q}} 
 \qquad\begin{cases}\text{by \eqref{eq:SaSp<ap2} and the}\\ \text{exactness of } (\ ,\ )_{\theta,q}  \end{cases} \\
  &=C_{p_0}^{1-\theta}C_{p_1}^\theta \Norm{a}{ \ell^{p,q}(\ell^2) }\qquad\text{by \eqref{eq:Kree}}.
\end{split}
\end{equation*}

If $p>2$, we choose $2<p_0<p<p_1<\infty$ so that $C_{p_0}=C_{p_1}= C$, and hence
\begin{equation*}
   \Norm{\S_a}{S^{p,q}}
   \leq C\Norm{a}{ \ell^{p,q}(\ell^2) },\qquad p>2.
\end{equation*}
If $p<2$, we choose $0<p_0<p<p_1<2$ so that $C_{p_h}=(C\delta^{-(i\wedge j)\Delta})^{\frac{1}{p_h}-\frac12}$, and hence
\begin{equation*}
\begin{split}
   \Norm{\S_a}{S^{p,q}}
   &\leq (C\delta^{-(i\wedge j)\Delta})^{(\frac{1}{p_0}-\frac12)(1-\theta)}(C\delta^{-(i\wedge j)\Delta})^{(\frac{1}{p_1}-\frac12)\theta}
   \Norm{a}{ \ell^{p,q}(\ell^2) },\\
   &=(C\delta^{-(i\wedge j)\Delta})^{\frac{1}{p}-\frac12}
   \Norm{a}{ \ell^{p,q}(\ell^2) },\qquad p<2.
\end{split}
\end{equation*}
Let finally $p=2$. Given $\eps\in(0,1)$, we choose $p_0$ so that $\frac{1}{p_0}=\frac{1}{2}+\eps=\frac{1+2\eps}{2}$ and $p_1=4$, say. Then $p_0<2<p_1$ and hence $\frac12=\frac{1-\theta}{p_0}+\frac{\theta}{p_1}$ for some $\theta\in(0,1)$ that can be easily solved from the previous equation (but we do not need it). Then $C_{p_1}=C$ and $1-\theta<1$; hence
\begin{equation*}
   \Norm{\S_a}{S^{2,q}}
   \leq (C\delta^{-(i\wedge j)\Delta})^{(\frac{1}{p_0}-\frac12)(1-\theta)}C^{\theta}
   \Norm{a}{ \ell^{2,q}(\ell^2) }
   \leq C\delta^{-(i\wedge j)\Delta\eps}   \Norm{a}{ \ell^{p,q}(\ell^2) }.
\end{equation*}
Since this bound is valid for any $\eps\in(0,1)$, taking the limit $\eps\to 0$, we finally deduce that
\begin{equation*}
   \Norm{\S_a}{S^{2,q}}
   \leq C \Norm{a}{ \ell^{p,q}(\ell^2) },\qquad p=2,
\end{equation*}
and we have verified the claimed bound in all cases.
\end{proof}

We now specialise the previous considerations to the special dyadic shifts $\S_a=\S_{\ave{b}}$ arising from the analysis of the commutator $[b,\S]$.

\begin{corollary}\label{cor:SbBd}
Let $(X,\rho,\mu)$ be a space of homogeneous type with separation dimension $\Delta$. Let $\S$ be a normalised dyadic shift on $L^2(\mu)$, and let $S_{\ave{b}}$ be new shift related to the commutator $[b,\S]$ of $\S$ with $b\in L^1_{\loc}(\mu)$.
For all $p\in(0,\infty)$ and $q\in(0,\infty]$,
\begin{equation*}
  \Norm{\S_{\ave{b}}}{S^{p,q}}
  \lesssim (C\delta^{-(i\wedge j)\Delta})^{(\frac1p-\frac12)_+}(1+i\wedge j)^{\frac1p}
  \Norm{b}{\operatorname{Osc}^{p,q}(\mathscr D)}
\end{equation*}
The factor $(1+i\wedge j)^{\frac1p}$ can be omitted if $(X,\rho,\mu)$ has positive lower dimension.
\end{corollary}

\begin{proof}
The components of the shift $\S_{\ave{b}}$ are now of the form
\begin{equation}\label{eq:aPQRprime}
  a_{PQR}':=
  a_{PQR}(\ave{b}_P-\ave{b}_Q),\qquad
  \abs{a_{PQR}}\lesssim\frac{\sqrt{\mu(P)\mu(Q)}}{\mu(R)}.
\end{equation}
Then
\begin{equation*}
\begin{split}
  \Big(\sum_{\substack{P\in\operatorname{ch}^i(R,k) \\ Q\in\operatorname{ch}^j(R,k)}} &\abs{a_{PQR}'}^2\Big)^{\frac12}
  \lesssim\Big(\sum_{\substack{P\in\operatorname{ch}^i(R,k) \\ Q\in\operatorname{ch}^j(R,k)}}
    \frac{\mu(P)\mu(Q)}{\mu(R)^2}\abs{\ave{b}_P-\ave{b}_Q}^2\Big)^{\frac12} \\
  &\leq\Big(\sum_{\substack{P\in\operatorname{ch}^i(R,k) \\ Q\in\operatorname{ch}^j(R,k)}}
    \frac{\mu(P)\mu(Q)}{\mu(R)^2}\abs{\ave{b}_P-\ave{b}_R}^2\Big)^{\frac12} \\
   &\quad+\Big(\sum_{\substack{P\in\operatorname{ch}^i(R,k) \\ Q\in\operatorname{ch}^j(R,k)}}
    \frac{\mu(P)\mu(Q)}{\mu(R)^2}\abs{\ave{b}_Q-\ave{b}_R}^2\Big)^{\frac12} =:I+II.
\end{split}
\end{equation*}
Here
\begin{equation*}
  I=\Big(\sum_{P\in\operatorname{ch}^i(R,k)}
    \frac{\mu(P)}{\mu(R)}\abs{\ave{b}_P-\ave{b}_R}^2\Big)^{\frac12}
\end{equation*}
and
\begin{equation*}
  \abs{\ave{b}_P-\ave{b}_R}^2
  =\Babs{\fint_P(b-\ave{b}_R)}^2
  \leq\fint_P\abs{b-\ave{b}_R}^2,
\end{equation*}
so that
\begin{equation*}
\begin{split}
  I &\leq\Big(\sum_{P\in\operatorname{ch}^i(R,k)}
    \frac{\mu(P)}{\mu(R)}\fint_P\abs{b-\ave{b}_R}^2\Big)^{\frac12} \\
    &=\Big(\sum_{P\in\operatorname{ch}^i(R,k)}
    \frac{1}{\mu(R)}\int_P\abs{b-\ave{b}_R}^2\Big)^{\frac12} 
    =\Big(\frac{1}{\mu(R)}\int_R\abs{b-\ave{b}_R}^2\Big)^{\frac12}.
\end{split}
\end{equation*}
By a symmetric argument, $II$ satisfies exactly the same bound.

From Proposition \ref{prop:SaBd}, it then follows that
\begin{equation*}\begin{split}
  \Norm{\S_{\ave{b}}}{S^{p,q}}
  &\leq(C\delta^{-(i\wedge j)d})^{(\frac1p-\frac12)_+}\BNorm{\Big\{
    \Big(\sum_{\substack{P\in\operatorname{ch}^i(R,k) \\ Q\in\operatorname{ch}^j(R,k)}}\abs{a_{PQR}'}^2\Big)^{\frac12}
    \Big\}_{(R,k)\in\mathscr D^*_{i,j}}}{\ell^{p,q}} \\ 
  &\leq(C\delta^{-(i\wedge j)d})^{(\frac1p-\frac12)_+}\BNorm{\Big\{
    \Big(\fint_R\abs{b-\ave{b}_R}^2\Big)^{\frac12}
    \Big\}_{(R,k)\in\mathscr D^*_{i,j}}}{\ell^{p,q}}. 
\end{split}\end{equation*}

Abbreviating $c_R:=(\fint_R\abs{b-\ave{b}_R}^2)^{\frac12}$, note that $\{c_R\}_{(R,k)\in\mathscr D^*_{i,j}}$ a modification of the sequence $\{c_R\}_{R\in\mathscr D}$, where each $c_R$ is repeated $\#\mathcal K_{i,j}(R)\leq\min(i,j)+1$ times.
If $\sigma$ and $\sigma^*$ are two sequences, where $\sigma^*$ is formed by repeating each element of $\sigma$ at most $N$ times, its is clear that
\begin{equation*}
  \#\{\alpha:\abs{\sigma^*(\alpha)}>t\}\leq N\#\{\alpha:\abs{\sigma(\alpha)}>t\}
\end{equation*}
From the equivalent formula for the Lorentz norm (see e.g. \cite[Lemma F.3.4]{HNVW2})
\begin{equation*}
  \Norm{\sigma}{\ell^{p,q}}=p^{\frac1q}\Norm{t\mapsto \#\{\alpha:\abs{\sigma(\alpha)}>t\}^{\frac1p}}{L^q(\R_+,\frac{\ud t}{t})},
\end{equation*}
it then follows that
\begin{equation*}
  \Norm{\sigma^*}{\ell^{p,q}}\leq N^{\frac1p}\Norm{\sigma}{\ell^{p,q}}.
\end{equation*}
Applying this observation to the concrete sequence above, we conclude that
\begin{equation}\label{eq:DstarVsD}
\begin{split}
  &\BNorm{\Big\{
    \Big(\fint_R\abs{b-\ave{b}_R}^2\Big)^{\frac12}
    \Big\}_{(R,k)\in\mathscr D^*_{i,j}}}{\ell^{p,q}} \\
    &\qquad\leq (\min(i,j)+1)^{\frac1p}\BNorm{\Big\{
    \Big(\fint_R\abs{b-\ave{b}_R}^2\Big)^{\frac12}
    \Big\}_{R\in\mathscr D }}{\ell^{p,q}},
\end{split}
\end{equation}
where
\begin{equation*}
  \BNorm{\Big\{
    \Big(\fint_R\abs{b-\ave{b}_R}^2\Big)^{\frac12}
    \Big\}_{R\in\mathscr D}}{\ell^{p,q}}
    =\Norm{b}{\operatorname{Osc}^{p,q}_2(\mathscr D)}
    \approx \Norm{b}{\operatorname{Osc}^{p,q}(\mathscr D)}
\end{equation*}
by Proposition \ref{prop:Osc}.

If $(X,\rho,\mu)$ has positive lower dimension, then Lemma \ref{lem:DuniqueLevel} guarantees that each $Q\in\mathscr D$ appears in only one level $\mathscr D_k$. Thus the parameter $k$ in $(R,k)\in\mathscr D^*_{i,}$ is uniquely determined by $R\in\mathscr D$, and \eqref{eq:DstarVsD} holds with equality and without the factor $(\min(i,j)+1)^{\frac1p}$. This explains the missing of this factor in the conclusions under this special case; the rest of the proof is unaffected.
\end{proof}

We have now all components to complete:

\begin{proof}[Proof of Theorem \ref{thm:shift}]
We can estimate
\begin{equation*}
\begin{split}
  &\Norm{[b,\S]}{S^{p,q}} \\
  &\lesssim\Big(\Norm{\Pi_b}{S^{p,q}}+\sum_{1\leq\alpha<M}\Norm{\Gamma_b^\alpha}{S^{p,q}}\Big)
  +\Norm{S_{\ave{b}}}{S^{p,q}}\qquad\text{by Corollary \ref{cor:bScom}} \\
  &\lesssim\Norm{b}{\operatorname{Osc}^{p,q}(\mathscr D)}\Big(1+
  \delta^{-(i\wedge j)\Delta(\frac1p-\frac12)_+}(1+i\wedge j)^{\frac1p}\Big),
  \ \begin{cases}\text{by Lemma \ref{lem:PiGammaCom}} \\ \text{and Corollary \ref{cor:SbBd}}\end{cases}
\end{split}
\end{equation*}
where Corollary \ref{cor:SbBd} further guarantees that the factor $(1+i\wedge j)^{\frac 1p}$ can be omitted if $(X,\rho,\mu)$ has positive lower dimension $d>0$.
\end{proof}

\section{Commutators of dyadic shifts: weighted estimates}\label{sec:shift-w}

Some of the easier estimates in Section \ref{sec:shift} were already stated and proved in their weighted form. The aim of this section is to obtain a weighted version of the remaining results of Section \ref{sec:shift}, concluding in the weighted extension of Theorem \ref{thm:shift}, which is the main result of this section. While Theorem \ref{thm:shift} extends the Euclidean results of \cite{WZ:24b,WZ:24} to general spaces of homogeneous type (with slight technical improvement also in $\R^d$, as discussed in Remark \ref{rem:WZshiftBd}), the following Theorem \ref{thm:shift-w} is new even in $\R^d$, except for a particular case of special low-complexity shifts treated in \cite[Section 3.2]{GLW:23}. The said special case avoids the issue of controlling the dependence on the complexity parameters, which is a key aspect of our Theorem \ref{thm:shift-w}, to ensure the summability of a series over these shifts in the Dyadic Representation Theorem \ref{thm:DRT} below.

(For the mentioned special low-complexity shifts on $\R^d$, also more general {\em two-weighted} Schatten estimates have been obtained in \cite{LLW:2wH,LLWW:2wR}. While we have left such two-weight inequalities outside the scope of the present work, we believe that the proof of our following Theorem \ref{thm:shift-w} will suggest extensions of the main results of \cite{LLW:2wH,LLWW:2wR} beyond the particular Hilbert and Riesz transforms considered there.)

\begin{theorem}\label{thm:shift-w}
Let $(X,\rho,\mu)$ be a space of homogeneous type of separation dimension $\Delta$ in the sense of Definition \ref{def:dims}, and let $w\in A_2$. Let $\S$ be a normalised dyadic shift of complexity $(i,j)$ on $L^2(\mu)$, and let $b\in L^1_{\loc}(\mu)$. Then for all $p\in(0,\infty)$ and $q\in(0,\infty]$, we have
\begin{equation*}
  \Norm{[b,\S]}{S^{p,q}(L^2(w))}
  \lesssim \delta^{-(i\wedge j)\Delta(\frac1p-\frac12)_+}(1+i\wedge j)^{\frac1p}\Norm{b}{\operatorname{Osc}^{p,q}},
\end{equation*}
where $\delta\in(0,1)$ is the parameter from the dyadic system $\mathscr D$.
The factor $(1+i\wedge j)^{\frac1p}$ can be omitted, if $(X,\rho,\mu)$ has positive lower dimension $d>0$.
\end{theorem}

In the unweighted case, we had the simple formula $\Norm{A_{(R,k)}}{S^2(L^2(\mu))}=\Norm{a_{(R,k)}}{\ell^2}$. The weighted version of this will also involve the weight on the right-hand side.

\begin{lemma}\label{lem:ARS2}
For a weight $w$, a component $A_{(R,k)}$ of a dyadic shift as in \eqref{eq:ARk} satisfies
\begin{equation}\label{eq:ARS2}
\begin{split}
  \Norm{A_{(R,k)}}{S^2(L^2(w))}
  &\lesssim
  \Big(\sum_{\substack{P\in\operatorname{ch}^i(R,k) \\ Q\in\operatorname{ch}^j(R,k)}}
    \abs{a_{PQR}}^2 \ave{w}_P\ave{\sigma}_Q\Big)^{\frac12} \\
    &=:\Norm{a_{(R,k)}}{\ell^2(w,\sigma)},
\end{split}
\end{equation}
where $\sigma=w^{-1}$.
\end{lemma}

\begin{proof}
Note that pointwise multiplication by $w^{\frac12}$ is an isometric bijection from $L^2(w)$ to $L^2(\mu)$, and pointwise multiplication by $\sigma^{\frac12}$ is an isometric bijection from $L^2(\mu)$ to $L^2(w)$. Hence we can factorise an operator $T$ on $L^2(w)$ as
\begin{equation*}
  T=\sigma^{\frac12}\circ w^{\frac12}T\sigma^{\frac12}\circ w^{\frac12},
\end{equation*}
and
\begin{equation*}
  \Norm{T}{S^{p,q}(L^2(w))}=\Norm{w^{\frac12}T\sigma^{\frac12}}{S^{p,q}(L^2(\mu))}.
\end{equation*}
Recalling that the $S^2(L^2(\mu))$ norm of an integral operator is the $L^2(L^2)$ norm of its kernel, we obtain in particular
\begin{equation*}
\begin{split}
  \Norm{A_{(R,k)}}{S^2(L^2(w))}^2 &=\Norm{w^{\frac12}A_{(R,k)}\sigma^{\frac12}}{S^2(L^2(\mu))}^2 \\
  &=\iint_{R\times R}\Babs{\sum_{\substack{P\in\operatorname{ch}^i(R,k) \\ Q\in\operatorname{ch}^j(R,k)}}
    a_{PQR} w^{\frac12}h_P\otimes\sigma^{\frac12}h_Q}^2\ud\mu\ud\mu \\
  &=\sum_{\substack{P\in\operatorname{ch}^i(R,k) \\ Q\in\operatorname{ch}^j(R,k)}}
    \abs{a_{PQR}}^2\Norm{w^{\frac12}h_P}{L^2(\mu)}^2\Norm{\sigma^{\frac12}h_Q}{L^2(\mu)}^2,
\end{split}  
\end{equation*}
since the cubes $P\in\operatorname{ch}^i(R,k)$ are pairwise disjoint, and similarly with the cubes $Q$. Since $\abs{h_P}\lesssim\mu(P)^{-\frac12}1_P$, we get
\begin{equation*}
  \Norm{w^{\frac12}h_P}{L^2(\mu)}^2\lesssim\frac{\Norm{w^{\frac12}1_P}{L^2(\mu)}^2}{\mu(P)}=\frac{w(P)}{\mu(P)}=\ave{w}_P,
\end{equation*}
and similarly for the cubes $Q$.
\end{proof}

Lemma \ref{lem:ARS2} suggest a weighted version of Proposition \ref{prop:SaBd}, which turns out to be correct:

\begin{proposition}\label{prop:SaBd-w}
Let $(X,\rho,\mu)$ be a space of homogeneous type of separation dimension $\Delta$ in the sense of Definition \ref{def:dims}.
Let $\S_a$ be a dyadic shift of complexity $(i,j)$ on $L^2(\mu)$. For all $p\in(0,\infty)$ and $q\in(0,\infty]$ and all $w\in A_2$, its weighted Schatten norm satisfies
\begin{equation*}
  \Norm{\S_a}{S^{p,q}(L^2(w))}\leq
  (C\delta^{-(i\wedge j)\Delta})^{(\frac1p-\frac12)_+}\Norm{a}{\ell^{p,q}(\ell^2(w,\sigma))},
\end{equation*}
where
\begin{equation*}
  \Norm{a}{\ell^{p,q}(\ell^2(w,\sigma))}
  :=\Norm{\{\Norm{a_{(R,k)}}{\ell^2(w,\sigma)} \}_{(R,k)\in\mathscr D^*_{i,j}}}{\ell^{p,q}}
\end{equation*}
and $\Norm{a_{(R,k)}}{\ell^2(w,\sigma)}$ is as in \eqref{eq:ARS2}
\end{proposition}

For the proof, we will also need the following square function estimate, which is certainly well known to experts:

\begin{proposition}\label{prop:sqFnEst}
Let $(X,\rho,\mu)$ be a space of homogeneous type.
For $w\in A_2$, we have the two-sided estimate
\begin{equation*}
  \Norm{\ave{f}_X}{L^2(w)}^2+\sum_{Q\in\mathscr D}\Norm{\D_Q f}{L^2(w)}^2\approx \Norm{f}{L^2(w)}^2,\quad
\end{equation*}
for all $f\in L^2(w)$, where
\begin{equation*}
  \ave{f}_X:=\begin{cases} \mu(X)^{-1}\int_X f\ud\mu, & \text{if }\mu(X)<\infty, \\ 0, &\text{otherwise}.\end{cases}
\end{equation*}
\end{proposition}

\begin{proof}
The estimate $\Norm{\ave{f}_X}{L^2(w)}\lesssim\Norm{f}{L^2(w)}$ is only needed when $\mu(X)<\infty$, in which case $X=B$ is a ball. Then this bound is immediate from $f=fw^{\frac12}\sigma^{\frac12}$, Cauchy--Schwarz, and the definition of $w\in A_2$. For the bound
\begin{equation}\label{eq:sqFnBd}
  \sum_{Q\in\mathscr D}\Norm{\D_Q f}{L^2(w)}^2\lesssim \Norm{f}{L^2(w)}^2,\qquad w\in A_2,
\end{equation}
see e.g.~\cite[Corollary 1.3]{Treil:23}. (The cited paper, dealing with not-necessarily doubling matrix-valued weights, is somewhat of an overkill for the case at hand, but it certainly covers what we need, and the proof is fairly short.)

It remains to prove the other direction. From the identities
\begin{equation*}
  f=\ave{f}_X+\sum_{Q\in\mathscr D}\D_Q f,\qquad
  \pair{f}{g}
  =\pair{\ave{f}_{X}}{\ave{g}_{X}}+\sum_{Q\in\mathscr D}\pair{\D_Q f}{\D_Q g}
\end{equation*}
it follows that
\begin{equation*}
\begin{split}
  &\abs{\pair{f}{g}} \\
  &\leq \Big(\Norm{\ave{f}_X}{L^2(w)}^2
  +\sum_{Q\in\mathscr D}\Norm{\D_Q f}{L^2(w)}^2\Big)^{\frac12}
  \Big(\Norm{\ave{g}_X}{L^2(\sigma)}^2+\sum_{Q\in\mathscr D}\Norm{\D_Q g}{L^2(\sigma)}^2\Big)^{\frac12} \\
  &\lesssim \Big(\Norm{\ave{f}_X}{L^2(w)}^2
  +\sum_{Q\in\mathscr D}\Norm{\D_Q f}{L^2(w)}^2\Big)^{\frac12}\Norm{g}{L^2(\sigma)},
\end{split}  
\end{equation*}
where the last step follows from the fact that \eqref{eq:sqFnBd} also holds for $\sigma=w^{-1}\in A_2$ in place of $w$. Taking the supremum over $g\in L^2(\sigma)$ of norm one, we obtain an upper bound for $\Norm{f}{L^2(w)}$, which completes the claimed equivalence.
\end{proof}

\begin{proof}[Proof of Proposition \ref{prop:SaBd-w}]
Let us denote
\begin{equation*}
  \D_{(R,k)}^{(i)}:=\sum_{P\in\operatorname{ch}^i(R,k)}\D_P
  =\sum_{P\in\operatorname{ch}^i(R,k)}\D_{(P,k+i)}
\end{equation*}
Then
\begin{equation*}
\begin{split}
  \sum_{k\in\Z} &\sum_{R\in\mathscr D_k}\Norm{\D_{(R,k)}^{(i)}f}{L^2(w)}^2 \\
  &=\sum_{k\in\Z}\sum_{R\in\mathscr D_k}\sum_{P\in\operatorname{ch}^i(R,k)}
  \Norm{\D_{(P,k+i)} f}{L^2(w)}^2\qquad\text{by disjoint supports} \\
  &=\sum_{k\in\Z}\sum_{P\in\mathscr D_{k+i}}  \Norm{\D_{(P,k+i)} f}{L^2(w)}^2
  =\sum_{k\in\Z}\sum_{P\in\mathscr D_{k}}  \Norm{\D_{(P,k)} f}{L^2(w)}^2.
\end{split}
\end{equation*}
Thus Proposition \ref{prop:sqFnEst} implies that
\begin{equation*}
  \Norm{f}{L^2(w)}^2\approx\Norm{\ave{f}_X}{L^2(w)}^2
  +\sum_{k\in\Z}\sum_{R\in\mathscr D_k}\Norm{\D^{(i)}_{(R,k)}f}{L^2(w)}^2.
\end{equation*}
Hence, for every $i$,
\begin{equation*}
  f\mapsto \ave{f}_X\oplus \bigoplus_{\substack{k\in\Z \\ R\in\mathscr D_k}}\D_{(R,k)}^{(i)}f,\qquad
\end{equation*}
is an isomorphism of $L^2(w)$ onto a Hilbert space direct sum of subspaces of $L^2(w)$ consisting of constant functions (if $\mu(X)<\infty$) and the ranges $\D_{(R,k)}^{(i)}L^2(w)$.

Next, we observe that the component of the dyadic shift $A_{(R,k)}$ acts from the subspace $\D_{(R,k)}^{(j)}L^2(w)$ to $\D_{(R,k)}^{(i)}L^2(w)$ and annihilates all other $\D_{(R',k')}^{(j)}L^2(w)$. Thus the dyadic shift $\S_a$ is isomorphic to a block diagonal operator with respect to these decompositions. By direct comparison of the singular value decompositions, it follows that
\begin{equation*}
   \Norm{\S_a}{S^p(L^2(w))}^p\approx\sum_{(R,k)\in\mathscr D^*_{i,j}}\Norm{A_{(R,k)}}{S^p(L^2(w))}^p.
\end{equation*}

The rest of the proof now follows essentially by repeating the proof of Proposition \ref{prop:SaBd}, using Lemma \ref{lem:ARS2} in place of the trivial identity for $\Norm{A_{(R,k)}}{S^2}$ in the unweighted case. We have
\begin{equation*}
\begin{split}
  \Norm{A_{(R,k)}}{S^p(L^2(w))}
  &\leq(\operatorname{rank}A_{(R,k)})^{(\frac1p-\frac12)_+}\Norm{A_{(R,k)}}{S^2(L^2(w))} \\
  &\lesssim \delta^{-(i\wedge j)\Delta(\frac1p-\frac12)_+}\Norm{a_{(R,k)}}{\ell^2(w,\sigma)},
\end{split}
\end{equation*}
and thus
\begin{equation*}
  \Norm{\S_a}{S^p(L^2(w))}\lesssim \delta^{-(i\wedge j)\Delta(\frac1p-\frac12)_+}\Norm{a}{\ell^p(\ell^2(w,\sigma))}.
\end{equation*}
This is the weighted analogue of \eqref{eq:SaSp<ap2} in the proof of Proposition \ref{prop:SaBd}, and the rest of the proof runs verbatim by using the same general interpolation results as in the proof of Proposition \ref{prop:SaBd}; in \eqref{eq:Kree}, we take $B=\ell^2(w,\sigma)$ instead of plain $\ell^2$, but this does not matter, since the result is valid any Banach space $B$ anyway.
\end{proof}

As in Section \ref{sec:shift}, we will need to specialise Proposition \ref{prop:SaBd-w} to the particular shifts $\S_a=\S_{\ave{b}}$ arising from the analysis of the commutator $[b,\S]$. The weighted analogue of Corollary \ref{cor:SbBd} is the following, where we once again get rid of the weights on the right-hand side.

\begin{corollary}\label{cor:SbBd-w}
Let $(X,\rho,\mu)$ be a space of homogeneous type with separation dimension $\Delta$, and let $w\in A_2$. Let $\S$ be a normalised dyadic shift on $L^2(\mu)$, and let $S_{\ave{b}}$ be new shift related to the commutator $[b,\S]$ of $\S$ with $b\in L^1_{\loc}(\mu)$.
For all $p\in(0,\infty)$ and $q\in(0,\infty]$,
\begin{equation*}
  \Norm{\S_{\ave{b}}}{S^{p,q}(L^2(w))}
  \lesssim (C\delta^{-(i\wedge j)\Delta})^{(\frac1p-\frac12)_+}(1+i\wedge j)^{\frac1p}
  \Norm{b}{\operatorname{Osc}^{p,q}(\mathscr D)}
\end{equation*}
The factor $(1+i\wedge j)^{\frac1p}$ can be omitted if $(X,\rho,\mu)$ has positive lower dimension.
\end{corollary}

\begin{proof}
As in the proof of Corollary \ref{cor:SbBd}, the coefficients $a_{PQR}'$ of the shift $\S_{\ave{b}}=\S_{a'}$ are of the form \eqref{eq:aPQRprime}. Then
\begin{equation*}
\begin{split}
  \Norm{a_{(R,k)}'}{\ell^2(w,\sigma)} &\lesssim\Big(\sum_{\substack{P\in\operatorname{ch}^i(R,k) \\ Q\in\operatorname{ch}^j(R,k)}}
    \frac{w(P)\sigma(Q)}{\mu(R)^2}\abs{\ave{b}_P-\ave{b}_R}^2 \Big)^{\frac12} \\
   &\quad+\Big(\sum_{\substack{P\in\operatorname{ch}^i(R,k) \\ Q\in\operatorname{ch}^j(R,k)}}
    \frac{w(P)\sigma(Q)}{\mu(R)^2}\abs{\ave{b}_Q-\ave{b}_R}^2 \Big)^{\frac12} =:I+II.
\end{split}
\end{equation*}
Here
\begin{equation*}
  I=\Big(\sum_{P\in\operatorname{ch}^i(R,k)}
    \frac{\mu(P)}{\mu(R)} \abs{\ave{b}_P-\ave{b}_R}^2\ave{w}_P \sum_{Q\in\operatorname{ch}^j(R,k)}
    \frac{\sigma(Q)}{\mu(R)} \Big)^{\frac12},
\end{equation*}
where
\begin{equation}\label{eq:sigmaEasy}
  \sum_{Q\in\operatorname{ch}^j(R,k)}    \frac{\sigma(Q)}{\mu(R)}
  =\frac{\sigma(R)}{\mu(R)}=\fint_R \sigma\ud\mu
\end{equation}
For an auxiliary number $r\in(1,\infty)$, we estimate
\begin{equation*}
\begin{split}
  \sum_{P\in\operatorname{ch}^i(R,k)}
   & \frac{\mu(P)}{\mu(R)} \abs{\ave{b}_P-\ave{b}_R}^2 \ave{w}_P \\
  &\leq\Big(\sum_{P\in\operatorname{ch}^i(R,k)}\frac{\mu(P)}{\mu(R)} \abs{\ave{b}_P-\ave{b}_R}^{2r'}\Big)^{\frac{1}{r'}}
  \Big(\sum_{P\in\operatorname{ch}^i(R,k)}\frac{\mu(P)}{\mu(R)}\ave{w}_P^r\Big)^{\frac1r},
\end{split}
\end{equation*}
where
\begin{equation*}
  \mu(P)\ave{w}_P^r\leq\mu(P)\ave{w^r}_P=\int_P w^r\ud\mu,
\end{equation*}
and hence
\begin{equation*}
\begin{split}
  \Big(\sum_{P\in\operatorname{ch}^i(R,k)}\frac{\mu(P)}{\mu(R)}\ave{w}_P^r\Big)^{\frac1r}
  \leq\Big(\frac{1}{\mu(R)}\sum_{P\in\operatorname{ch}^i(R,k)}\int_P w^r \ud\mu\Big)^{\frac1r}
  =\Big(\fint_R w^r\ud\mu\Big)^{\frac1r}.
\end{split}
\end{equation*}
Writing $r=\frac{1}{1-\eps}$, the product of this term and \eqref{eq:sigmaEasy} is
\begin{equation}\label{eq:useRHI}
  \fint_R\sigma\ud\mu\Big(\fint_R w^r\ud\mu\Big)^{\frac1r}
  =\fint_R\sigma\ud\mu\Big(\fint_R \sigma^{-\frac{1}{1-\eps}}\ud\mu\Big)^{1-\eps}
  \lesssim 1
\end{equation}
provided that $\sigma\in A_{2-\eps}$. By the self-improvement property of $A_2$ weights, \cite[Lemma I.8]{ST:89}, we can fix an $\eps>0$ and hence $r=\frac{1}{1-\eps}\in(1,\infty)$, such that this holds.

Then we are only left with estimating
\begin{equation*}
\begin{split}
   \sum_{P\in\operatorname{ch}^i(R,k)}\frac{\mu(P)}{\mu(R)} \abs{\ave{b}_P-\ave{b}_R}^{2r'}
   &\leq\sum_{P\in\operatorname{ch}^i(R,k)}\frac{\mu(P)}{\mu(R)} \fint_P\abs{b-\ave{b}_R}^{2r'}\ud\mu \\
   &= \fint_R\abs{b-\ave{b}_R}^{2r'}\ud\mu.
\end{split}
\end{equation*}
Combining the estimates, we obtain
\begin{equation*}
  I\lesssim\Big(\fint_R\abs{b-\ave{b}_R}^{2r'}\ud\mu\Big)^{\frac{1}{2r'}},
\end{equation*}
and the estimate of $II$ is similar, reversing the roles of $w$ and $\sigma$.

Thus, for the coefficients \eqref{eq:aPQRprime} we have seen that
\begin{equation*}
  \Norm{a_{(R,k)}'}{\ell^2(w,\sigma)}\lesssim \Big(\fint_R\abs{b-\ave{b}_R}^{2r'}\ud\mu\Big)^{\frac{1}{2r'}}.
\end{equation*}
From Proposition \ref{prop:SaBd-w}, it then follows that
\begin{equation*}\begin{split}
  \Norm{\S_{\ave{b}}}{S^{p,q}(L^2(w))}
  &=\Norm{\S_{a'}}{S^{p,q}(L^2(w))} \\
  &\leq(C\delta^{-(i\wedge j)d})^{(\frac1p-\frac12)_+}\BNorm{\Big\{ \Norm{a_{(R,k)}'}{\ell^2(w,\sigma)}
    \Big\}_{(R,k)\in\mathscr D^*_{i,j}}}{\ell^{p,q}} \\ 
  &\leq(C\delta^{-(i\wedge j)d})^{(\frac1p-\frac12)_+}\BNorm{\Big\{
    \Big(\fint_R\abs{b-\ave{b}_R}^{2r'}\Big)^{\frac{1}{2r'}}
    \Big\}_{(R,k)\in\mathscr D^*_{i,j}}}{\ell^{p,q}}. 
\end{split}\end{equation*}
By repeating the argument leading to \eqref{eq:DstarVsD} in the proof of Corollary \ref{cor:SbBd}, we obtain
\begin{equation*}
\begin{split}
  &\BNorm{\Big\{
    \Big(\fint_R\abs{b-\ave{b}_R}^{2r'}\Big)^{\frac{1}{2r'}}
    \Big\}_{(R,k)\in\mathscr D^*_{i,j}}}{\ell^{p,q}} \\
    &\qquad\leq (\min(i,j)+1)^{\frac1p}\BNorm{\Big\{
    \Big(\fint_R\abs{b-\ave{b}_R}^{2r'}\Big)^{\frac{1}{2r'}}
    \Big\}_{R\in\mathscr D }}{\ell^{p,q}},
\end{split}
\end{equation*}
where
\begin{equation*}
  \BNorm{\Big\{
    \Big(\fint_R\abs{b-\ave{b}_R}^{2r'}\Big)^{\frac{1}{2r'}}
    \Big\}_{R\in\mathscr D}}{\ell^{p,q}}
    =\Norm{b}{\operatorname{Osc}^{p,q}_{2r'}(\mathscr D)}
    \approx \Norm{b}{\operatorname{Osc}^{p,q}(\mathscr D)}
\end{equation*}
by Proposition \ref{prop:Osc}. The argument for omitting the factor $(1+\min(i,j))^{\frac1p}$ for spaces of positive lower dimension is the same as in the proof of Corollary \ref{cor:SbBd}.
\end{proof}

\begin{proof}[Proof of Theorem \ref{thm:shift-w}]
This is the same as the proof of Theorem \ref{thm:shift}, only using the weighted Corollary \ref{cor:SbBd-w} in place of the unweighted Corollary \ref{cor:SbBd}. All other ingredients of the proof were already proved for weighted space from to beginning.
\end{proof}

\section{The dyadic representation theorem and commutators}

With significant predecessors in \cite{Figiel:90,NTV:Tb,Pet:00}, a dyadic representation theorem in the sense of today was introduced in \cite{Hyt:A2} for Calder\'on--Zygmund operators on $\R^d$. Several variants have appeared since then; a version that is particularly relevant for us is contained in \cite{GH:18}. (We note that \cite{WZ:24b,WZ:24} use a version of \cite{Hyt:Expo} instead.)

The approach of \cite{GH:18} has been extended to doubling metric spaces in \cite{Hyt:Enflo} for the study of the mapping properties of the Calder\'on--Zygmund operator itself. For efficient estimation of commutators of the operator, we will still require some elaboration, and state and prove the following version here:

\begin{theorem}\label{thm:DRT}
Let $(X,\rho,\mu)$ be a space of homogeneous type, and 
let $T\in\bddlin(L^2(\mu))$ have an $\omega$-Calder\'on--Zygmund kernel in the sense of Definition \ref{def:CZomega}. Then $T$ has a representation, in the sense of bilinear forms on $L^2(\mu)$,
\begin{equation}\label{eq:DRT}
\begin{split}
  T &=  T_1 +\E\Big(\Pi_{T1}+\Pi^*_{T^*1}\Big) \\
  &\quad +\E\sum_{m=0}^\infty\omega(\delta^m)\Big\{
  \S_{m,m}+  \sum_{i=0}^{m-1}(\S_{m,i}+\S_{i,m})+  \sum_{i=1}^m(\S_{0,i}+\S_{i,0})\Big\},
\end{split}
\end{equation}
where
\begin{enumerate}[\rm(1)]
  \item $T_1$ is an operator of $\operatorname{rank}T_1\leq 1$ given by
\begin{equation*}
  T_1=\begin{cases} \mu(X)^{-2}\pair{T1}{1} 1\otimes 1, & \text{if}\quad\mu(X)<\infty,\\ 0, & \text{if}\quad\mu(X)=\infty;\end{cases}
\end{equation*}
  \item each $\S_{m,n}$ is a (sum of $O(1)$ many) normalised dyadic shift(s) of complexity $(m,n)$ in the sense of Definition \ref{def:shift}, and
  \item  each $\Pi_b$ is a dyadic paraproduct with symbol $b\in\{T1,T^*1\}$, i.e.,
\begin{equation*}
  \Pi_b=\sum_{Q\in\mathscr D}\D_Q b\otimes\frac{1_Q}{\mu(Q)}.
\end{equation*}
\end{enumerate}
\end{theorem}

We will give the proof of Theorem \ref{thm:DRT} in Section \ref{sec:proofDRT}. In the remainder of this section, we show how to apply this theorem to commutators. From the operator identity \eqref{eq:DRT}, we obtain the corresponding identity for commutators,
\begin{equation}\label{eq:DRTcom}
\begin{split}
   [b,T]=  [b,T_1] &+\E\Big([b,\Pi_{T1}] +[b,\Pi^*_{T^*1}]\Big) +\E\sum_{m=0}^\infty \omega(\delta^m)\Big\{ [b,\S_{m,m}]+ \\
  &\qquad +  \sum_{i=0}^{m-1}([b,\S_{m,i}]+[b,\S_{i,m}])+\sum_{i=1}^m([b,\S_{0,i}]+[b,\S_{i,0}])\Big\}.
\end{split}
\end{equation}

\begin{remark}
Since \eqref{eq:DRT} is understood as an identity of bilinear forms on $L^2(\mu)$, the commutator identity \eqref{eq:DRTcom} would {\em a priori} follow as an identity of bilinear forms on
\begin{equation*}
  L^2_b(\mu):=\{f\in L^2(\mu): bf\in L^2(\mu)\}.
\end{equation*}
For any measurable, a.e.\ finite-valued $b$, this subspace is seen to be dense in $L^2(\mu)$. (Given $f\in L^2(\mu)$, the functions $f_n:=1_{\{\abs{b}\leq n\}}f$ belong to $L^2_b(\mu)$ and approach $f$ in $L^2(\mu)$ by dominated convergence.) Hence, as soon as we check the convergence of the right-hand side of \eqref{eq:DRTcom} in a Schatten $S^{p,q}$ norm (and hence, {\em a fortiori}, in the operator norm), the identity extends to an identity of bilinear forms on all of $L^2(\mu)$ by density. Similar considerations can be adapted to the weighted spaces $L^2(w)$ and $L^2(\sigma)$.
\end{remark}

Next, taking $S^{p,q}$ norms of both sides of \eqref{eq:DRTcom}, we obtain the estimate
\begin{equation}\label{eq:DRTcomSpq}
\begin{split}
  \Norm{[b,T]}{S^{p,q}} &\lesssim \Norm{[b,T_1]}{S^{p,q}} + \E\Big( \Norm{[b,\Pi_{T1}]}{S^{p,q}}+\Norm{[b,\Pi^*_{T^*1}]}{S^{p,q}}\Big) \\
  &\quad +\E\sum_{m=0}^\infty \omega(\delta^m)\Big\{ 
  \sum_{i=0}^{m-1}(\Norm{[b,\S_{m,i}]}{S^{p,q}}
  +\Norm{[b,\S_{i,m}]}{S^{p,q}}) \\
  &\quad+\sum_{i=1}^{m}(\Norm{[b,\S_{m,i}]}{S^{p,q}}+\Norm{[b,\S_{i,m}]}{S^{p,q}}) 
+\Norm{[b,\S_{m,m}]}{S^{p,q}} \Big\};
\end{split}
\end{equation}
note that this is only valid when $p,q$ are such that $\Norm{\ }{S^{p,q}}$ is (at least equivalent to) a genuine norm, not just a quasi-norm.
The same estimate \eqref{eq:DRTcomSpq} is also valid with $S^{p,q}$ replaced by $S^{p,q}(L^2(w))$ throughout.
To estimate the left-hand side, it thus remains to estimate the terms on the right.

\begin{lemma}
\begin{equation*}
  \Norm{[b,T_1]}{S^{p,q}(L^2(w))}\lesssim\Norm{b}{\operatorname{Osc}^{p,q}}.
\end{equation*}
\end{lemma}

\begin{proof}
If $\mu(X)=\infty$, then $T_1=0$, so we only need to consider $\mu(X)<\infty$. In this case we can write
\begin{equation*}
  [b,T_1]=\frac{\pair{T1}{1}}{\mu(X)^2}(b\otimes 1-1\otimes b),
\end{equation*}
where further
\begin{equation*}
  b\otimes 1-1\otimes b
  =(b-\ave{b}_X)\otimes 1-1\otimes(b-\ave{b}_X).
\end{equation*}
Hence $\Norm{[b,T_1]}{S^{p,q}(L^2(w))}$ is estimated by the sum of
\begin{equation*}
  I:=\Babs{\Bpair{T\frac{1}{\mu(X)^{\frac12}}}{\frac{1}{\mu(X)^{\frac12}}}}
  \BNorm{\frac{b-\ave{b}_X}{\mu(X)^{\frac12}}}{L^2(w)}\BNorm{\frac{1}{\mu(X)^{\frac12}}}{L^2(\sigma)},
\end{equation*}
and a similar expression $II$ with the roles of $w$ and $\sigma$ reversed.
Here
\begin{equation*}
  \Babs{\Bpair{T\frac{1}{\mu(X)^{\frac12}}}{\frac{1}{\mu(X)^{\frac12}}}}
  \leq\Norm{T}{\bddlin(L^2(\mu))}\BNorm{\frac{1}{\mu(X)^{\frac12}}}{L^2(\mu)}^2
  =\Norm{T}{\bddlin(L^2(\mu))}\lesssim 1.
\end{equation*}
Moreover,
\begin{equation*}
\begin{split}
   \BNorm{\frac{b-\ave{b}_X}{\mu(X)^{\frac12}}}{L^2(w)} &\BNorm{\frac{1}{\mu(X)^{\frac12}}}{L^2(\sigma)}
   =\Big(\fint_X\abs{b-\ave{b}_X}^2 w\ud\mu\Big)^{\frac12}\Big(\fint_X \sigma\ud\mu\Big)^{\frac12} \\
  &\leq\Big(\fint_X\abs{b-\ave{b}_X}^{2r'}\ud\mu\Big)^{\frac{1}{2r'}}
  \Big(\fint_X w^r\ud\mu\Big)^{\frac1r}\fint_X\sigma\ud\mu.
\end{split}
\end{equation*}
When $\mu(X)<\infty$, we have $X=Q=B_Q$ for some $Q\in\mathscr D$. Then the product of the last two factors is bounded by a constant by the $A_2$ condition and the same argument as in \eqref{eq:useRHI}, for $r>1$ close enough to $1$. On the other hand, the first factor above is one of the terms appearing in the definition of $\Norm{b}{\operatorname{Osc}^{p,q}_{2r'}}$ for any $p,q$. Then, for any $r'$, we have $\Norm{b}{\operatorname{Osc}^{p,q}_{2r'}}\approx \Norm{b}{\operatorname{Osc}^{p,q}}$ by Proposition \ref{prop:Osc}.

Since the same argument applies to the other term $II$ with the roles of $w$ and $\sigma$ reversed, we obtain the claimed bound.
\end{proof}

For the paraproduct term, Proposition \ref{prop:paracomest} shows that
\begin{equation*}
  \Norm{[b,\Pi_{T1}]}{S^{p,q}(L^2(w))}
  \lesssim\Norm{b}{\operatorname{Osc}^{p,q}(\mathscr D)}\Norm{T1}{\BMO(\mathscr D)}
  \lesssim\Norm{b}{\operatorname{Osc}^{p,q}(\mathscr D)}
\end{equation*}
by the $L^\infty\to\BMO$ boundedness of Calder\'on--Zygmund operators, a well-known result also in the generality of spaces of homogeneous type.

For the adjoint paraproduct, the same Proposition \ref{prop:paracomest} gives
\begin{equation*}
  \Norm{[b,\Pi_{T^*1}^*]}{S^{p,q}(L^2(w))}
  \lesssim\Norm{b}{\operatorname{Osc}^{p,q}(\mathscr D)}\Norm{T^*1}{\BMO(\mathscr D)}
  \lesssim\Norm{b}{\operatorname{Osc}^{p,q}(\mathscr D)},
\end{equation*}
since $T^*$ is another Calder\'on--Zygmund operator that also maps $L^\infty\to\BMO$.

Finally, we come to the dyadic shifts. Here, the bounds also depend on the summation variables $(m,i)$, and we need to take care of convergence. By Theorem \ref{thm:shift}, noting that $0\wedge j=j\wedge 0=0$ and $m\wedge i=i\wedge m=i$,
\begin{equation*}
\begin{split}
   &\Norm{[b,\S_{m,m}]}{S^{p,q}}+  \sum_{i=0}^{m-1}(\Norm{\S_{m,i}}{S^{p,q}}+\Norm{\S_{i,m}}{S^{p,q}})+ 
    \sum_{j=1}^m(\Norm{\S_{0,j}}{S^{p,q}}+\Norm{\S_{j,0}}{S^{p,q}}) \\
    &\lesssim\Norm{b}{\operatorname{Osc}^{p,q}}\Big(\delta^{-m\Delta(\frac1p-\frac12)_+}(1+m)^{\frac1p}
    \!+\!\sum_{i=0}^{m-1}\delta^{-i\Delta(\frac1p-\frac12)_+}(1+i)^{\frac1p}\!+\!\sum_{j=1}^m (1+j)^{\frac1p}\Big) \\
 &\approx \Norm{b}{\operatorname{Osc}^{p,q}}(1+m)^{\frac1p}\begin{cases}\delta^{-m\Delta(\frac1p-\frac12)}, & \text{if}\quad p\in(0,2), \\
 (1+m), & \text{if}\quad p\in[2,\infty),\end{cases}   
\end{split}
\end{equation*}
by summing either a geometric series or a constant series in the last step; the factors $(1+i)^{\frac1p},(1+m)^{\frac1p}$ can be omitted if $(X,\rho,\mu)$ has positive lower dimension. By Theorem \ref{thm:shift-w}, the same estimate is valid with $S^{p,q}(L^2(w))$ in place of $S^{p,q}$.

Substituting back, we obtain the following result, which, up to the verification of the Dyadic Representation Theorem \ref{thm:DRT}, completes the proof of Theorem \ref{thm:main}\eqref{it:p>d} (see also Remark \ref{rem:bTSpq} below):

\begin{corollary}\label{cor:bTSpq}
Let $(X,\rho,\mu)$ be a space of homogeneous type with separation dimension $\Delta$ in the sense of Definition \ref{def:dims}. Let $w\in A_2$ and $T\in\bddlin(L^2(\mu))$ be a singular integral operator with $\omega$-Calder\'on--Zygmund kernel in the sense of Definition \ref{def:CZomega}. Let $p\in(1,\infty)$ and $q\in[1,\infty]$, and let $b\in\operatorname{Osc}^{p,q}$. Then
\begin{equation*}
\begin{split}
  \frac{\Norm{[b,T]}{S^{p,q}(L^2(w))}}{\Norm{b}{\operatorname{Osc}^{p,q}}}
  &\lesssim1+\sum_{m=0}^\infty\omega(\delta^m)(1+m)^{\frac1p} 
  \begin{cases}\delta^{-m\Delta(\frac1p-\frac12)}, &  p\in[1,2), \\ (1+m), & p\in[2,\infty),\end{cases} \\
  &\approx 1+\int_0^1\frac{\ud t}{t}\omega(t)(1+\log\tfrac1t)^{\frac1p}
  \begin{cases} t^{-\Delta(\frac1p-\frac12)}, & p\in[1,2) \\ (1+\log\frac1t), & p\in(2,\infty),
  \end{cases}
\end{split}
\end{equation*}
where the factors $(1+m)^{\frac1p}$ and $(1+\log\frac1t)^{\frac1p}$ may be omitted if $X$ has positive lower dimension. In particular, $[b,T]\in S^{p,q}(L^2(w))$ and $\Norm{[b,T]}{S^{p,q}(L^2(w))}\lesssim\Norm{b}{\operatorname{Osc}^{p,q}}$ whenever the series or the integral above converges.
\end{corollary}

While the estimates for individual shifts are valid for all $p\in(0,\infty)$, this final estimate via the dyadic representation theorem requires that $S^{p,q}$ is a normed space, and hence we impose the restriction on $p$.

\begin{remark}\label{rem:bTSpq}
It is clear that the series in Corollary \ref{cor:bTSpq} converges
provided that $T$ is an $\omega$-Calder\'on--Zygmund operator with
\begin{equation}\label{eq:OtEta}
   \omega(t)=O(t^\eta),\qquad\eta>\Delta(\frac1p-\frac12)_+.
\end{equation}
However, for $p\geq 2$, it is enough to require the modified Dini condition
\begin{equation*}
  \int_0^1\omega(t)(1+\log\frac{1}{t})^{1+\frac1p}\frac{\ud t}{t}\approx\sum_{m=0}^\infty\omega(\delta^m)(1+m)^{1+\frac1p}<\infty,
\end{equation*}
or just the log-Dini condition
\begin{equation}\label{eq:logDini}
  \int_0^1\omega(t)(1+\log\frac{1}{t})\frac{\ud t}{t}\approx\sum_{m=0}^\infty\omega(\delta^m)(1+m)<\infty,
\end{equation}
if $(X,\rho,\mu)$ has positive lower dimension.

Recall that, by the much simpler Janson--Wolff method of Corollary \ref{cor:JWSpBp}, on the {\em unweighted} $L^2(\mu)$ with Schatten exponents $p=q\in(2,\infty)$, we do not require any Calder\'on--Zygmund regularity, only the kernel bound $\abs{K(x,y)}\lesssim V(x,y)^{-1}$.
\end{remark}

\begin{remark}\label{rem:cfWZ}
Let us compare the sufficient conditions of Corollary %\ref{cor:JWSpBp} and
\ref{cor:bTSpq} and Remark \ref{rem:bTSpq} with \cite[Theorems 1.3 and 1.4]{WZ:24b} for $X=\R^n$, which clearly has dimensions $d=\Delta=D=n$, and hence in particular positive lower dimension. Recall that $\Norm{b}{\operatorname{Osc}^{p,p}}\approx\Norm{b}{\dot B^p(\mu)}$ by Proposition \ref{prop:Bp=Osc} and, for $n\geq 2$, that
\begin{equation*}
  \Norm{b}{\operatorname{Osc}^{n,\infty}(\R^m)}
  \approx\Norm{m_b}{L^{n,\infty}(\nu_n)}
  \approx\Norm{b}{\dot M^{1,n}(\R^n)}
  \approx\Norm{b}{\dot W^{1,n}(\R^n)}
\end{equation*}
by Proposition \ref{prop:Osc=LdWeak}, Theorem \ref{thm:Rupert}, and \cite[Theorem 1]{Hajlasz:96}, respectively.

The power-type modulus condition \eqref{eq:OtEta} coincides with the assumptions of \cite{WZ:24b}, while the Dini-type condition \eqref{eq:logDini} is more general, and allows for a slightly larger class of kernels than \cite{WZ:24b}. This difference is thanks to our adapting the sharper dyadic representation theorem of \cite{GH:18} instead of that of \cite{Hyt:Expo} used in \cite{WZ:24b}.

On the other hand, the weighted version of Corollary \ref{cor:bTSpq} is new compared to \cite{WZ:24b}, even for  power-type moduli. The only weighted results of this type in the previous literature seem to be those of \cite{GLW:23}, which are restricted to the Riesz transforms only.
\end{remark}

\section{Random dyadic cubes}\label{sec:random}

The statement of Theorem \ref{thm:DRT} involves an expectation over a random choice of dyadic cubes of the underlying space. For the application of this theorem, these random cubes can be taken as a black box; however, to actually prove the theorem in Section \ref{sec:proofDRT}, it seems necessary to revisit some details of their construction. With Euclidean roots in \cite{NTV:Tb}, the first such construction in spaces of homogeneous type is from \cite{HM:12} with subsequent variants in \cite{AH:13,HK:12,HT:14,NRV}. We will quote some results from \cite{AH:13}, denoting by $A_0$ the constant in the quasi-triangle inequality $\rho(x,y)\leq A_0(\rho(x,z)+\rho(z,y))$. We follow a summary of the results of \cite{AH:13} given in \cite{Hyt:Enflo}, where only metric spaces with $A_0=1$ were treated, and hence some modifications are needed.

There are reference points $x^k_\alpha$, where $k\in\Z$ and $\alpha\in\mathscr A_k$, where each $\mathscr A_k$ is some countable index set. There is parameter space $\Omega=\prod_{k\in\Z}\Omega_k$, where each $\Omega_k$ is a copy of the same finite set $\Omega_0$. Thus, there is a natural probability measure on $\Omega$.

For every $\omega\in\Omega$, there is a partial order relation $\leq_\omega$ among pairs $(k,\alpha)$ and $(\ell,\gamma)$. Within a fixed level, we declare that $(k,\alpha)\leq_\omega(k,\beta)$ if and only if $\alpha=\beta$. The restriction of the relation $\leq_\omega$ between levels $k$ and $k+1$ depends only on the component $\omega_k$. The relation between arbitrary levels $k\leq\ell$ is obtained via extension by transitivity.

There are also new random reference points $z^k_\alpha=z^k_\alpha(\omega_k)$ such that
\begin{equation}\label{eq:z-x}
  d(z^k_\alpha(\omega),x^k_\alpha)<2A_0\delta^k. 
\end{equation}
(This follows from \cite[Eq. (2.1)]{AH:13} and the definition of $z^k_\alpha$ further down the same page.)

By \cite[Theorem 2.11]{AH:13} and \cite[Theorem 2.2]{HK:12}, there are open, ``half-open'', and closed sets (referred to as ``cubes''), respectively, $\tilde Q^k_\alpha(\omega)\subseteq Q^k_\alpha(\omega)\subseteq \bar Q^k_\alpha(\omega)$ that satisfy
\begin{equation}\label{eq:Q-BQ}
  B(z^k_\alpha,\tfrac16A_0^{-5}\delta^k)\subseteq\tilde Q^k_\alpha(\omega)
  \subseteq\bar Q^k_\alpha(\omega)\subseteq B(z^k_\alpha,6A_0^4\delta^k)
  \subseteq B(x^k_\alpha, 8A_0^5\delta^k):
\end{equation}
All but the last containment in \eqref{eq:Q-BQ} is stated in \cite[Theorem 2.11]{AH:13}, and the last containment is immediate from \eqref{eq:z-x} and the quasi-triangle inequality. The ``half-open cubes'' are not addressed in \cite{AH:13}, but they can be found in \cite{HK:12}; specifically, \cite[Lemma 2.18]{HK:12} shows how to these can be constructed with the open and closed cubes as input, but without a need to know how the open and closed cubes themselves were constructed. This observation is relevant, since the construction of \cite{AH:13} and \cite{HK:12} are slightly different; \cite[Lemma 2.18]{HK:12} shows that one can also construct $Q^k_\alpha(\omega)$ with the $\tilde Q^k_\alpha(\omega)$ and $\bar Q^k_\alpha(\omega)$ of \cite{AH:13} as input.

Note that the word ``half-open'' is used only metaphorically, without assigning a precise meaning to it in the setting that we consider. The point is simply that the ``half-open'' $Q^k_\alpha(\omega)$ play a similar role as the usual half-open dyadic cubes of $\R^d$.

We refer to these sets as ``(dyadic) cubes''. The parameter $\delta\in(0,1)$ in \eqref{eq:Q-BQ} is a small number that indicates the ratio of the scales of two consecutive generations of the dyadic system. In $\R^d$, one typically takes $\delta=\frac12$. These cubes  satisfy the disjointness
\begin{equation*}
  \tilde Q^k_\alpha(\omega)\cap\bar Q^k_\beta(\omega)=\varnothing
  = Q^k_\alpha(\omega)\cap Q^k_\beta(\omega)
  \quad(\alpha\neq\beta)
\end{equation*}
and the covering properties
\begin{equation*}
\begin{split}
  X &=\bigcup_\alpha\bar Q^k_\alpha(\omega),\quad
  \bar Q^k_\alpha(\omega)=\bigcup_{\beta:(k+1,\beta)\leq_\omega(k,\alpha)}\bar Q^{k+1}_\beta(\omega), \\
  X &=\bigcup_\alpha Q^k_\alpha(\omega),\quad
  Q^k_\alpha(\omega)=\bigcup_{\beta:(k+1,\beta)\leq_\omega(k,\alpha)} Q^{k+1}_\beta(\omega);
\end{split}
\end{equation*}
see \cite[Theorem 2.11]{AH:13} and \cite[Theorem 2.2]{HK:12} for the versions involving $\bar Q^k_\alpha(\omega)$ and $Q^k_\alpha(\omega)$, respectively.

The properties above are valid for each individual $\omega\in\Omega$.
The benefit of considering a random choice of $\omega\in\Omega$ is to achieve the probabilistic ``smallness of the boundary'' \cite[(2.3)]{AH:13}
\begin{equation}\label{eq:smallBdry}
\begin{split}
  \P_\omega\Big(u\in\bigcup_\alpha\ &\partial_\eps Q^k_\alpha(\omega)\Big)\leq C\eps^\eta,\\
  &\partial_\eps Q^k_\alpha(\omega):=\{u\in \bar Q^k_\alpha(\omega):d(u,(\tilde Q^k_\alpha(\omega))^c)<\eps\delta^k\},
\end{split}
\end{equation}
for some fixed $0<\eta\leq 1\leq C<\infty$ and all $\eps>0$. In particular, as $\eps\to 0$,
\begin{equation*}
  \P_\omega\Big(x\in\bigcup_{k,\alpha} \partial Q^k_\alpha(\omega)\Big)=0,\quad
  \partial Q^k_\alpha(\omega):=\bar Q^k_\alpha(\omega)\setminus \tilde Q^k_\alpha(\omega).
\end{equation*}
The sets $\tilde Q^k_\alpha(\omega)$ and $\bar Q^k_\alpha(\omega)$ (and hence also $\partial_\eps Q^k_\alpha(\omega)$ and $\partial Q^k_\alpha(\omega)$) depend only on $(\omega_i)_{i=k}^\infty$ but, unfortunately, it is not known whether the $Q^k_\alpha(\omega)$ can be guaranteed to have this property. Indeed, this is the main reason that we discuss the triple of open, half-open, and closed cubes. While the half-open cubes $Q^k_\alpha(\omega)$ enjoy the best partitioning properties for each fixed $\omega$, the open and closed cubes have cleaner probabilistic properties as functions (random variables) of $\omega$.

 We will also need to understand the probability of the event described in the following lemma, which is a straightforward extension of \cite[Lemma 3.1]{Hyt:Enflo} from metric to quasi-metric spaces:

\begin{lemma}\label{lem:m0}
Let $(X,\rho,\mu)$ be a space of homogeneous type.
Let us fix $\eps>0$ so small that the bound in \eqref{eq:smallBdry} satisfies $C\eps^\eta\leq\frac12$, and let
\begin{equation}\label{eq:eps0}
  \eps_0:=\frac12 A_0^{-2}\eps.
\end{equation}
Then there exists an $m_0\in\Z_+$ such that we have the following:
 For some $k\in\Z$ and $m\geq m_0$, let $x^{k+m}_\beta,x^{k+m}_\gamma$ be two reference points at level $k+m$ with
 \begin{equation}\label{eq:xkm}
   \rho(x^{k+m}_\gamma,x^{k+m}_\beta)\leq
   \eps_0\delta^k.
\end{equation}
Consider the random event
\begin{equation*}
\begin{split}
  A &:=\{\omega\in\Omega: \exists\alpha\text{ such that }(k+m,\beta)\leq_\omega(k,\alpha)\text{ and }(k+m,\gamma)\leq_\omega(k,\alpha)\} \\
  &\phantom{:}=\{\omega\in\Omega:  \exists\alpha\text{ such that }\bar Q^{k+m}_\beta(\omega)\subseteq\bar Q^k_\alpha(\omega)\text{ and }
  \bar Q^{k+m}_\gamma(\omega)\subseteq\bar Q^k_\alpha(\omega)\} \\
  &\phantom{:}=\{\omega\in\Omega:  \exists\alpha\text{ such that }Q^{k+m}_\beta(\omega)\subseteq Q^k_\alpha(\omega)\text{ and }
  Q^{k+m}_\gamma(\omega)\subseteq Q^k_\alpha(\omega)\}
\end{split}
\end{equation*}
Then
\begin{enumerate}[\rm(1)]
  \item\label{it:dep} $A$ depends only on the components $(\omega_{k+i})_{i=0}^{m-1}$, and
  \item\label{it:PAhalf} $\P(A)\geq\frac12$,
\end{enumerate}
\end{lemma}

\begin{proof}
The coincidence of the alternative descriptions of $A$ with the defining formula follows easily from the covering properties of the cubes.

Claim \eqref{it:dep} is immediate from the definition of the relation $\leq_\omega$ between levels $k$ and $k+1$ in a way that depends only on the component $\omega_{k}$, and its general definition by requiring transitivity.

When $\eps>0$ is chosen as stated, with probability at least $\frac12$, the point $x^{k+m}_\beta$ is not contained in any $\partial_\eps Q^k_\alpha(\omega)$. In particular, choosing $\alpha$ so that $x^{k+m}_\beta\in \bar Q^k_\alpha(\omega)$ (which must exist by the covering property), it follows that $\rho(x^{k+m}_\beta,(\tilde Q^k_\alpha(\omega))^c)\geq \eps\delta^k$. Now suppose that \eqref{eq:xkm} holds.
By the quasi-triangle inequality,
\begin{equation*}
  \eps\delta^k
  \leq \rho(x^{k+m}_\beta,(\tilde Q^k_\alpha(\omega))^c)
  \leq A_0\rho(z^{k+m}_\gamma(\omega),(\tilde Q^k_\alpha(\omega))^c) + A_0\rho(z^{k+m}_\gamma(\omega),x^{k+m}_\beta),
\end{equation*}
where, by another quasi-triangle inequality, \eqref{eq:z-x}, \eqref{eq:xkm}, and \eqref{eq:eps0},
\begin{equation*}
\begin{split}
  A_0\rho(z^{k+m}_\gamma(\omega),x^{k+m}_\beta)
  &\leq A_0^2\rho(z^{k+m}_\gamma(\omega),x^{k+m}_\gamma)
    +A_0^2\rho(x^{k+m}_\gamma,x^{k+m}_\beta) \\
  &\leq 2A_0^3\delta^{k+m}+A_0^2\eps_0\delta^k
  =2A_0^3\delta^{k+m}+\frac12\eps\delta^k.
\end{split}
\end{equation*}
Substituting the previous estimate in the one before it, we obtain
\begin{equation*}
  \eps\delta^k\leq A_0\rho(z^{k+m}_\gamma(\omega),(\tilde Q^k_\alpha(\omega))^c) +2A_0^3\delta^{k+m}+\frac{1}{2}\eps\delta^k,
\end{equation*}
and hence
\begin{equation*}
    A_0\rho(z^{k+m}_\gamma(\omega),(\tilde Q^k_\alpha(\omega))^c)
    \geq \delta^k\big(\frac12\eps-2A_0^3\delta^m\big)
    \geq \delta^k\big(\frac12\eps-2A_0^3\delta^{m_0}\big)>0
\end{equation*}
provided that $m_0$ is large enough, depending only on $A_0$, $\delta$, and $\eps_0$; thus, in the end, only on the parameters of $X$ and its family of dyadic systems. Hence $z^{k+m}_\gamma(\omega)\in\tilde Q^k_\alpha(\omega)$, and thus $(k+m,\gamma)\leq_\omega(k,\alpha)$.
\end{proof}

In the sequel, we will drop the somewhat heavy notation above, and denote a typical $Q^k_\alpha(\omega)$ simply by $Q$. We let $\mathscr D_k=\mathscr D_k(\omega):=\{Q^k_\alpha(\omega):\alpha\in\mathscr A_k\}$ be the dyadic cubes of generation $k$. We also write $\ell(Q):=\delta^k$ if $Q\in\mathscr D_k$. We denote by $B_Q:=B(x^k_\alpha,8A_0^5\delta^k)$ the (non-random!) ball that contains $Q=Q^k_\alpha(\omega)$ by \eqref{eq:Q-BQ}.

It will be convenient to rephrase Lemma \ref{lem:m0} in this lighter notation as follows:

\begin{lemma}\label{lem:indep}
Let $(X,\rho,\mu)$ be a space of homogeneous type.
Let $\eps_0>0$ be small enough, depending only on the properties of the underlying space. Then there exists $m_0\in\Z_+$ such that the following property holds: Suppose that two reference points $x_P,x_Q$ of dyadic cubes $\ell(P)=\ell(Q)$ satisfy 
\begin{equation*}
   \rho(x_P,x_Q)\leq\eps_0\delta^{-m}\ell(P)
\end{equation*}
for some $m\geq m_0$, and consider the random event 
 \begin{equation*}
    \{P^{(m)}=Q^{(m)}\}
\end{equation*}
that $P$ and $Q$ have the same $m$-th dyadic ancestor. Then
\begin{enumerate}[\rm(1)]
  \item\label{it:indep} $\{P^{(m)}=Q^{(m)}\}$ is independent of $\{[S]: S\in\mathscr D,\ell(S)\leq\ell(P)\}$, where $[S]$ is the equivalence class of $S\in\mathscr D$ modulo sets of $\mu$-measure zero.
  \item\label{it:piHalf} $\pi_m(P,Q):=\prob(P^{(m)}=Q^{(m)})\geq\frac12$.
\end{enumerate}
\end{lemma}

\begin{proof}
Denoting by $k+m$ the common level of the reference points $x_P=x_\beta^{k+m}$, $x_Q=x_\gamma^{k+m}$, we see that the even $\{P^{(m)}=Q^{(m)}\}$ coincides with the event denoted by $A$ in Lemma \ref{lem:m0}. Then Claim \eqref{it:piHalf} is immediate from Lemma \ref{lem:m0}\eqref{it:PAhalf}.

Concerning Claim \ref{it:indep}, Lemma \ref{lem:m0} guarantees that $A=\{P^{(m)}=Q^{(m)}\}$ only depends on $(\omega_{k+i})_{i=0}^{m-1}=(\omega_i)_{i=k}^{k+m-1}$. On the other hand
\begin{equation*}
\begin{split}
  \{\bar S:S\in\mathscr D(\omega),\ell(S)\leq\ell(P)\}
  &=\{\bar Q^j_{\eta}(\omega): j\geq k+m,\eta\in\mathscr A_j\} \\
  &=\{\bar Q^j_{\eta}((\omega_i)_{i=j}^\infty): j\geq k+m,\eta\in\mathscr A_j\}
\end{split}
\end{equation*}
only depends on $(\omega_i)_{i=k+m}^\infty$, and is hence independent of $A$.

Next, given the $\sigma$-finite measure $\mu$, \cite[Theorem 5.6]{HK:12} guarantees that
\begin{equation*}
  \mu\Big(\bigcup_{S\in\mathscr D(\omega)}\partial S\Big)=
  \mu\Big(\bigcup_{k,\eta}\partial Q^k_\eta(\omega)\Big)=0\quad\text{for a.e. }\omega\in\Omega.
\end{equation*}
Let us denote the exceptional set by $Z$. Since $\tilde S\subseteq S\subseteq\bar S$ and $\partial S=\bar S\setminus\tilde S$, it follows that $\bar S\in[S]$ whenever $S\in\mathscr D(\omega)$ and $\omega\in\Omega\setminus Z$.

We have proved that $\{P^{(m)}=Q^{(m)}\}$ is independent of $\{\bar S:S\in\mathscr D,\ell(S)\leq\ell(P)\}$, and hence of $\{[\bar S]:S\in\mathscr D,\ell(S)\leq\ell(P)\}$. But the latter collection coincides with $\{[S]:S\in\mathscr D,\ell(S)\leq\ell(P)\}$ almost surely (i.e., for almost every $\omega\in\Omega$). Since independence is preserved under almost sure equality, we conclude that $\{P^{(m)}=Q^{(m)}\}$ is independent of $\{[S]:S\in\mathscr D,\ell(S)\leq\ell(P)\}$, confirming Claim \eqref{it:indep}.
\end{proof}

For each fixed $\omega\in\Omega$, the system $\mathscr D(\omega)$ is an instance of a system of dyadic cubes as in Definition \ref{def:cubes}. In particular, there are the associated Haar functions etc. We will make extensive use of them in the following section.

\section{Proof of the dyadic representation theorem}\label{sec:proofDRT}

This section is dedicated to the proof of Theorem \ref{thm:DRT} with the help of the random dyadic cubes discussed in Section \ref{sec:random}. Throughout this section, a dyadic cube $Q$ will be understood as abbreviation for the pair $(Q,k)$ including the information about the level $k$ of $Q$. Thus, we will only use the lighter notation $Q$ instead of $(Q,k)$. We recall the conditional expectation operators and their differences
\begin{equation*}
\begin{split}
  \E_Q f &:=\ave{f}_Q 1_Q=\pair{f}{h_Q^0}h_Q^0,\\
  \D_Q f &:=\sum_{Q'\in\operatorname{ch}(Q)}\E_{Q'}f-\E_Qf=\sum_{1\leq\alpha<M_Q}\pair{f}{h_Q^\alpha}h_Q^\alpha.
\end{split}
\end{equation*}
where $h_Q^0=\mu(Q)^{-1/2}1_Q$ and $h_Q^\alpha$, $1\leq\alpha<M_Q\leq M<\infty$ are the non-cancellative and the cancellative Haar functions associated with a cube $Q$; see Remark \ref{rem:Haar}. We also define
\begin{equation*}
  \E_k:=\sum_{Q\in\mathscr D_k}\E_Q,\qquad\D_k:=\E_{k+1}-\E_k=\sum_{Q\in\mathscr D_k}\D_Q.
\end{equation*}
For $f\in L^2(\mu)$, it is easy to check that
\begin{equation*}
   \E_k f\to\begin{cases} f, & k\to\infty, \\ \ave{f}_X, & k\to-\infty,\end{cases}\qquad
   \ave{f}_X:=\begin{cases} 0, & \mu(X)=\infty, \\ \fint_X f\ud\mu, & \mu(X)<\infty.\end{cases}
\end{equation*}
Hence, for $T\in\bddlin(L^2(\mu))$ and $f,g\in L^2(\mu)$, it follows that
\begin{equation*}
  \pair{T\E_k f}{\E_k g}\to\begin{cases}\pair{Tf}{g}, & k\to\infty, \\ \ave{f}_X\ave{g}_X\pair{T1}{1},  & k\to-\infty.\end{cases}
\end{equation*}
Thus, in the sense of bilinear forms on $L^2(\mu)$, we can write
\begin{equation}\label{eq:T0T1}
  T=\lim_{\substack{b\to\infty \\ a\to-\infty}}(\E_bT\E_b-\E_aT\E_a)+\pair{T1}{1}\frac{1}{\mu(X)}\otimes\frac{1}{\mu(X)}=:T_0+T_1,
\end{equation}
where $T_1$ is as in Theorem \ref{thm:DRT}; note that $T_1$, and hence also $T_0=T-T_1$, is independent of the dyadic system used in the above decomposition.

We begin the analysis of $T_0$ with a basic decomposition valid for any fixed dyadic system $\mathscr D$; i.e., the eventual random selection does not yet play any role at this point.

\begin{lemma}\label{lem:basicDec}
For $T_0$ as in \eqref{eq:T0T1}, we have
\begin{equation*}
\begin{split}
    T_0 &=\Pi_{T1}+\Pi_{T^*1}^*+T_{00}, \\
    T_{00} &=\sum_{k\in\Z}\sum_{P,Q\in\mathscr D_k}
    \Big(\D_PT\D_Q+\D_{P,Q}^*T\D_Q+\D_PT\D_{Q,P}\Big),
\end{split}
\end{equation*}
where $\Pi_b$ are dyadic paraproducts associated with $b\in\{T1,T^*1\}$, and
\begin{equation}\label{eq:DQP}
  \D_{Q,P}:=1_Q\otimes\Big(\frac{1_Q}{\mu(Q)}-\frac{1_P}{\mu(P)}\Big).
\end{equation}
\end{lemma}

\begin{proof}
Starting from the defining formula \eqref{eq:T0T1},
we expand the difference by the telescoping identity
\begin{equation*}
\begin{split}
  \E_bT\E_b-\E_aT\E_a
  &=\sum_{k=a}^{b-1}(\E_{k+1}T\E_{k+1}-\E_kT\E_k) \\
  &=\sum_{k=a}^{b-1}(\D_kT\D_k+\E_k T\D_k+\D_kT\E_k) \\
  &=\sum_{k=a}^{b-1}\sum_{P,Q\in\mathscr D_k}
  (\D_PT\D_Q+\E_P T\D_Q+\D_PT\E_Q),
\end{split}
\end{equation*}
and hence
\begin{equation}\label{eq:T0exp}
  T_0=\sum_{k\in\Z}\sum_{P,Q\in\mathscr D_k}
  (\D_PT\D_Q+\E_P T\D_Q+\D_PT\E_Q).
\end{equation}

For the localised operators, we further decompose
\begin{equation*}
\begin{split}
  \D_PT\E_Q
  &=\D_P T1_Q\otimes \frac{1_Q}{\mu(Q)} \\
  &=\D_P T1_Q\otimes (\frac{1_Q}{\mu(Q)}-\frac{1_P}{\mu(P)})
  +\D_PT1_Q\otimes \frac{1_P}{\mu(P)} \\
 &=\D_PT\D_{Q,P}+\D_PT1_Q\otimes \frac{1_P}{\mu(P)}.
\end{split}
\end{equation*}
Expanding $\D_P$ in terms of the Haar functions $h_P^\alpha$
\begin{equation*}
  \pair{\D_PT1_Q}{g}
  =\pair{T1_Q}{\D_P g}
  =\sum_{1\leq\alpha<M_P}\sqrt{\mu(Q)}\pair{Th_Q^0}{h_P^\alpha}\pair{h_P^\alpha}{g}
\end{equation*}
and using the bound \cite[(5.3)]{Hyt:Enflo}
\begin{equation}\label{eq:HaarEst}
   \abs{\pair{Th_P^\alpha}{h_Q^\beta}}\lesssim
   \omega\Big(\frac{\ell(P)}{\ell(P)+\rho(z_P,z_Q)}\Big)\frac{\sqrt{\mu(P)\mu(Q)}}{V(z_P,\ell(P)+\rho(z_P,z_Q))},
\end{equation}
we obtain
\begin{equation*}
\begin{split}
  \Norm{\D_PT1_Q}{2}
  &=\sup\{ \abs{\pair{\D_PT1_Q}{g}}:\Norm{g}{2}\leq 1\} \\
  &\lesssim\omega\Big(\frac{\ell(P)}{\ell(P)+\rho(z_P,z_Q)}\Big)\frac{\sqrt{\mu(P)}\mu(Q)}{V(z_P,\ell(P)+\rho(z_P,z_Q))}.
\end{split}
\end{equation*}
(While the results of \cite{Hyt:Enflo} are formulated for metric spaces only, it is easy to check that the distinction between metric and quasi-metric spaces makes no difference in the above considerations.)

Hence
\begin{equation*}
\begin{split}
  \sum_{Q\in\mathscr D_k}  \frac{\Norm{\D_PT1_Q}{2} }{\sqrt{\mu(P)}}
    &=\Big(\sum_{\substack{Q\in\mathscr D_k \\ \frac{\rho(z_P,z_Q)}{\ell(P)} < 1}}
  +\sum_{m=1}^\infty\sum_{\substack{Q\in\mathscr D_k \\ \delta \leq \frac{\rho(z_P,z_Q)}{\delta^{-m}\ell(P)} < 1}}\Big)
   \frac{\Norm{\D_PT1_Q}{2} }{\sqrt{\mu(P)}} \\
  &\lesssim\sum_{m=0}^\infty\sum_{\substack{Q\in\mathscr D_k \\ \frac{\rho(z_P,z_Q)}{\delta^{-m}\ell(P)} < 1}}
  \frac{\omega(\delta^m)\mu(Q)}{V(z_P,\delta^{-m}\ell(P))}
  \lesssim\sum_{m=0}^\infty\omega(\delta^m).
\end{split}
\end{equation*}
Thus the series below converges and we can make the following definition:
\begin{equation*}
  \D_PT1:= \sum_{Q\in\mathscr D_k}  \D_PT1_Q.
\end{equation*}
It follows that
\begin{equation*}
  \sum_{Q,P\in\mathscr D_k}\D_PT1_Q\otimes\frac{1_P}{\mu(P)}
  =\sum_{P\in\mathscr D_k}\D_PT1\otimes\frac{1_P}{\mu(P)}
  =\sum_{P\in\mathscr D_k}(\D_PT1)\E_P.
\end{equation*}
We thus have
\begin{equation*}
  \sum_{P,Q\in\mathscr D_k}\D_PT\E_Q
  =\sum_{P,Q\in\mathscr D_k}\D_PT\D_{Q,P}+\sum_{P\in\mathscr D_k}(\D_PT1)\E_P.
\end{equation*}
Applying a symmetric reasoning to $\E_PT\D_Q$ but observing that considerations on the dual side produce operators adjoint to the ones we already encountered, it follows that
\begin{equation}\label{eq:extractParap}
\begin{split}
  \sum_{P,Q\in\mathscr D_k}
  &\Big(\D_PT\D_Q+\E_PT\D_Q+\D_PT\E_Q\Big) \\
  &=\sum_{P,Q\in\mathscr D_k}\Big(\D_PT\D_Q+\D_{P,Q}^*T\D_Q+\D_PT\D_{Q,P}\Big) \\
  &\qquad+\sum_{P\in\mathscr D_k}\Big(((\D_P T^*1)\E_P)^*+ (\D_PT1)\E_P\Big).
\end{split}
\end{equation}
Summing over all levels $k\in\Z$ of dyadic cubes, the last line in \eqref{eq:extractParap} gives the paraproducts
\begin{equation}\label{eq:identifyParap}
  \sum_{k\in\Z}\sum_{P\in\mathscr D_k}(\D_PT1)\E_P=\Pi_{T1},\qquad
  \sum_{k\in\Z}\sum_{P\in\mathscr D_k}((\D_PT^*1)\E_P)^*=\Pi_{T^*1}^*.
\end{equation}
A combination of \eqref{eq:T0exp}, \eqref{eq:extractParap}, and \eqref{eq:identifyParap} gives the identities claimed in the statement of the lemma.
\end{proof}

The next is to rearrange the pairs of cubes $P,Q\in\mathscr D_k$ in the decomposition of Lemma \ref{lem:basicDec} under common ancestors. It is here that the random selection of the dyadic system enters, and only prove an identity ``on average'', i.e., with an expectation over the random selection on both sides.

\begin{lemma}\label{lem:underCommon}
For each choice of
\begin{equation*}
  T_{P,Q}\in\{\D_PT\D_Q,\quad\D_{P,Q}T\D_Q,\quad\D_PT\D_{Q,P}\},
\end{equation*}
the expectation $\E$ over a random choice of the dyadic system $\mathscr D$ satisfies
\begin{equation}\label{eq:underCommon}
  \E\sum_{k\in\Z}\sum_{P,Q\in\mathscr D_k}T_{P,Q}
 =\E\sum_{m=m_0}^\infty\sum_{k\in\Z}\sum_{\substack{ R\in\mathscr D_k \\ P,Q\in\operatorname{ch}^m(R)}}
 \theta_m(P,Q)T_{P,Q}
\end{equation}
where $m_0$ is the parameter from Lemma \ref{lem:indep}, and $\theta_m(P,Q)\in\{0\}\cup[1,2]$ are some coefficients depending on the ensemble of dyadic systems over which we randomise. Moreover, for $m>m_0$, we have
\begin{equation}\label{eq:thetaNon0}
  \theta_m(P,Q)\neq 0\quad\Rightarrow\quad
  \rho(x_P,x_Q)\gtrsim\delta^{-m}\ell(P).
\end{equation}
\end{lemma}

\begin{proof}
Let us consider any of the terms. $T_{P,Q}$ as in the lemma.
With parameters $\eps_0,m_0$ provided by Lemma \ref{lem:indep}, we rearrange the summation of $T_{P,Q}$ over $P,Q\in\mathscr D_k$ by pigeonholing according to the relative distance of (the reference points of) these cubes:
\begin{equation*}
  \sum_{P,Q\in\mathscr D_k}T_{P,Q}
  =\sum_{m=m_0}^\infty \sum_{P,Q\in\mathscr D_k}\chi_m(P,Q)T_{P,Q},
\end{equation*}
where we define
\begin{equation*}
 \chi_{m_0}(P,Q) :=\begin{cases} 1, & \text{if }d(x_P,x_Q)<\eps_0\delta^{-m_0}\ell(P), \\ 
  0, &\text{otherwise},\end{cases}
\end{equation*}
and, for $m>m_0$,
\begin{equation}\label{eq:defChi}
  \chi_{m}(P,Q) :=\begin{cases} 1, & \text{if }\eps_0\delta^{1-m}\ell(P)\leq d(x_P,x_Q)<\eps_0\delta^{-m}\ell(P), \\ 
  0, &\text{otherwise}.\end{cases}
\end{equation}

For $P,Q\in\mathscr D_k$, notice that both $\D_P$ and $\D_{P,Q}$ depend only on $\{S\in\mathscr D:\ell(S)\leq\ell(P)\}$; indeed, $\D_{P,Q}$ depends only on $\{S\in\mathscr D:\ell(S)=\ell(P)\}$ and $\D_P$ only on $\{S\in\mathscr D:\ell(S)\in\{\delta\ell(P),\ell(P)\}\}$.
Thus, under a random choice of the dyadic system as described in Section \ref{sec:random}, it follows from the independence guaranteed by Lemma \ref{lem:indep} (noting that $T_{P,Q}$, as an operator on $L^2(\mu)$, only depends on the equivalence classes of $P$ and $Q$ modulo sets of $\mu$-measure zero, that appear in part \eqref{it:indep} of Lemma \ref{lem:indep}) that
\begin{equation*}
\begin{split}
  \E(1_{\{P^{(m)}=Q^{(m)}\}}T_{P,Q})
  &=\E(1_{\{P^{(m)}=Q^{(m)}\}})\E T_{P,Q} \\
  &=\P(P^{(m)}=Q^{(m)})\E T_{P,Q}
  =:\pi_m(P,Q)\E T_{P,Q}.
\end{split}
\end{equation*}

If $\chi_m(P,Q)=1$, then $\rho(x_P,x_Q)<\eps_0\delta^{-m}\ell(P)$, and Lemma \ref{lem:indep} guarantees that $\pi_m(P,Q)\in[\frac12,1]$. Hence there is no danger of division by zero in solving
\begin{equation*}
\begin{split}
  \chi_m(P,Q)\E T_{P,Q}
  &=\chi_m(P,Q)\frac{  \E(1_{\{P^{(m)}=Q^{(m)}\}}T_{P,Q}) }{\pi_m(P,Q)} \\
  &=\E\Big(1_{\{P^{(m)}=Q^{(m)}\}}\frac{\chi_m(P,Q)}{\pi_m(P,Q)}T_{P,Q}\Big);
\end{split}
\end{equation*}
note that $\chi_m(P,Q)$, defined via the non-random reference points $x_P,x_Q$, may be freely moved in and out of the expectation.

Now
\begin{equation}\label{eq:defThetam}
  \theta_m(P,Q):=\frac{\chi_m(P,Q)}{\pi_m(P,Q)}\in\{0\}\cup[1,2],
\end{equation}
since $\chi_m(P,Q)\in\{0,1\}$ and $\pi_m(P,Q)\in[\frac12,1]$ when $\chi_m(P,Q)=1$.

Then
\begin{equation*}
\begin{split}
  \E\sum_{P,Q\in\mathscr D_k}\chi_m(P,Q)T_{P,Q}
 &=\E 
 \sum_{\substack{P,Q\in\mathscr D_k  \\
  P^{(m)}=Q^{(m)} }}\theta_m(P,Q)T_{P,Q} \\
 &=\E\sum_{\substack{ R\in\mathscr D_{k-m} \\ P,Q\in\operatorname{ch}^m(R)}}\theta_m(P,Q)T_{P,Q} 
\end{split}
\end{equation*}
where we took the common ancestor $R:=P^{(m)}=Q^{(m)}$ as a new summation variable.

Summing over $m\geq m_0$ and $k\in\Z$ (and taking $k':=k-m$ as the new summation variable, but denoting it again by $k$), we obtain the claimed identity \eqref{eq:underCommon}. The other claim \eqref{eq:thetaNon0} follows immediate from the definition of $\theta_m(P,Q)$ in \eqref{eq:defThetam} and the definition of $\chi_m(P,Q)$ in \eqref{eq:defChi}.
\end{proof}

The final task is to identify the series inside the expectation of the right of \eqref{eq:underCommon} as a sum of dyadic shift.
In this remaining part of the argument, we can once again concentrate on a fixed dyadic system at a time, i.e., the expectation in front plays no further role in the analysis.

We start with the easiest case of $T_{P,Q}=\D_PT\D_Q$:

\begin{lemma}\label{lem:P-Q-shifts}
For $m\geq m_0$ and $\theta_m(P,Q)$ as in Lemma \ref{lem:underCommon}, we have
\begin{equation*}
  \sum_{k\in\Z}\sum_{\substack{R\in\mathscr D_k \\ P,Q\in\operatorname{ch}^m(R)}}\theta_m(P,Q)\D_PT\D_Q
  =\omega(\delta^{m})\S_{(m,m)},
\end{equation*}
where $\S_{(m,m)}$ is a sum of $O(1)$ normalised dyadic shifts of complexity $(m,m)$.
\end{lemma}

\begin{proof}
Expanding $\D_P,\D_Q$ in terms of Haar functions, we have
\begin{equation*}
  \sum_{\substack{P,Q\in \\ \operatorname{ch}^m(R) }} \theta_m(P,Q)\D_P T\D_Q
  =\sum_{\substack{ P,Q\in \\ \operatorname{ch}^m(R)}}\sum_{\substack{1\leq\alpha<M_P \\ 1\leq\beta<M_Q}} \theta_m(P,Q)
   \pair{h_P^\alpha}{Th_Q^\beta} h_P^\alpha\otimes h_Q^\beta.
\end{equation*}
Whenever $\theta_m(P,Q)\neq 0$, we obtain, by \eqref{eq:HaarEst},
\begin{equation}\label{eq:HaarEst2}
\begin{split}
 \abs{ \pair{h_P^\alpha}{Th_Q^\beta} }
 &\lesssim \omega\Big(\frac{\ell(P)}{\ell(P)+\rho(x_P,x_Q)}\Big)\frac{\sqrt{\mu(P)\mu(Q)}}{V(x_P,\ell(P)+\rho(x_P,x_Q))} \\
 &\lesssim \omega(\delta^m)\frac{\sqrt{\mu(P)\mu(Q)}}{\mu(R)}
\end{split}
\end{equation}
using \eqref{eq:thetaNon0} for $m>m_0$ and simply $\ell(P)+\rho(x_P,x_Q)\geq\ell(P)\approx \ell(R)$ for $m=m_0$, absorbing the dependence on the fixed $m_0$ into the implied constants.

Since $M_P,M_Q\leq M$, we see that $\sum_{P,Q\in\operatorname{ch}^m(R)}\theta_m(P,Q)\D_PT\D_{Q}$ is a sum of $O(1)$ series of the form
\begin{equation*}
  \omega(\delta^m)\sum_{P,Q\in\operatorname{ch}^m(R)} a_{PQR}h_P\otimes h_Q,
\end{equation*}
where $h_P=h_P^\alpha$ and $h_Q=h_Q^\beta$ are Haar functions and
\begin{equation*}
  \abs{a_{PQR}}\lesssim\frac{\sqrt{\mu(P)\mu(Q)}}{\mu(R)}.
\end{equation*}
These are precisely the conditions for a component $A_R$ of a normalised dyadic shift of complexity $(m,m)$, and hence
\begin{equation}\label{eq:P-Q-shifts}
  \sum_{k\in\Z}\sum_{R\in\mathscr D_k}\sum_{P,Q\in\operatorname{ch}^m(R)}\theta_m(P,Q) \D_P T\D_Q
  =\omega(\delta^m)\S_{(m,m)},
\end{equation}
as claimed.
\end{proof}

We then turn to the other instances of $T_{P,Q}$.

\begin{lemma}\label{lem:P-QP-shifts}
For $m\geq m_0$ and $\theta_m(P,Q)$ as in Lemma \ref{lem:underCommon}, we have
\begin{equation*}
  \sum_{k\in\Z}\sum_{R\in\mathscr D_k}
  \sum_{P,Q\in\operatorname{ch}^m(R)}\theta_m(P,Q)
  \D_PT\D_{Q,P}
  =\omega(\delta^m)\Big(\sum_{i=0}^{m-1}\S_{(m,i)}+\sum_{j=1}^m\S_{(j,0)}\Big),
\end{equation*}
where each $\S_{a,b}$ is a sum of $O(1)$ many normalised dyadic shifts of complexity $(a,b)$.
\end{lemma}

\begin{proof}
Recalling the definition \eqref{eq:DQP} of $\D_{Q,P}$, we expand
\begin{equation*}
  \frac{1_Q}{\mu(Q)}=\sum_{j=1}^{m} \Big( \frac{ 1_{ Q^{(j-1)} } }{\mu(Q^{(j-1)} ) } -\frac{ 1_{Q^{(j)} } }{\mu(Q^{(j)})}\Big)
  +\frac{1_{Q^{(m)} }}{\mu(Q^{(m)})}
\end{equation*}
and similarly with $\displaystyle \frac{1_P}{\mu(P)}$. Since $Q^{(m)}=R=P^{(m)}$, it follows that
\begin{equation}\label{eq:QPdiff}
\begin{split}
  \frac{1_Q}{\mu(Q)}-\frac{1_P}{\mu(P)}
  &=\sum_{j=1}^{m} \Big( \frac{ 1_{ Q^{(j-1)} } }{\mu(Q^{(j-1)} ) } -\frac{ 1_{Q^{(j)} } }{\mu(Q^{(j)})}\Big) \\
  &\qquad-\sum_{j=1}^{m} \Big( \frac{ 1_{ P^{(j-1)} } }{\mu(P^{(j-1)} ) } -\frac{ 1_{P^{(j)} } }{\mu(P^{(j)})}\Big).
\end{split}
\end{equation}
Now
\begin{equation}\label{eq:QPdiff2}
\begin{split}
  &\sum_{P,Q\in\operatorname{ch}^m(R)}
  \theta_m(P,Q)\D_P T\D_{Q,P} \\
  &=\sum_{P,Q\in\operatorname{ch}^m(R)}
  \sum_{1\leq\alpha<M_P}\theta_m(P,Q)\pair{h_P^\alpha}{T 1_Q} h_P^\alpha\otimes
  \Big( \frac{1_Q}{\mu(Q)}-\frac{1_P}{\mu(P)}\Big).
\end{split}
\end{equation}
Using \eqref{eq:QPdiff} to expand the last factor in \eqref{eq:QPdiff2}, we obtain $2m$ series, each of which will be identified with (a component of) one of the $2m$ dyadic shifts on the right hand side of the claim.

With $\displaystyle \frac{ 1_{ Q^{(j-1)} } }{\mu(Q^{(j-1)} ) } -\frac{ 1_{Q^{(j)} } }{\mu(Q^{(j)})} $ in place of $\displaystyle\frac{1_Q}{\mu(Q)}-\frac{1_P}{\mu(P)}$ in \eqref{eq:QPdiff2}, and taking $S:=Q^{(j-1)}$ as a new summation variable, we get
\begin{equation}\label{eq:QjTerm}
\begin{split}
   &\sum_{P,Q\in\operatorname{ch}^m(R) }
   \theta_m(P,Q) \pair{h_P^\alpha}{T 1_Q} h_P^\alpha\otimes
  \Big( \frac{1_{Q^{(j-1)}}}{\mu(Q^{(j-1)})}-\frac{1_{Q^{(j)}}}{\mu(Q^{(j)})}\Big) \\
  &=   \sum_{\substack{ P\in\operatorname{ch}^m(R) \\ S\in\operatorname{ch}^{m-j+1}(R) }}
  \Big(\sum_{Q\in\operatorname{ch}^{j-1}(S)}\theta_m(P,Q)
    \pair{h_P^\alpha}{T 1_Q} \Big)h_P^\alpha\otimes\Big(\frac{1_S}{\mu(S)}-\frac{1_{S^{(1)}}}{\mu(S^{(1)})}\Big).
\end{split}
\end{equation}
By \eqref{eq:HaarEst2} and recalling that $1_Q=\sqrt{\mu(Q)}h_Q^0$, we have the estimate
\begin{equation}\label{eq:HaarEst3}
\begin{split}
  \sum_{Q\in\operatorname{ch}^{j-1}(S)}
   \abs{ \theta_m(P,Q) \pair{h_P^\alpha}{T 1_Q} }
   &\lesssim \sum_{Q\in\operatorname{ch}^{j-1}(S)}
   \omega(\delta^m)\frac{\sqrt{\mu(P)}\mu(Q)}{\mu(R)} \\
   &\leq\omega(\delta^m)\frac{\sqrt{\mu(P)}\mu(S)}{\mu(R)}.
\end{split}
\end{equation}
On the other hand, the function $\displaystyle \frac{1_S}{\mu(S)}-\frac{1_{S^{(1)}}}{\mu(S^{(1)})}$ is recognised as a linear combination, with coefficients of size $O(\mu(S)^{-1/2})$, of the Haar functions $h_{S^{(1)}}^\beta$. Thus
\begin{equation*}
  \eqref{eq:QjTerm}
  = \omega(\delta^m)\sum_{\substack{ P\in\operatorname{ch}^m(R) \\ S\in\operatorname{ch}^{m-j+1}(R) }}
  \sum_{1\leq\beta<M_{S^{(1)}}}O\Big( \frac{\sqrt{\mu(P)\mu(S)}}{\mu(R)}\Big)h_P^\alpha\otimes h_{S^{(1)}}^\beta
\end{equation*}
Taking $S^{(1)}\in\operatorname{ch}^{m-j}(R)$ as a new summation variable in place of $S$, we recognise this as a term of a shift of complexity $(m,m-j)$. Hence
\begin{equation*}
  \sum_{k\in\Z}\sum_{R\in\mathscr D_k}\eqref{eq:QjTerm}
  =\omega(\delta^m)\S_{(m,m-j)},
\end{equation*}
where $\S_{(m,m-j)}$ is a sum of boundedly many normalised dyadic shifts of complexity $(m,m-j)$. Summing over $j\in\{1,\ldots,m\}$ and denoting $i:=m-j\in\{0,\ldots,m-1\}$, we obtain the first half of the claimed decomposition.

It remains to consider the expressions obtained by substituting $\displaystyle \frac{ 1_{ P^{(j-1)} } }{\mu(P^{(j-1)} ) } -\frac{ 1_{P^{(j)} } }{\mu(P^{(j)})} $ in place of $\displaystyle\frac{1_Q}{\mu(Q)}-\frac{1_P}{\mu(P)}$ in \eqref{eq:QPdiff2}. Taking $S:=P^{(j-1)}\in\operatorname{ch}^{m-j+1}(R)$ as a new summation variable, we get
\begin{equation}\label{eq:PjTerm}
\begin{split}
   &\sum_{P,Q\in\operatorname{ch}^m(R)}
   \theta_m(P,Q)\pair{h_P^\alpha}{T 1_Q} h_P^\alpha\otimes
  \Big( \frac{1_{P^{(j-1)}}}{\mu(P^{(j-1)})}-\frac{1_{P^{(j)}}}{\mu(P^{(j)})}\Big) \\
  &=   \sum_{\substack{ S\in\operatorname{ch}^{m-j+1}(R) \\ P\in\operatorname{ch}^{j-1}(S)}}
  \Big(\sum_{Q\in\operatorname{ch}^{m}(R)}\theta_m(P,Q)
    \pair{h_P^\alpha}{T 1_Q} \Big)h_P^\alpha\otimes\Big(\frac{1_S}{\mu(S)}-\frac{1_{S^{(1)}}}{\mu(S^{(1)})}\Big).
\end{split}
\end{equation}
By \eqref{eq:HaarEst2}, similarly to \eqref{eq:HaarEst3}, we obtain the estimate
\begin{equation*}
  \sum_{Q\in\operatorname{ch}^{m}(R)}
    \Babs{\theta_m(P,Q)\pair{h_P^\alpha}{T 1_Q} }
    \lesssim \sum_{Q\in\operatorname{ch}^{m}(R)}\omega(\delta^m)\frac{\sqrt{\mu(P)}\mu(Q)}{\mu(R)}
    =\omega(\delta^m)\sqrt{\mu(P)}.
\end{equation*}
As before, $\displaystyle \frac{1_S}{\mu(S)}-\frac{1_{S^{(1)}}}{\mu(S^{(1)})}$ is a linear combination, with coefficients of the order $\displaystyle O(\mu(S)^{-1/2})=O(\sqrt{\mu(S)}/\mu(S))$, of the Haar functions $h_{S^{(1)}}^\beta$.
Thus
\begin{equation*}
\begin{split}
  \eqref{eq:PjTerm}
  &=\omega(\delta^m)\sum_{S\in\operatorname{ch}^{m-j+1}(R)}
  \sum_{\substack{P\in\operatorname{ch}^{j-1}(S) \\ 1\leq\beta<M_{S^{(1)}}}}O( \frac{ \sqrt{\mu(P)\mu(S)}}{\mu(S)})
    h_P^\alpha\otimes h_{S^{(1)}}^\beta \\
  &=\omega(\delta^m)\sum_{U\in\operatorname{ch}^{m-j}(R)}
  \sum_{\substack{P\in\operatorname{ch}^{j}(U) \\ 1\leq\beta<M_{U}}}O( \frac{ \sqrt{\mu(P)\mu(U)}}{\mu(U)})
    h_P^\alpha\otimes h_{U}^\beta,
\end{split}
\end{equation*}
taking $U:=S^{(1)}\in\operatorname{ch}^{m-j}(R)$ as a new summation variable. We can recognise the inner summation as a component $A_{U}$ of a normalised dyadic shift of complexity $(j,0)$. Then
\begin{equation*}
  \sum_{k\in\Z}\sum_{R\in\mathscr D_k}\eqref{eq:PjTerm}
  =\omega(\delta^m)\sum_{k\in\Z}\sum_{R\in\mathscr D_k}\sum_{U\in\operatorname{ch}^{m-j}(R)}A_U
  =\omega(\delta^m)\S_{(j,0)},
\end{equation*}
where $\S_{(j,0)}$ is a normalised dyadic shift of complexity $(j,0)$.

Summing over $j\in\{1,\ldots,m\}$, we obtain the second half of the claimed decomposition.
\end{proof}

By symmetry with the case considered in Lemma \ref{lem:P-QP-shifts}, we also obtain:

\begin{lemma}\label{lem:PQ-Q-shifts}
For $m\geq m_0$ and $\theta_m(P,Q)$ as in Lemma \ref{lem:underCommon}, we have
\begin{equation}\label{eq:PQ-Q-shifts}
  \sum_{k\in\Z}\sum_{\substack{ R\in\mathscr D_k \\ P,Q\in\operatorname{ch}^m(R) }}\theta_m(P,Q)\D_{P,Q}^*T\D_Q
  =\omega(\delta^m)\Big(\sum_{i=0}^{m-1}\S_{i,m}+\sum_{j=1}^m\S_{0,j}\Big),
\end{equation}
where each $\S_{a,b}$ is a (sum of $O(1)$ many) shift(s) of complexity $(a,b)$.
\end{lemma}

We can now collect the pieces together to give:

\begin{proof}[Proof of the Dyadic Representation Theorem \ref{thm:DRT}]
By \eqref{eq:T0T1} and Lemma \ref{lem:basicDec}, we have
\begin{equation*}
  T=T_1+\Pi_{T1}+\Pi_{T^*1}^*+\sum_{\substack{k\in\Z \\ P,Q\in\mathscr D_k}}
  \Big(\D_PT\D_Q+\D_{P,Q}^*T\D_Q+\D_PT\D_{Q,P}\Big).
\end{equation*}
This identity is valid for every fixed dyadic system $\mathscr D$, and we can take the expectation of both sides under a random choice of $\mathscr D$. Note that $T$ and $T_1$ are independent of $\mathscr D$. This gives
\begin{equation*}
\begin{split}
  T &-T_1 -\E(\Pi_{T1}+\Pi_{T^*1}^*)=\E\sum_{\substack{k\in\Z \\ P,Q\in\mathscr D_k}}
  \Big(\D_PT\D_Q+\D_PT\D_{Q,P}+\D_{P,Q}^*T\D_Q\Big) \\
  &=\E\sum_{m=m_0}^\infty\sum_{\substack{k\in\Z,R\in\mathscr D_k \\ P,Q\in\operatorname{ch}^m(R)}}
  \theta_m(P,Q)\Big(\D_PT\D_Q+\D_PT\D_{Q,P}+\D_{P,Q}^*T\D_Q\Big) \\
  &\qquad\qquad\text{by Lemma \ref{lem:underCommon}} \\
  &=\E\sum_{m=m_0}^\infty\omega(\delta^m)\Big(\S_{m,m}
  +[\sum_{i=0}^{m-1}\S_{m,i}+\sum_{j=1}^m\S_{j,0}]
  +[\sum_{i=0}^{m-1}\S_{i,m}+\sum_{j=1}^m\S_{0,j}]\Big)  \\
  &\qquad\qquad\text{by Lemmas \ref{lem:P-Q-shifts}, \ref{lem:P-QP-shifts}, and \ref{lem:PQ-Q-shifts}}.
\end{split}
\end{equation*}
Moving the terms $T_1$ and $\E(\Pi_{T1}+\Pi_{T^*1}^*)$ to the right-hand side, this is precisely the identity claimed in Theorem \ref{thm:DRT}. (In the statement of Theorem \ref{thm:DRT}, we started the summation from $m=0$ instead of $m=m_0$, but this is easily fixed by simply defining the ``extra'' shifts to be identically zero.)
\end{proof}

\subsection*{Acknowledgements}
This work was originally part of a longer manuscript ``Schatten properties of commutators and equivalent Sobolev norms on metric spaces'', coauthored with Riikka Korte and containing both the present text and the companion paper \cite{HK:W1p}, whose results play a critical role in completing the end-point case \eqref{it:p=d} of Main Theorem \ref{thm:main} of the present work. While we decided to split the original manuscript into two, the author would like to thank Riikka Korte for the fruitful collaboration that allowed the completion of the larger project. The author would also like to thank Lorenzo Zacchini and an anonymous referee for their careful reading of the manuscript, which helped in eliminating several typos and some oversights of earlier versions.

%\bibliography{metricW1p}
%\bibliographystyle{abbrv}

\end{document}